\documentclass[10pt]{article}

\usepackage[margin=1.2in]{geometry}
\usepackage[round]{natbib}
\usepackage{hyperref}

\usepackage{amsmath,amssymb,amsthm,bm,paralist,enumitem,booktabs,graphicx,tikz,xcolor}
\usepackage{mathrsfs}   % \mathscr
\usepackage{dsfont}     % \mathds
\usepackage{relsize}    % \mathlarger

\newtheorem{theorem}{Theorem}
\newtheorem{lemma}[theorem]{Lemma}
\newtheorem{corollary}[theorem]{Corollary}

\newtheorem{definition}[theorem]{Definition}
\newtheorem{remark}[theorem]{Remark}
\newtheorem{example}[theorem]{Example}

\DeclareMathOperator*{\bigtimes}{\mathlarger{\times}}

\begin{document}

	\title{Hilbert space embeddings of independence tests and interaction measures of several variables}
	
	% Standard article author block (removes JMLR \name, \email, \addr)
	\author{
		Jean C. Guella \\
		\small Universidade Estadual de Mato Grosso do Sul \\
		\small Nova Andradina, Brazil \\
		\small \texttt{jean.guella@uems.br}
	}
	
	\date{} % Leaves date blank, or remove this line to include today's date
	
	\maketitle

\begin{abstract}		We present a unified theoretical framework for kernel-based measures of dependence on product spaces. Building on the ideas underlying distance covariance, distance multivariance, and the Hilbert-Schmidt Independence Criterion (HSIC), we define a new family of kernels on an $n$-fold Cartesian product, termed positive definite independent of order $k$ (PDI$_{k}$ kernels). These kernels extend the concepts of positive definite  and conditionally negative definite kernels to higher orders and provide the foundation for generalized independence and interaction tests, such as the generalized Lancaster interaction of order $k$ ($\Lambda_{k}^{n}$), and the Streitberg interaction ($\Sigma$). Our analysis focuses on the continuous setting, where we prove a Kernel Mean Embedding Theorem for PDI$_{k}$ kernels and establish the corresponding integrability restrictions. Based on these results, we characterize how the Kronecker products of PDI kernels behave.
\end{abstract}

\tableofcontents

		\section{Introduction}

Kernel-based measures of dependence have become fundamental tools in statistics and machine learning.  They include the \emph{distance covariance}  \citep{ Szekely2007, Szekely2009, Feuerverger1993, Bakirov2006,  Szekely2014, Dueck2014, Han2024, Janson2021, MartinezGomez2014, Szekely2013, Yao2018}) and the \emph{Hilbert-Schmidt Independence Criterion} (HSIC) \cite{Gretton2005, Gretton2008, Albert2022, Gretton2005, Gretton2008, Pfister2018, Sejdinovic2013a, Poczos2012, Tjoestheim2022, Zhu2020}, which characterize independence between random variables by the nonnegativity of quadratic forms involving positive definite (PD) or conditionally negative definite (CND) kernels.

In a recent paper \cite{Guella2025}, we obtained a Schoenberg type of characterization (extending the seminal paper \cite{Schoenberg1938})  for all radial kernels of several variables that can be used as an independence test in any dimension. Precisely, it provides a description of the continuous functions $g:[0, \infty)^{n} \to \mathbb{R} $ such that for any  $d \in \mathbb{N}$ and any  discrete probability measure $P $ in $(\mathbb{R}^{d})_{n}:=\prod_{i=1}^{n}\mathbb{R}^{d}$, we have that $P \neq  \bigtimes_{i=1}^{n}P_{i}$ if and only if 
\begin{equation}\label{objective}
	\int_{(\mathbb{R}^{d})_{n}} \int_{(\mathbb{R}^{d})_{n}} g(\|x_{1} - y_{1} \|^{2}, \ldots, \|x_{n} - y_{n} \|^{2} )d[P - \bigtimes_{i=1}^{n}P_{i}](x )d[P - \bigtimes_{i=1}^{n}P_{i}]( y )> 0. 
\end{equation}

Although Equation \ref{objective} reduces to a double sum, we adopt integration terminology to simplify the broader theoretical expressions.

The interest in these functions lies in obtaining an all purpose independence test on Euclidean spaces (that is, with no restrictions in the dimension).

A key difficulty in obtaining such a characterization is that  the set of signed measures 
\[
\{P - \bigtimes_{i=1}^{n}P_{i}, \quad P \text{ is a  discrete probability measure}\}
\] 
is very difficult to analyze, as it is not a vector space (not even a convex space). However, if we restrict the functions that satisfy  Equation \eqref{objective} by additionally demanding that they can distinguish whether two discrete probabilities $P,Q$ in $(\mathbb{R}^{d})_{n}$ are equal, provided that $P_{i}=Q_{i}$ for any $1\leq i \leq n$, we are essentially analyzing the problem on the vector space $\mathcal{M}_{2}((\mathbb{R}^{d})_{n})$ proposed by the author (see Remark \ref{hanhjordanequivalence}) and from this starting point a characterization is feasible.

Surprisingly, it turns out that the class of functions that satisfies  Equation \eqref{objective} and those with this additional requirement are the same, see Theorem \ref{bernsksevndimpart3} for the case $k=2$. As a matter of fact, even the standard definitions of distance covariance and HSIC satisfy  this property on $\mathcal{M}_{2}$.

Based on this equivalence, we   analyze in this paper what properties a kernel   $\mathfrak{I}: \mathds{X}_{n} \times \mathds{X}_{n} \to \mathbb{R}$ must satisfy so that for every nonzero discrete measure $\mu \in \mathcal{M}_{2}( \mathds{X}_{n})$  it holds
\begin{equation*}
	\int_{\mathds{X}_{n} }	\int_{\mathds{X}_{n} }\mathfrak{I}(u,v)d\mu(u)d\mu(v) > 0.
\end{equation*}	

Similar to the radial case, in various other contexts, this property is equivalent to being an independence test, see Corollary \ref{productnkernels}  on the case $k=2$.

However, in many real world scenarios the multivariate data might not be independent, but the probability might interact with its marginals in different ways that are relevant for the problem, see for instance \cite{Liu2023} and \cite{Liu2024} and references therein.

In this sense, two types of generalized independence (usually called interactions) have gained attention in the literature of kernel methods recently: the Streitberg \cite{Streitberg1990} and the Lancaster interaction \cite{Lancaster1969} (they are defined in Section \ref{Terminology}, and they are part of the broader context of partition lattices, see \cite{Liu2023}). In \cite{Guella2025}, it is also described the  set of continuous functions $g:[0, \infty)^{n} \to \mathbb{R} $ such that for any $d \in \mathbb{N}$ and any discrete probability measure $P $ in $(\mathbb{R}^{d})_{n}$ satisfies that its Streitberg interaction $\Sigma[P]\neq 0$ (or its Lancaster interaction $\Lambda[P]\neq 0$) if, and only if 
\begin{equation}\label{objective2}
	\int_{(\mathbb{R}^{d})_{n}} \int_{(\mathbb{R}^{d})_{n}} (-1)^{n}g(\|x_{1} - y_{1} \|^{2}, \ldots, \|x_{n} - y_{n} \|^{2} )d\Sigma[P](x )d\Sigma[P](y) >0.
\end{equation} 

Similar to the independence tests above, such a task is difficult, but if we additionally impose that the functions that satisfy  Equation \eqref{objective2} can also  differentiate whether $P=Q$, provided that $P_{F}=Q_{F}$ for any $F \subset \{1, \ldots, n\}$, $|F|\leq n-1$, we are essentially analyzing the problem on the vector space $\mathcal{M}_{n}((\mathbb{R}^{d})_{n})$ (see Remark \ref{hanhjordanequivalence}, and by Lemma \ref{exm02xn} both $\Sigma[P]$ and $\Lambda[P]$ are elements of $\mathcal{M}_{n}((\mathbb{R}^{d})_{n})$ for any probability $P$). As before, it turns out that the class of functions that satisfies Equation \ref{objective2} and the ones with this additional requirement are the same, see Theorem \ref{bernsksevndimpart3} for the case $k=n$. We emphasize that for this class of radial kernels on all Euclidean spaces, we obtained that the tests capable of discerning whether $\Sigma[P]=0$ are the same as those for $\Lambda[P]=0$, even though these two equalities have different probabilistic conclusions.

Based on this equivalence, we   analyze in this paper what properties a kernel  $\mathfrak{I}: \mathds{X}_{n} \times \mathds{X}_{n} \to \mathbb{R}$ must satisfy so  that for every  nonzero discrete   measure $\mu \in \mathcal{M}_{n}( \mathds{X}_{n})$, it holds    
\begin{equation*} 
	\int_{\mathds{X}_{n} } \int_{\mathds{X}_{n} }(-1)^{n}	\mathfrak{I}(u,v)d\mu(u)d\mu(v) > 0.
\end{equation*}	
The term $(-1)^{n}$ is convenient for the development of the subject.  This approach is based on the theory of Distance Multivariance, see \cite{Boettcher2018, Boettcher2019, Chakraborty2019}, where it is studied the Kronecker product of $n$ CND radial kernels. As the radial case, in some scenarios this property is equivalent at being  a test for whether or not  $\Sigma[P]=0$, but also for  $\Lambda[P]=0$, see Corollary \ref{productnkernels}  on the case $k=n$.

In \cite{Guella2025}, a gap is noted between Lancaster/Streitberg interactions and the standard independence test.  It is filled with an index $k$, where  $2\leq k\leq  n$,  by defining a  vector spaces $\mathcal{M}_{k}$ (see Subsection \ref{Terminology}), which lie between $\mathcal{M}_{2}$ and $\mathcal{M}_{n}$, and  by defining a  generalization of the Lancaster interaction with an index $k$, where when $k=2$ we have the standard independence test and when $k=n$ we have the standard Lancaster interaction. A problem similar to Equations \ref{objective} and \ref{objective2} can be analyzed for these intermediate cases; surprisingly, testing whether the generalized Lancaster interaction is zero is equivalent to operating within the vector space  $\mathcal{M}_{k}$.

Based on this equivalence, we   analyze in this paper what properties, for an $2\leq k \leq n $,   a kernel   $\mathfrak{I}: \mathds{X}_{n} \times \mathds{X}_{n} \to \mathbb{R}$ must satisfy so that for every nonzero $\mu \in \mathcal{M}_{k}( \mathds{X}_{n})$  it holds
\begin{equation*}
	\int_{\mathds{X}_{n} }	\int_{\mathds{X}_{n} }(-1)^{k}\mathfrak{I}(u,v)d\mu(u)d\mu(v) > 0.
\end{equation*}	

As the radial case,  in other different scenarios  this property is equivalent to being a test for an generalized independence test for $\Lambda_{k}^{n}$, see Corollary \ref{productnkernels}.

A concurrent path that we take along the text is to understand these new types of kernels on the continuous setting, with an emphasis on how  we can produce a kernel mean embedding for them. This path leads to some intriguing results even for PD kernels, such as Example  \ref{exkerneldiag}, where it is proved that the set of measures that have good integrability restrictions (FIP integrability) with respect to a PD kernel is not necessarily a vector space.

The paper is divided into $4$ sections and  $2$ appendices.

In Section \ref{Definitions}, we review in detail the most important results of the literature for the development of this paper. It   includes a few definitions, notations, the properties of the spaces $\mathcal{M}_{k}$ (discrete measures) and $\mathfrak{M}_{k}$ (continuous measures), along with the main results of Bernstein functions of order $k$ with $n$ variables, whose properties are the main guide for some results we aim in this paper. We emphasize that the results of  \cite{Guella2025} are not a prerequisite for this paper, but they serve as a guide to several results. Also, \cite{Guella2025} it is the first paper to define and give a statistical meaning to the spaces  $\mathcal{M}_{k}$ and for the Lancaster interaction of order $k$, denoted as  $\Lambda_{k}^{n}$. 

After Section \ref{Definitions}, all definitions and  results presented are new in the literature.

In Section \ref{Positivedefiniteindependentkernelsofordern}, our objective is to analyze the behavior of the kernels $\mathfrak{I}: \mathds{X}_{n} \times \mathds{X}_{n} \to \mathbb{R}$ such that for every discrete measure $\mu \in \mathcal{M}_{n}( \mathds{X}_{n})$ it satisfies
\begin{equation}\label{gencasen}
	\int_{\mathds{X}_{n} }	\int_{\mathds{X}_{n} }(-1)^{n}\mathfrak{I}(u,v)d\mu(u)d\mu(v) \geq 0.
\end{equation}	

Initially, we obtain the main properties of those kernels  following a similar path as Section $3$ in \cite{Berg1984}, and we also obtain a geometrical interpretation of those kernels in Theorem \ref{PDIngeometryrkhs}.  In Subsection \ref{Integrabilityrestrictionsn} we move to the continuous case. First, we focus on the analysis of which probabilities we can compare using this method and a version of the Kernel Mean Embedding for them in Theorem \ref{integPDIn}. Here, we define a setting that is suitable for integration of partition Lattices, which we call FIP integrals, defined in this subsection but developed in Appendix \ref{FIPintegrals}, as it has an importance in its own definition.

In Section \ref{IndependencetestsembeddedinHilbertspaces}, our objective is to analyze the behavior of the kernels $\mathfrak{I}: \mathds{X}_{n} \times \mathds{X}_{n} \to \mathbb{R}$ such that for every discrete measure $\mu \in \mathcal{M}_{k}( \mathds{X}_{n})$, with a focus on cases where $2\leq k \leq n-1$, it satisfies
\begin{equation}\label{gencasek}
	\int_{\mathds{X}_{n} }	\int_{\mathds{X}_{n} }(-1)^{k}\mathfrak{I}(u,v)d\mu(u)d\mu(v) \geq 0.
\end{equation}	

Again,  we obtain the main properties of those kernels  following a similar path as Section $3$ in \cite{Berg1984}. In general, this type of kernel is more difficult to analyze compared to those in Section \ref{Positivedefiniteindependentkernelsofordern}, and we  sometimes  make use of the additional symmetry assumption of complete $n$-symmetry to avoid a combinatorial burden. A particularly interesting  example is Theorem \ref{ineqpdik}, which is inspired by the radial result of  Corollary \ref{growthpdiknvar}, and with it several results in Section \ref{IndependencetestsembeddedinHilbertspaces} are obtained as a direct application of the arguments from Section \ref{Positivedefiniteindependentkernelsofordern}. 	In Subsection \ref{Integrabilityrestrictionsnk} we move to the continuous setting and the focus is to analyze which probabilities we can compare using this method and a version of the Kernel Mean Embedding for them in Theorem \ref{integPDIkn}, by making use of FIP integrals.

In Section \ref{KroneckerproductsofPDIkernels}, we move  to a generalization of the core idea of Distance covariance/HSIC: if $\gamma_{1}: X\times X \to \mathbb{R}$ and $\gamma_{2}: Y\times Y \to \mathbb{R}$ are both Integrally Strictly Positive Definite (ISPD)  or  CND-Characteristic, then 
$$
\int_{X\times Y} \int_{X\times Y} \gamma_{1}(x_{1}, x_{2})\gamma_{2}(y_{1}, y_{2})d[P -P_{1}\times P_{2}](x_{1}, y_{1} )d[P -P_{1}\times P_{2}](x_{2}, y_{2} )> 0, 
$$
whenever $P \neq P_{1}\times P_{2}$. This property was generalized in Corollary $3.11$ in \cite{guella2023} by replacing $P -P_{1}\times P_{2}$ by any measure $\mu \in \mathfrak{M}_{2}(X\times Y)$, provided that it satisfies some integrability conditions. In the present paper we complete characterize the properties of the  Kronecker products of PDI  kernels in Theorem  \ref{Kroengeral} (discrete case) and in Theorem \ref{Kroenconti} (continuous case). We note that in general, there are not many examples of Kronecker product of kernels that can be used to use as an standard independence test when the dimension of the space is at least $3$, see Corollary \ref{classkroenne}.

We emphasize Corollary \ref{productnkernels}, which describes the main properties of the Kronecker product of $n$ kernels on an $n$-fold Cartesian product.   We conjecture that the behavior of  Corollary \ref{productnkernels} regarding the equivalence between the approach of using the vector spaces $\mathcal{M}_{k}$ and of generalized Lancaster/Streitberg interactions should be similar for the most common used families of kernels.

On the Appendix we present two topics.

In Section \ref{geoint}  our focus is to obtain a  geometrical interpretation for PDI$_{k}$ kernels defined in $\mathds{X}_{n}$ by using the RKHS of the related positive definite kernel $	K^{\mathfrak{I}}$. In it, we obtain $\mathfrak{I}$ can be written in several ways as a sum of squares of norms of vectors in the RKHS of $K^{\mathfrak{I}}$.  However, we have found explicit methods  so that this sum of squares only involves nonnegative coefficients (thus showing that $\mathfrak{I}$ must be nonnegative function) only for when  the codimension $n-k$ is is either $0,1$ or $2$. We conjecture that this property holds in general, but the combinatorial complexity increases rapidly as the codimension $n-k$ increases.

In Appendix \ref{FIPintegrals}, we develop the main results of FIP integrals. It also includes the Example \ref{FIPbadintegration} about the non linearity aspects of FIP integrals, and   Theorem  \ref{FIPproduct}, which is an interesting  Holder inequality for partition Lattices.  

The Table \ref{tab:notation} consists of the main symbols and definitions used in the text and where to find them.

\section{Preliminaries and Notation}\label{Definitions}

This section introduces notation and background material used throughout the paper. We review basic concepts from measure theory, kernel embedding, and interaction measures, and we establish the framework within which higher-order dependence structures will be studied. For the reader's convenience, Table \ref{tab:notation} summarizes the key symbols and definitions used, along with where they are first formally presented.

\begin{table}[!htbp]
	\centering
	\caption{Summary of Key Notation and Definitions used in the paper}
	\label{tab:notation}
	% Using p{} allows long descriptions to wrap to the next line nicely
	\begin{tabular}{lp{10cm}}
		\textbf{Symbol / Definition} & \textbf{Description \& Location} \\
		$\mathds{X}_{n}$ & Cartesian product $\prod_{i=1}^{n}X_{i}$ (Subsection \ref{Terminology}). \\
		$\mathcal{M}(X)$ / $\mathfrak{M}(X)$ & Discrete/finite signed Radon measures on $X$ (Equation \ref{medidaarbirad}). \\
		$\mathcal{M}_{k}(\mathds{X}_n)$ / $\mathfrak{M}_{k}(\mathds{X}_n)$ & Measures encoding interactions up to order $k$ in discrete/continuous settings (Equation \ref{medidamargk}). \\
		$\Lambda[P]$ & Standard Lancaster interaction of a probability $P$, (above Theorem \ref{generallancaster}). \\
		$\Lambda_k^n[P]$ & Generalized Lancaster interaction of order $k$ (above Theorem \ref{generallancaster}, see\citep{Guella2025}). \\
		$\Sigma[P]$ & Streitberg interaction of a probability $P$ (Theorem \ref{streit}). \\
		$\Delta_{k-1}^{n}$ & Extended diagonal of order $k-1$ in $\mathds{X}_{n}\times \mathds{X}_{n}$ (Equation \ref{exdia1def}). \\
		$\mu^{n}_{k}[x_{\vec{1}}, x_{\vec{2}}]$ & Discrete measure acting as a building block for order $k$ (Equation \ref{measorderk}). \\
		$P_{\pi}$ & Probability factoring according to a partition $\pi$ (Equation \ref{partitionprobability}). \\
		$\mathcal{H}_K$/$\mathcal{H}_{\mathfrak{I}}$ & Reproducing Kernel Hilbert Space (RKHS) of the PD kernel $K$/ PDI$_{k}$ kernel $\mathfrak{I}$. \\
		$\mathbb{N}^{n}_{m}$ / $\mathbb{N}^{0,n}_{m}$ & Sets $\{1, \ldots, m\}^n$/$\{0, 1, \ldots, m\}^n$ (Subsection \ref{Terminology}). \\
		$\vec{1}$, $\vec{2}$ & Vectors with all entries equal to 1 or 2 (Subsection \ref{Terminology}). \\
		$x_{\alpha_F + \beta_{F^c}}$ & Mixed multi-index element in $\mathds{X}_n$ (Subsection \ref{Terminology}). \\
		FIP & \emph{Fully Integrable by Partitions} integral condition (Definition \ref{FIP}, subject developed in Appendix \ref{FIPintegrals}). \\
		Marginally Delimited & Boundedness condition on the marginals of a signed  measure relative to  a probability (Definition \ref{margdeli}). \\
		$\text{PDI}_n$/$\text{PDI}_k$ & Positive Definite Independent kernel of order $n$/$k$ (Definition \ref{PDIn} and  \ref{PDI2}). \\
		$\mathcal{P}[\mathfrak{I}]$/$\mathfrak{M}[\mathfrak{I}]$ & Probabilities/Measures that are  FIP for the PDI kernel $\mathfrak{I}$ (Lemma \ref{integPDIn0} and \ref{integPDInk}). \\
		$\mathfrak{M}[\mathfrak{I}, P]$/$\mathfrak{M}_{k}[\mathfrak{I}, P]$ & Measures in $\mathfrak{M}[\mathfrak{I}]$/$\mathfrak{M}_{k}[\mathfrak{I}]$ marginally delimited by $P$ (Lemma \ref{integPDIn0} and \ref{integPDInk}). \\
		$\text{PDI}_n$-Characteristic & A $\text{PDI}_n$ kernel that induces  strict inner products (Definition \ref{PDIn-Characteristic}). \\
		$\text{PDI}_k$-Characteristic & A $\text{PDI}_k$ kernel that induces  strict inner products (Definition \ref{PDIkn-Characteristic}). \\
		$n$-symmetric kernel &  Symmetry for a kernel defined on $\mathds{X}_{n}$ (Equation  \ref{nsymmetrydefn}).\\
		complete $n$-symmetric kernel &  Additional symmetry for a kernel defined on $\mathds{X}_{n}$ (Definition \ref{completensymmetry}).\\
	\end{tabular}
\end{table}

First, for a Hausdorff space $X$ we use the notation

\begin{equation}\label{medidaarbirad}
	\mathfrak{M}(X):=\{\text{The vector space of all real Radon regular measures of finite variation  in } X\}. 
\end{equation}

By having finite variation, measures on  $\mathfrak{M}(X)$ are always  inner and outer regular. For more information on measures on Hausdorff spaces see Chapter $2$ in \cite{Berg1984}. Recall that every Borel measure of finite variation (in particular, every probability measure) on a separable, complete metric space (also known as Polish space) is necessarily Radon. 

Additional important definitions and results about measures are given in Section \ref{Terminology}.

\subsection{Positive definite and conditionally negative definite kernels}

A symmetric kernel $K: X \times X \to \mathbb{R}$ is called Positive Definite (PD) if, for every finite quantity of distinct points $x_{1}, \ldots, x_{n} \in X$ and scalars $c_{1}, \ldots, c_{n} \in \mathbb{R}$, we have 
$$
\sum_{i, j =1}^{n}c_{i}c_{j} K(x_{i}, x_{j}) \geq 0.
$$
The kernel $K$ is Strictly Positive Definite (SPD) if equality holds in the previous inequality only when all scalars $c_{i}$ are zero.

The Reproducing Kernel Hilbert Space (RKHS) of a PD kernel $K: X \times X \to \mathbb{R}$ is the Hilbert space $\mathcal{H}_{K} \subset \mathcal{F}(X, \mathbb{R})$, which satisfies \cite{Steinwart2008}
\begin{enumerate}
	\item[$(i)$] The function $x \in X \mapsto K_{y}(x):= K(x,y) \in \mathcal{H}_{K}$ for any $y \in X$;
	\item[$(ii)$] $\langle K_{x}, K_{y}\rangle = K(x,y) $ for any $x,y \in X$;
	\item[$(iii)$] $\langle K_{x}, f\rangle =f(x) $ for any $f \in \mathcal{H}_{K}$ and $x \in X$;
	\item[$(iv)$] $\overline{ \operatorname{span}\{ K_{y}, \quad y \in X\}}= \mathcal{H}_{K}$.
\end{enumerate} In particular, if $X$ is a Hausdorff space and $K$ is continuous, then $\mathcal{H}_{K} \subset C(X)$.

The well known Kernel Mean Embedding Theorem, shows how a continuous PD kernel induces a semi-inner product on a subspace of $\mathfrak{M}(X)$. 

\begin{theorem}\label{initialextmmddominio}
	If $K: X \times X \to \mathbb{R}$ is a continuous positive definite kernel and $\mu \in \mathfrak{M}(X)$ satisfies $\sqrt{K(x,x)} \in L^{1}(\lvert \mu \rvert)$ ($\mu \in \mathfrak{M}_{\sqrt{K}}(X)$), then 
	$$
	z \in X \mapsto K_{\mu}(z):=\int_{X} K(x,z)d\mu(x) \in \mathbb{R}
	$$	
	is an element of $\mathcal{H}_{K}$. Moreover, if $\eta$ is another measure with the same condition, then
	$$
	\langle K_{\eta}, K_{\mu}\rangle_{\mathcal{H}_{K}}= \int_{X} \int_{X}K(x,y)d\eta(x)d\mu(y).
	$$
	In particular, the map $(\eta,\mu)\mapsto \langle K_{\eta}, K_{\mu}\rangle_{\mathcal{H}_{K}}$ defines a semi-inner product on $\mathfrak{M}_{\sqrt{K}}(X)$
\end{theorem}

Note that if $K$ is bounded, then $\mathfrak{M}_{\sqrt{K}}(X)=\mathfrak{M}(X)$. The kernel is Integrally Strictly Positive Definite (ISPD) if it is bounded and the semi-inner product in Theorem \ref{initialextmmddominio} is an inner product. If $K$ is bounded and the semi-inner product is an inner product on the subspace $\{\mu, \quad \mu(X)=0\}$, we say that $K$ is Characteristic. On a bounded Characteristic kernel
$$
D_{K}(P,Q):=\sqrt{\int_{X} \int_{X}K(x,y)d[P-Q](x)d[P-Q](y) }= \|K_{P} - K_{Q}\|_{\mathcal{H}_{K}}
$$
is a metric on the space of probabilities. The pseudometric $D_{K}$ is usually called the Maximum Mean Discrepancy (MMD). By definition, every ISPD kernel is Characteristic, but the converse does not hold. For these results regarding the kernel mean embedding and the subsequent definitions of ISPD and Characteristic kernels, we suggest the extensive review done in \cite{Muandet2017} and references there in. 

Moreover, from functional analysis, a measure $\mu \in \mathfrak{M}_{\sqrt{K}}(X)$ satisfies $\langle K_{\mu}, K_{\mu}\rangle =0$ if and only if
\begin{equation}\label{rkhsfunctionalzero}
	\int_{X} f(x)d\mu(x)=0
\end{equation}
for every $f \in \mathcal{H}_{K}$. This follows because $f \in \mathcal{H}_{K}$ if and only if there exists $C>0$ such that $K(x,y)-C\,f(x)f(y)$ is a PD kernel, which is due to  Theorem 12, p. 30, in \cite{Berlinet2011} and the Kernel Mean Embedding.

A symmetric kernel $\gamma: X \times X \to \mathbb{R}$ is called Conditionally Negative Definite (CND) if, for every finite quantity of distinct points $x_{1}, \ldots, x_{n} \in X$ and scalars $c_{1}, \ldots, c_{n} \in \mathbb{R}$, with the restriction that $\sum_{i=1}^{n}c_{i}=0$, we have that
$$
\sum_{i, j =1}^{n}c_{i}c_{j} \gamma(x_{i}, x_{j}) \leq 0.
$$

CND kernels are intrinsically related to PD kernels: a symmetric kernel $\gamma: X\times X\to \mathbb{R}$ is CND if and only if, for any (equivalently, every) $w \in X$ the kernel
\begin{equation}\label{Kgamma}
	K^{\gamma}(x,y):=\gamma(x,w) + \gamma(w, y) - \gamma(x,y) - \gamma(w,w)
\end{equation}
is positive definite. This yields the relation between CND kernels and Hilbert spaces,  precisely, independently of the choice of $w \in X$, by Equation \ref{Kgamma} we have
\begin{equation}\label{geogamma}
	\gamma(x,y)= \frac{1}{2}\|(	K^{\gamma})_{x}- (	K^{\gamma})_{y}\|_{\mathcal{H}_ {	K^{\gamma}}}^{2} + \gamma(x,x)/2 + \gamma(y,y)/2.
\end{equation}

Another classical relation states that a symmetric kernel $\gamma: X\times X\to \mathbb{R}$ is CND if and only if, for every $r>0$, the kernel
\begin{equation}\label{schoenmetriccond}
	(x,y) \in X\times X \to e^{-r\gamma(x,y)}
\end{equation}
is PD. A useful inequality for a CND kernel $\gamma$ with $\gamma(x,x)=0$ for every $x\in X$ is
\begin{equation}\label{CNDINEQ}
	0 \leq \gamma(x_{1}, x_{2})\leq 2 \left [\gamma(x_{1}, x_{3}) + \gamma(x_{2}, x_{3})\right ],
\end{equation}
which holds for any $x_{1}, x_{2}, x_{3} \in X $ by taking $c_{1}=c_{2}=1$ and $c_{3}=-2$.

These classical results on PD and CND kernels are crucial for the development of the subject and can be found in Chapter 3 of \cite{Berg1984}.

It is also possible to define semi-inner products on subspaces of $\mathfrak{M}(X)$ using CND kernels, as described in the next lemma and whose proof uses Equation \ref{CNDINEQ}. We say that a kernel $\gamma: X \times X \to \mathbb{R}$ has a bounded diagonal if the function $x\mapsto \gamma(x,x)$ is bounded.

\begin{theorem}\label{estimativa} Let $\gamma: X \times X \to \mathbb{R}$ be a continuous CND kernel with bounded diagonal and $\mu \in \mathfrak{M}(X)$. The following are equivalent:
	\begin{enumerate}
		\item[$(i)$] $\gamma \in L^{1}(\vert \mu \vert \times \vert \mu \vert)$;
		\item[$(ii)$] The function $x \in X \to \gamma(x,z) \in L^{1}(\vert \mu \vert)$ for some $z \in X$;
		\item[$(iii)$] The function $x \in X \to \gamma(x,z) \in L^{1}( \vert \mu \vert )$ for every $z \in X$.
	\end{enumerate}
	Moreover, the set of measures satisfying these conditions is a vector space. In particular, consider the vector space
	$$
	\mathfrak{M}_{1}(X; \gamma):= \{ \eta \in \mathfrak{M}(X), \quad \gamma(x,y)\in L^{1}( \vert \eta \vert \times \vert \eta \vert ) \text{ and } \eta(X)=0 \},
	$$
	then the function
	$$
	(\mu, \nu ) \in \mathfrak{M}_{1}(X; \gamma) \times \mathfrak{M}_{1}(X; \gamma) \to I(\mu, \nu)_{\gamma}:=\int_{X} \int_{X} -\gamma(x,y)d\mu(x)d\nu(y)
	$$
	defines a semi-inner product on $\mathfrak{M}_{1}(X; \gamma)$.\end{theorem}

When the semi-inner product in the previous theorem is an inner product, we say that $\gamma$ is CND-Characteristic. The interesting aspect of a CND-Characteristic kernel $\gamma$ is that 
$$
E_{\gamma}(P,Q):=\sqrt{\int_{X} \int_{X}-\gamma(x,y)d[P-Q](x)d[P-Q](y) } = \sqrt{\|(K^{\gamma})_{P} - (K^{\gamma})_{Q}\|_{\mathcal{H}_{K^{\gamma}}}} 
$$
is a metric on the space of probability measures that satisfies any of the $3$ equivalent conditions in the first part of Theorem \ref{estimativa}. The pseudometric $E_{\gamma}$ is usually called the Energy distance \citep{Szekely2013}. Furthermore, the exact equivalence between this distance-based metric and the RKHS-based Maximum Mean Discrepancy via the associated PD kernel $K^{\gamma}$ was fundamentally established by \citep{Sejdinovic2013a}. A detailed proof of the integrability equivalences in Theorem \ref{estimativa} can be found in Section $3$ of \citep{Guella2022}.

The characterization of continuous CND radial kernels in all Euclidean spaces was proved in \cite{Schoenberg1938} and reads as follows: 
\begin{theorem}\label{reprcondneg} Let $\psi :[0, \infty)\to \mathbb{R}$ be a continuous function. The following conditions are equivalent
	\begin{enumerate}
		\item[$(i)$] The kernel
		$$
		(x,y) \in \mathbb{R}^{d}\times \mathbb{R}^{d} \to \psi(\|x-y\|^{2}) \in \mathbb{R}
		$$
		is CND for every $d \in \mathbb{N}$. 
		\item[$(ii)$] The function $\psi$ is a Bernstein function, that is $\psi \in C^{\infty}((0,\infty))$ and $\psi^{(1)}$ is completely monotone.
	\end{enumerate}
\end{theorem}

For more information on Bernstein functions see \cite{Schilling2012}.

For further background on PD/CND kernels and probability metrics, see \cite{Berg1975,Berlinet2011,Rachev2013}.

\subsection{Vector spaces of measures and probability interactions}\label{Terminology}
The results and terminology in this subsection are primarily drawn from \cite{Guella2025}.	

Let $X_{i}$, $1\leq i \leq n$, be nonempty sets, and consider the $n$-fold Cartesian product $\prod_{i=1}^{n}X_{i}$, which we denote as $\mathds{X}_{n}$.

For $m,n\in\mathbb{N}$, define $\mathbb{N}_{m}^{n}:=\{1,\ldots,m\}^{n}$ (which has $m^{n}$ elements). Similarly, define $\mathbb{N}_{m}^{0,n}:=\{0, 1, \ldots, m\}^{n}$ which has $(m+1)^{n}$ elements. If $x_{i}^{1}, \ldots, x_{i}^{m} \in X_{i}$, $1\leq i \leq n$, we define for $\alpha=(\alpha_{1}, \ldots, \alpha_{n}) \in \mathbb{N}_{m}^{n}$ (or $\mathbb{N}_{m}^{0,n}$) the element $x_{\alpha}:=(x_{1}^{\alpha_{1}}, \ldots, x_{n}^{\alpha_{n}})$.

We frequently use $\vec{1}$ for the vector with all entries equal to $1$ (similarly for $\vec{0}$ and $\vec{2}$). The dimension of these vectors are omitted as they are clear from the context. Also, for a subset $F \subset \{1, \ldots, n\}$ and coefficients $\alpha, \beta \in \mathbb{N}^{n}$, we use notations such as $x_{\alpha_{F} + \beta_{F^{c}}}$ to denote the element of $\mathds{X}_{n}$ whose coordinates in $F$ are those of $x_{\alpha}$ and whose coordinates in $F^{c}$ are those of $x_{\beta}$. In case $\beta= \vec{0}$, we often write $x_{\alpha_{F}}$ instead of $x_{\alpha_{F} + \vec{0}_{F^{c}}}$.

Although the definitions of Positive Definite Independent kernels presented in Sections \ref{Positivedefiniteindependentkernelsofordern} and \ref{IndependencetestsembeddedinHilbertspaces} are given in a discrete setting, it is convenient to use integral terminology to simplify expressions. For that, we define 

\begin{equation*}\label{medidaarbi}
	\mathcal{M}( \mathds{X}_{n}):=\{\text{The vector space of all discrete measures in } \mathds{X}_{n}\}. 
\end{equation*}

The continuous version, using finite Radon measures $\mathfrak{M}( \mathds{X}_{n})$ will also be used throughout the text. For $0\leq k\leq n$, important subspaces of $\mathcal{M}(\mathds{X}_{n})$ (and of $\mathfrak{M}(\mathds{X}_{n})$) are
\begin{equation}\label{medidamargk}
	\mathcal{M}_{k}( \mathds{X}_{n}):=\{\mu \in \mathcal{M}( \mathds{X}_{n}), \quad \mu (\prod_{i=1}^{n}A_{i} )=0, \text{ if } | \{i,\quad A_{i}=X_{i}\} |\geq n-k+1 \}, 
\end{equation}	
where $\mathfrak{M}_{k}( \mathds{X}_{n})$ is defined similarly. Note that $\mathcal{M}_{0}(\mathds{X}_{n})=\mathcal{M}(\mathds{X}_{n})$ and that $\mathcal{M}_{1}(\mathds{X}_{n})$ is related to the definition of CND kernels on $\mathds{X}_{n}$. They satisfy the following inclusions 

\begin{equation}\label{inclumeas}
	\mathcal{M}_{n}( \mathds{X}_{n}) \subset \mathcal{M}_{n-1}( \mathds{X}_{n}) \subset \ldots \subset \mathcal{M}_{2}( \mathds{X}_{n}) \subset \mathcal{M}_{1}( \mathds{X}_{n}) \subset \mathcal{M}_{0}( \mathds{X}_{n}), 
\end{equation}	
which is similar for the continuous case.

A simple but frequently used technical property is that if $f:\mathds{X}_{n}\to\mathbb{R}$ is integrable with respect to $ \mu \in \mathcal{M}_{k}(\mathds{X}_{n})$ (or $\mathfrak{M}_{k}( \mathds{X}_{n}) $) and depends on at most $k-1$ of the $n$ variables (for instance if $f(x_{1}, \ldots, x_{n}) = g(x_{1}, \ldots, x_{k-1})$ for some $g: \prod_{i=1}^{k-1}X_{i}\to \mathbb{R}$), then 
\begin{equation}\label{integmu0n}
	\int_{\mathds{X}_{n}}f(x_{1}, \ldots,x_{n})d\mu(x_{1}, \ldots, x_{n})=0.
\end{equation}

As an example (by the pigeonhole principle), if $\mu_{i}\in \mathcal{M}(X_{i})$ (or $\mathfrak{M}( X_{i}) $), $1\leq i \leq n$, with $|\{i:\ \mu_{i}(X_{i})=0\}|\ge k$, then $\bigtimes_{i=1}^{n}\mu_{i}\in \mathcal{M}_{k}(\mathds{X}_{n})$ (or $\mathfrak{M}_{k}( \mathds{X}_{n})$). This simple but crucial property is used, when possible, together with Theorem \ref{genlancastercartesianproduct} and Lemma \ref{exm02xn} to establish an equivalence between a generalized independence test of order $k$ in $n$ variables (based on the Streitberg interaction or the generalization of the Lancaster interaction) with the concept of PDI$_{k}$-Characteristic kernel on a $n$-Cartesian product space. See for instance the results in Subsection \ref{Bernsteinfunctionsofordern} and also Corollary \ref{lastkronprod}. 

\begin{remark}\label{hanhjordanequivalence} When $k \geq 1$, by the Hahn-Jordan decomposition, if $\mu \in \mathcal{M}_{k}(\mathds{X}_{n})$ (or $\mathfrak{M}_{k}( \mathds{X}_{n}) $) then there exist $M \in \mathbb{R}$ and probabilities $P, P^{\prime}$	in $\mathcal{M}(\mathds{X}_{n})$ (or $\mathfrak{M}( \mathds{X}_{n}) $) such that $P_{F}= (P^{\prime})_{F}$ for any $F \subset \{1, \ldots, n\}$ that satisfies $|F|=k-1$ and
	$$
	\mu = M[P - P^{\prime}].
	$$
	Similarly, if two probabilities $P$ and $P^{\prime}$	in $\mathcal{M}(\mathds{X}_{n})$ (or $\mathfrak{M} ( \mathds{X}_{n}) $) are such that $P_{F}= (P^{\prime})_{F}$ for any $F \subset \{1, \ldots, n\}$ that satisfies $|F|=k-1$, then $M[P - P^{\prime}]$ is an element of $\mathcal{M}_{k}(\mathds{X}_{n})$ (or $\mathfrak{M}_{k}( \mathds{X}_{n}) $) for every $M\in \mathbb{R}$.
\end{remark}

To obtain an important class of examples of measures in $\mathcal{M}_{k}(\mathds{X}_{n})$, let us recall the Lancaster interaction of a probability, see Chapter XII, p. 255, in \cite{Lancaster1969}:
\begin{equation*}
	\Lambda[P]:=	 \sum_{|F|=0}^{n} (-1)^{n-|F|} \left ( P_{F} \times \left [ \bigtimes_{j \in F^{c}}P_{j}\right ]\right ).
\end{equation*}

For probabilities $P,Q \in \mathcal{M}(\mathds{X}_{n})$ (or $\mathfrak{M}( \mathds{X}_{n}) $), the following generalization was proposed for $2\leq k\leq n$ in Section $3$ of \cite{Guella2025} 
$$
\Lambda_{k}^{n}[P,Q]:=P+ \sum_{j=0}^{k-1}(-1)^{k-j}\binom{n-j-1}{n-k}\sum_{|F|=j}P_{F}\times Q_{F^{c}}.
$$
and when $Q=\bigtimes_{i =1}^{n}P_{i}$ we simply write $\Lambda_{k}^{n}[P]$. Note that
$$
\Lambda_{n}^{n}[P ]=P+ \sum_{j=0}^{n-1}(-1)^{n-j}\sum_{|F|=j}P_{F}\times \left [\bigtimes_{i \in F^{c}}P_{i} \right ] = \Lambda[P],
$$	
$$
\Lambda_{2}^{n}[P]:= P + (n-1)\left(\bigtimes_{i=1}^{n}P_{i}\right) - n \left(\bigtimes_{i=1}^{n}P_{i}\right)= P - \bigtimes_{i=1}^{n}P_{i}.
$$

\begin{theorem}\label{generallancaster} The generalized Lancaster interaction $\Lambda_{k}^{n}$ satisfies the following properties:
	\begin{enumerate}
		\item [$(i)$] For probabilities $P, Q$ in $\mathcal{M}(\mathds{X}_{n})$ (or $\mathfrak{M}( \mathds{X}_{n}) $), the generalized Lancaster interaction $\Lambda_{k}^{n}[P,Q] \in \mathcal{M}_{k}(\mathds{X}_{n})$ (or $\mathfrak{M}_{k}(\mathds{X}_{n})$). 
		\item [$(ii)$] If for some $1\leq k \leq n-1$ we have that $\Lambda_{k}^{n}[P]=0$ then $\Lambda_{k+1}^{n}[P]=0$.
		\item [$(iii)$] $\Lambda_{n}^{n}[P]$ is multiplicative, in the sense that if $P= P_{\pi}$ for some partition $\pi=\{ F_{1}, \ldots, F_{\ell}\}$ of $\{1, \ldots, n\}$ then
		$$
		\Lambda_{n}^{n}[P]= \prod_{i=1}^{\ell}\Lambda_{|F_{i}|}^{|F_{i}|}[P_{F_{i}}].
		$$	 
	\end{enumerate}
\end{theorem}

Property (iii) highlights the differences between the Lancaster and Streitberg interaction (defined in Theorem \ref{streit}), as it implies that $\Lambda_{n}[P]=0$ when $P=P_{\pi}$ and one block of $\pi$ is a singleton. Property (ii) emphasizes that $\Lambda_{k}^{n}[P]$ serves as an indexed measure of independence for $P$.

For $1\leq k \leq n$ and $x_{\vec{1}}, x_{\vec{2}} \in \mathds{X}_{n} $, we define the measure
\begin{equation}\label{measorderk}
	\mu^{n}_{k}[x_{\vec{1}}, x_{\vec{2}}]:= \delta_{x_{\vec{1}}}+ \sum_{j=0}^{k-1}(-1)^{k-j}\binom{n-j-1}{n-k}\sum_{|F|=j}\delta_{x_{\vec{1}_{F} + \vec{2}_{F^{c}} }} = \Lambda_{k}^{n}[\delta_{x_{\vec{1}}},\delta_{x_{\vec{2}}}],
\end{equation}
which is an element of 	$\mathcal{M}_{k}( \mathds{X}_{n}) $ (and also $\mathfrak{M}_{k}( \mathds{X}_{n}) $). Moreover, if $L:=\{i:\ x_{i}^{1}\neq x_{i}^{2}\}$, then $\mu^{n}_{k}[x_{\vec{1}},x_{\vec{2}}]=0$ whenever $|L|<k$. This measure is a building block  to obtain a version of Equation \ref{Kgamma} to the context of PDI$_{k}$ kernels  in Section \ref{Positivedefiniteindependentkernelsofordern}	and Section \ref{IndependencetestsembeddedinHilbertspaces}.

We conclude our comments about the Lancaster interaction with the following result that is an important ingredient to obtain the equivalence mentioned at the introduction, that is of being a test for whether or not $\Lambda_{k}^{n}[P]\neq 0$ with the more restrictive tests for whether or not a  $\mu \in \mathfrak{M}_{k}( \mathds{X}_{n}) $ is the zero measure. As an example we have  Theorem \ref{bernsksevndimpart3} and  Corollary \ref{lastkronprod}.

\begin{theorem}\label{genlancastercartesianproduct}For $n\geq k \geq 2$, measures $\mu_{i} \in \mathcal{M}(X_{i})$ (or $\mathfrak{M}( X_{i}) $), $1\leq i \leq n$, with the restriction that $|i, \quad \mu_{i}(X_{i})=0|\geq k$, there exists an $M \geq 0$ and a probability $P$ in $\mathcal{M}(\mathds{X}_{n})$ (or $\mathfrak{M}( \mathds{X}_{n}) $) for which $\Lambda_{k}^{n}[P]= M(-1)^{n}(\bigtimes_{i=1}^{n}\mu_{i}) $.
\end{theorem}

Our final objective in this subsection is to define the Streitberg interaction. For that, let us recall that a partition $\pi$ of the set $\{1, \ldots, n\}$ is a collection of disjoint subsets $F_{1}, \ldots, F_{\ell}$ of $\{1, \ldots, n\}$, whose union is the entire set. In particular, we always have that $1\leq \ell \leq n$ and we sometimes use the notation $|\pi|$ to indicate $\ell$, that is, the amount of disjoint subsets in the partition $\pi$. Given a probability $P$ in $\mathcal{M}(\mathds{X}_{n})$ we define the probability 
\begin{equation}\label{partitionprobability}
	P_{\pi}:= \bigtimes_{i=1}^{\ell}P_{F_{i}}
\end{equation}
in $\mathcal{M}(\mathds{X}_{n})$,  	where $P_{F_{i}}$ is the marginal probability in $ \mathds{X}_{F_{i}}$.

A probability is called decomposable if there exists a partition $\pi$ with $|\pi| \geq 2$ for which $P = P_{\pi}$. When $n=2$, a probability $P$ is decomposable if and only if $P = P_{1}\times P_{2} $, and when $n=3$ a probability $P$ is decomposable  when 
$$
P_{123} - (P_{12}\times P_{3}) - (P_{13}\times P_{2})- (P_{23}\times P_{1})+ 2(P_{1}\times P_{2}\times P_{3}) \text{ is the zero measure}.
$$ 
However,  the converse is not true, as can be seen in Appendix C of \cite{Sejdinovic2013}.

When $n\geq 4$, a condition analogous to the $n=3$ case becomes more involved, a characterization is given in Proposition 2 of \cite{Streitberg1990}:

\begin{theorem}\label{streit}The collection of real numbers $a_{\pi}:= (-1)^{|\pi|-1}(|\pi|-1)!$, indexed over the partitions of the set $\{1, \ldots, n\}$, is the only one that satisfies the following conditions
	\begin{enumerate}
		\item[$(i)$] $a_{\{1,\ldots, n\}}=1$
		\item[$(ii)$] For any decomposable probability $P$ defined in the Cartesian product $\mathds{X}_{n}$ the measure
		$$
		\Sigma[P] := \sum_{\pi} a_{\pi} P_{\pi}
		$$
		is the zero measure.
		\item[$(iii)$] The operator $\Sigma$ is invariant under permutations of the coordinates: if   $\sigma: \{1,\ldots, n\} \to \{1,\ldots, n\} $ is a bijection, then
		$$
		\Sigma[P^{\sigma}]= [\Sigma[P]]^{\sigma}
		$$
		where for a measure $\mu$ in $\mathds{X}_{n}$ the measure $\mu^{\sigma}$ is defined in $\prod_{i=1}^{n}X_{\sigma(i)}$ for measurable sets $A_{j}$ of $X_{j}$ as $\mu^{\sigma}(\prod_{i=1}^{n}A_{\sigma(i)} ):=\mu (\prod_{i=1}^{n}A_{i} ) $.
	\end{enumerate}
	
\end{theorem}

The measure $\Sigma[P]$ is called the Streitberg interaction of the probability $P$. We emphasize that $\Sigma[P]$ can be the zero measure for a non decomposable probability. It can be proved that the amount of partitions in a set with $n$ elements is the Bell number $B_{n}$, see Section $26.7$ in \cite{NIST:DLMF}, defined by $B_{0}:=1$ and the recurrence relation
$$
B_{n+1}= \sum_{j=0}^{n}\binom{n}{j}B_{j}.
$$

Similar to the Lancaster interaction, the following result is valid.

\begin{lemma}\label{exm02xn} For a probability $P$ on $\mathcal{M}(\mathds{X}_{n})$ (or $\mathfrak{M}( \mathds{X}_{n}) $), the Streitberg interaction $\Sigma[P]$ is an element of $\mathcal{M}_{n}(\mathds{X}_{n})$ (or $\mathfrak{M}_{n}( \mathds{X}_{n}) $). Further, for measures $\mu_{i}$ in $\mathcal{M}(X_{i})$ (or $\mathfrak{M}( X_{i}) $) such that $\mu_{i}(X_{i})=0$, $1\leq i \leq n$, there exists an $M \geq 0$ and a probability $P$ in $\mathcal{M}(\mathds{X}_{n})$ (or $\mathfrak{M}( \mathds{X}_{n}) $) for which $\Sigma[P] = M(-1)^{n}(\bigtimes_{i=1}^{n}\mu_{i}) $.
\end{lemma}

\subsection{Radial PDI functions }\label{Bernsteinfunctionsofordern}

Our main objective in this subsection is to present  results from \cite{Guella2025} on the  theory of positive definite independent radial kernels of order $k$ in $n$ variables.

These results are not necessary for the development of the subject, but they serve as a guide to what type of results we  aim, in a similar way that the classical results of radial PD functions of Schoenberg inspired the results of Aronszajn on the RKHS theory and the  Kernel Mean Embedding results.

A function $h:(0, \infty)^{n}\to \mathbb{R}$ is completely monotone in $n$ variables if $h \in C^{\infty}((0, \infty)^{n})$ and $(-1)^{|\alpha|}\partial^{\alpha}h(t) \geq 0$, for every $\alpha \in \mathbb{Z}_{+}^{n}$ and $t\in (0,\infty)^{n}$. As in the Hausdorff-Bernstein-Widder theorem for one variable, the multivariable version is connected to Laplace transforms of nonnegative measures; see Section 4.2 of \cite{Bochner2005}.   Inspired by Theorem \ref {reprcondneg}, we define:

\begin{definition}For $0\leq k\leq n $, a function $g:(0, \infty)^{n} \to \mathbb{R}$ is called a Bernstein function of order $k$ if $g \in C^{\infty}((0, \infty)^{n})$ and the $\binom{n}{k}$ functions $[\partial^{\vec{1}_{F}}]g$ are completely monotone for every $|F|=k$. 
\end{definition}

The following indexed border of the set $[0, \infty)^{n}$  
$$
\partial_{k-1}^{n} :=\{t=(t_{1},\ldots, t_{n}) \in [0, \infty)^{n}, \quad |\{i, \quad t_{i } >0\}| < k\},
$$ 	
simplifies the characterization of the functions we want, and it will inspire the definition of  the extended diagonal of order $k-1$ in $\mathds{X}_{n}\times \mathds{X}_{n}$, denoted as 	$\Delta_{k-1}^{n}$, see Equation \ref{exdia1def}.

Now we state the main result of this subsection, which  contains Theorem $4.7$ and  Theorem $6.8$ of \cite{Guella2025}

\begin{theorem}\label{bernsksevndimpart3} Let $n\geq k\geq 2 $, $g:[0, \infty)^{n} \to \mathbb{R}$ be a continuous function such that $g(t)=0$ for every $t\in \partial_{k-1}^{n}$. The following conditions are equivalent:
	\begin{enumerate}
		\item [$(i)$] For any $d\in \mathbb{N}$ and discrete measures $\mu_{i}$ in $\mathbb{R}^{d}$, $1\leq i \leq n$, and with the restriction that $|i, \quad \mu_{i}(\mathbb{R}^{d})=0| \geq k $, it holds that
		$$
		\int_{(\mathbb{R}^{d})_{n}}\int_{ (\mathbb{R}^{d})_{n}}(-1)^{k}g(\|x_{1}-y_{1}\|^{2}, \ldots, \|x_{n} - y_{n}\|^{2})d\left [\bigtimes_{i=1}^{n}\mu_{i}\right ](x)d\left [ \bigtimes_{i=1}^{n}\mu_{i} \right ](y)\geq 0.
		$$
		\item [$(ii)$] For any $d\in \mathbb{N}$ and  discrete probability measure $P$ in $(\mathbb{R}^{d})_{n}$, it holds that
		$$
		\int_{(\mathbb{R}^{d})_{n}}\int_{ (\mathbb{R}^{d})_{n}}(-1)^{k}g(\|x_{1}-y_{1}\|^{2}, \ldots, \|x_{n} - y_{n}\|^{2})d[\Lambda_{k}^{n}[P] ](x)d[\Lambda_{k}^{n}[P]](y)\geq 0.
		$$
		\item [$(ii^{\prime})$] If $k=n$, then for any $d\in \mathbb{N}$ and  discrete probability measure $P$ in $(\mathbb{R}^{d})_{n}$, it holds that
		$$
		\int_{(\mathbb{R}^{d})_{n}}\int_{ (\mathbb{R}^{d})_{n}}(-1)^{k}g(\|x_{1}-y_{1}\|^{2}, \ldots, \|x_{n} - y_{n}\|^{2})d[\Sigma[P] ](x)d[\Sigma[P]](y)\geq 0.
		$$
		\item [$(iii)$] The function $g$ is PDI$_{k,n}^{\infty}$, that is 
		$$
		\int_{(\mathbb{R}^{d})_{n}}\int_{ (\mathbb{R}^{d})_{n}}(-1)^{k}g(\|x_{1}-y_{1}\|^{2}, \ldots, \|x_{n} - y_{n}\|^{2})d\mu(x)d\mu(y)\geq 0,
		$$
		for any $d\in \mathbb{N}$ and $\mu \in \mathcal{M}_{k}((\mathbb{R}^{d})_{n})$.
		\item[$(iv)$]The function $g$ is a Bernstein function of order $k$ in $(0, \infty)^{n}$.
	\end{enumerate}	
\end{theorem}

Theorem \ref{bernsksevndimpart3} still holds for $k=1$ and $k=0$ if condition (ii) is removed. For $k=1$, one may even replace $\Lambda_{k}^{n}[P]$ by the difference $P-Q$ of arbitrary discrete probability measures. In \cite{Guella2025} is also proved a Laplace transform type of integral representation using elementary symmetric polynomials for the functions that satisfies Theorem \ref{bernsksevndimpart3}, but we encourage the reader to look at it so that we omit defining these objects.

In Corollary $6.12$ of \cite{Guella2025} is also proved a version of this result for the strict case,

We emphasize that the assumption $g(t)=0$ for $t\in\partial_{k-1}^{n}$ simplifies the expression for $g$, since values on this set do not affect the double integrals in the theorem.

A  particular interesting result of the functions that satisfies Theorem \ref{bernsksevndimpart3} is that they are nonnegative and its growth is delimited by all the values of the function with $k$ variables.

\begin{corollary}\label{growthpdiknvar} Let $n> k\geq 1 $,  $g:[0, \infty)^{n} \to \mathbb{R}$ be a continuous function such that $g(t)=0$ for every $t\in \partial_{k-1}^{n}$ and  which satisfies the equivalences in Theorem \ref{bernsksevndimpart3}. Then  $g$ is nonnegative and increasing, in the sense that $g(t_{\vec{2}}) \geq g(t_{\vec{1}}) $ if $t_{\vec{2}} - t_{\vec{1}} \in [0, \infty)^{n}$. Also, it holds that
	$$
	0\leq \frac{k}{n}\sum_{|F|=k}g(t_{F}) \leq g(t) \leq \sum_{|F|=k}g(t_{F}), \quad t \in [0, \infty)^{n}.
	$$	
\end{corollary}

This corollary is one of the reasons for the term $(-1)^{k}$ in the double integrations of Theorem \ref{bernsksevndimpart3}.	This result is partially proved on the general context of PDI$_{k}$ kernels at Appendix \ref{geoint}.

\section{Positive definite independent kernels of order $n$}\label{Positivedefiniteindependentkernelsofordern}    

In this section, we define the concept of PDI kernels for several variables such that the case $n=1$ corresponds to CND kernels and the case $n=2$ recovers the PDI kernels presented in \cite{guella2023}. These kernels also generalize the PDI$_n$ radial functions presented in Theorem \ref{bernsksevndimpart3} of \cite{Guella2025}. Furthermore, we explain the relation between this new family of kernels and the concept of distance multivariance defined in \cite{Boettcher2018}, \cite{Boettcher2019}.

%   It is worth mentioning that several proofs in this section will proceed by induction on $n$, where the base case $n=2$ was established in \cite{guella2023}.

To avoid a combinatorial burden, we restrict our generalization to kernels that satisfy a multivariable symmetry relation analogous to the $2$-symmetry hypothesis in \cite{guella2023}. We say that a kernel $\mathfrak{I}: \mathds{X}_{n} \times \mathds{X}_{n}\to \mathbb{R}$ is $n$-symmetric if we can freely interchange elements between the two variables, that is,  
\[    
\mathfrak{I}(x_{\vec{1}}, x_{\vec{2}} )=     \mathfrak{I}((x_{1}^{\sigma_{1}(1)}, \ldots, x_{n}^{\sigma_{n}(1)}), (x_{1}^{\sigma_{1}(2)}, \ldots, x_{n}^{\sigma_{n}(2)}))    
\]
for every $x_{\vec{1}}=(x_{1}^{1}, \ldots, x_{n}^{1})$, $x_{\vec{2}}=(x_{1}^{2}, \ldots, x_{n}^{2}) \in \mathds{X}_{n}$ and bijective functions $\sigma_{i}: \{1,2\}\to \{1,2\}$, $1\leq i \leq n$ (there are $2^n$ such equalities). Alternatively, and frequently used in practice, $n$-symmetry can be defined by the condition
\begin{equation}\label{nsymmetrydefn}    
	\mathfrak{I}(x_{\vec{1}}, x_{\vec{2}} )=     \mathfrak{I}(x_{\alpha}, x_{\vec{3}- \alpha} ), \quad \alpha \in \mathbb{N}_{2}^{n},\quad  x_{\vec{1}}, x_{\vec{2}} \in \mathds{X}_{n}.     
\end{equation}
For $n=3$, this implies $\mathfrak{I}((x_1, y_1, z_1), (x_2, y_2, z_2)) = \mathfrak{I}((x_2, y_1, z_1), (x_1, y_2, z_2))$, and so forth. 

For an $n$-symmetric kernel $\mathfrak{I}$, the following type of double sum, which appears frequently, can be rewritten as
\begin{equation}\label{simplidouble}
	\begin{split}
		&\sum_{\alpha, \beta \in \mathbb{N}_{2}^{n}} C_{\alpha}C_{\beta} \mathfrak{I}(x_{\alpha}, x_{\beta}) \\
		&= \sum_{|F|=0}^{n} \sum_{\xi \in \mathbb{N}_{2}^{n-|F|}} \left [\sum_{\varsigma \in \mathbb{N}_{2}^{|F|}}C_{(\varsigma_{F}+\xi_{F^{c}})} C_{((\vec{3}-\varsigma)_{F}+ \xi_{F^{c}})} \right ] \mathfrak{I}(x_{\vec{1}_{F} + \xi_{F^{c}}},x_{\vec{2}_{F} + \xi_{F^{c}}} ).
	\end{split}
\end{equation}
Indeed, for a fixed pair $\alpha, \beta $, we have that $\alpha + \beta = \vec{3}_{F} + 2(\xi)_{F^{c}}$, where $F:=\{i, \quad \alpha(i) \neq \beta(i) \}$ and $\xi \in \mathbb{N}_{2}^{n-|F|}$ (since $F^{c}$ are the coordinates where $\alpha$ and $\beta$ are equal, we have to multiply $\xi_{F^{c}}$ by $2$). By the $n$-symmetry $ \mathfrak{I}(x_{\alpha}, x_{\beta})= \mathfrak{I}(x_{\alpha^{\prime}}, x_{\beta^{\prime}})$, whenever $\alpha + \beta = \alpha^{\prime} + \beta^{\prime}$, because in such a case, we have $\alpha= (\alpha_{F}, \xi_{F^{c}})$, $\beta= (\vec{3} - \alpha_{F}, \xi_{F^{c}})$, $\alpha^{\prime}= (\alpha^{\prime}_{F}, \xi_{F^{c}})$ and $\beta^{\prime}= (\vec{3} - \alpha^{\prime}_{F}, \xi_{F^{c}})$. Hence
\[
\sum_{ \alpha^{\prime} + \beta^{\prime}=\alpha + \beta }C_{\alpha^{\prime}}C_{\beta^{\prime}}\mathfrak{I}(x_{\alpha^{\prime}}, x_{\beta^{\prime}})= \left [\sum_{\varsigma \in \mathbb{N}_{2}^{|F|}}C_{(\varsigma_{F}+ \xi_{F^{c}})} C_{((\vec{3}-\varsigma)_{F}+ \xi_{F^{c}})} \right ] \mathfrak{I}(x_{\alpha }, x_{\beta }).
\]
The conclusion follows by summing over all possible values of $\alpha + \beta$, which is equivalent to the summation in Equation \eqref{simplidouble}.

For instance, given a function $g: [0, \infty)^{n} \to \mathbb{R}$ and CND kernels $\gamma_{i} : X_{i}\times X_{i} \to [0, \infty)$ for $1\leq i \leq n$, the composition $g(\gamma_{1}, \ldots, \gamma_{n})$ is an $n$-symmetric kernel on $\mathds{X}_{n}$. 

Now we define  the concept that extends the notion of distance multivariance.

\begin{definition}\label{PDIn}An $n$-symmetric kernel $\mathfrak{I}: \mathds{X}_{n} \times \mathds{X}_{n} \to \mathbb{R}$ is positive definite independent of order $n$ (PDI$_{n}$) if for every $\mu \in \mathcal{M}_{n}( \mathds{X}_{n})$ it satisfies
	\[
	\int_{\mathds{X}_{n} }    \int_{\mathds{X}_{n} }(-1)^{n}\mathfrak{I}(u,v)d\mu(u)d\mu(v) \geq 0.
	\]    
	If the previous inequality is an equality only when $\mu$ is the zero measure in $\mathcal{M}_{n}( \mathds{X}_{n})$, we say that $\mathfrak{I}$ is a strictly positive definite independent kernel of order $n$ (SPDI$_{n}$). 
\end{definition}

An important example of a PDI$_{n}$ kernel is the Kronecker product of $n$ CND kernels. Indeed, let $\gamma_{i}: X_{i} \times X_{i}\to \mathbb{R}$, $1\leq i \leq n$, be CND kernels and consider its Kronecker product
\[
\gamma (x, y):= \prod_{i=1}^{n} \gamma_{i}(x_{i}, y_{i}).
\]

By Equation \eqref{Kgamma} and Equation \eqref{integmu0n}, we have that for any fixed $x_{\vec{0}} \in \mathds{X}_{n} $ 
\begin{align*}
	&\int_{\mathds{X}_{n} }    \int_{\mathds{X}_{n} }(-1)^{n} \gamma (x_{\vec{1}}, x_{\vec{2}} )d\mu(x_{\vec{1}})d\mu(x_{\vec{2}})\\
	=& \int_{\mathds{X}_{n} }    \int_{\mathds{X}_{n} }(-1)^{n} \prod_{i=1}^{n} (\gamma_{i}(x_{i}^{1}, x_{i}^{2}) + \gamma_{n}(x_{i}^{0}, x_{i}^{0})- \gamma_{i}(x_{i}^{1},x_{i}^{0})- \gamma_{i}(x_{i}^{0}, x_{i}^{2}) )d\mu(x_{\vec{1}})d\mu(x_{\vec{2}}) \\
	=&\int_{\mathds{X}_{n} }    \int_{\mathds{X}_{n} }\prod_{i=1}^{n}    K^{\gamma_{i}}(x_{i}^{1}, x_{i}^{2})d\mu(x_{\vec{1}})d\mu(x_{\vec{2}})\geq 0,
\end{align*}
where the second equality occurs because the added terms do not depend on all $n$ variables of $x_{\vec{1}}$ or on all $n$ variables of $x_{\vec{2}}$. The final equality holds because the Kronecker product of PD kernels is a PD kernel as well. This property is essentially a generalization (in the discrete case, the continuous case is proved in Corollary \ref{lastkronprod}) to several variables of Theorem 24 in \cite{Sejdinovic2013a}, which proves the case for $n=2$.  

If $\gamma : X\times X \to \mathbb{R}$ is a CND kernel, then for every $\mu \in \mathcal{M}_{1}(X)$ 
\[
\int_{X}\int_{X}-\gamma(u,v)d\mu(u)d\mu(v) = \int_{X}\int_{X}\left [ \frac{\gamma(u,u)}{2} + \frac{\gamma(v,v)}{2} -\gamma(u,v) \right ]d\mu(u)d\mu(v), 
\]
because $\mu(X)=0$. Hence, in the analysis of the energy distance in Theorem \ref{estimativa}, we may assume that $\gamma$ is zero on the diagonal of $X$, that is, the function is zero on the set $\Delta_{0}^{1}:=\{(x,x), \quad x \in X\}$. Next, we generalize this property to PDI$_{n}$ kernels. To do so, we use the measure defined in Equation \eqref{measorderk} for $k=n$, precisely, for $x_{\vec{1}}, x_{\vec{2}} \in \mathds{X}_{n} $, the measure (and its equivalent definitions)
\begin{equation}\label{munn}
	\mu^{n}_{n}[x_{\vec{1}}, x_{\vec{2}}]:=\bigtimes_{i=1}^{n}[\delta(x_{i}^{1}) - \delta(x_{i}^{2}) ]  =    \sum_{j=0}^{n}(-1)^{n-j}\sum_{|F|=j}\delta(x_{\vec{2}-\vec{1}_{F}}) = \sum_{\alpha \in \mathbb{N}_{2}^{n}}(-1)^{n-|\alpha|}\delta(x_{\alpha}),
\end{equation}
and as stated in relation $(i)$ of Theorem \ref{generallancaster}, it is an element of the set $\mathcal{M}_{n}(\mathds{X}_{n})$. 

\begin{lemma}\label{PDInsimpli} Let $\mathfrak{I}: \mathds{X}_{n} \times \mathds{X}_{n} \to \mathbb{R}$ be an $n$-symmetric kernel. Consider the $n$-symmetric kernel $\mathfrak{I}^{\prime}: \mathds{X}_{n} \times \mathds{X}_{n} \to \mathbb{R}$ given by
	\begin{align*}
		\mathfrak{I}^{\prime} (x_{\vec{1}}, x_{\vec{2}} ):= \frac{(-1)^{n}}{2^{n}}\int_{\mathds{X}_{n} }\int_{\mathds{X}_{n} }\mathfrak{I}(u,v)d\mu^{n}_{n}[x_{\vec{1}}, x_{\vec{2}}](u)\mu^{n}_{n}[x_{\vec{1}}, x_{\vec{2}}](v). 
	\end{align*}
	Then, for any $\mu \in \mathcal{M}_{n}( \mathds{X}_{n})$ 
	\[ 
	\int_{\mathds{X}_{n}}\int_{\mathds{X}_{n}}(-1)^{n}\mathfrak{I}^{\prime} (u, v )d\mu(u)d\mu(v)=\int_{\mathds{X}_{n}}\int_{\mathds{X}_{n}}(-1)^{n}\mathfrak{I} (u, v )d\mu(u)d\mu(v),
	\]
	hence, $\mathfrak{I}^{\prime}$ is PDI$_{n}$ if and only if $\mathfrak{I}$ is PDI$_{n}$.\\
	If at least one of the coordinates of $x_{\vec{1}}$ and $x_{\vec{2}}$ is equal, then $\mathfrak{I}^{\prime}(x_{\vec{1}}, x_{\vec{2}})=0$.\\
	If $\mathfrak{I}(x_{\vec{1}}, x_{\vec{2}})=0$ whenever at least one coordinate of $x_{\vec{1}}$ and $x_{\vec{2}}$ is equal, then $\mathfrak{I}= \mathfrak{I}^{\prime}$.
\end{lemma}

\begin{proof}The kernel $\mathfrak{I}^{\prime}$ is $n$-symmetric because $\mu^{n}_{n}[x_{\alpha}, x_{\vec{3}-\alpha}]= (-1)^{n-|\alpha|}\mu^{n}_{n}[x_{\vec{1}}, x_{\vec{2}}]$ for every $\alpha \in \mathbb{N}_{2}^{n}$, then we obtain that $\mathfrak{I}^{\prime}(x_{\alpha},x_{\vec{3}-\alpha}) =\mathfrak{I}^{\prime}(x_{\vec{1}},x_{\vec{2}}) $.\\ 
	By Equation \eqref{simplidouble} the explicit expression for $\mathfrak{I}^{\prime}$ is 
	\begin{equation}\label{simplidoublen}
		\begin{split}
			\mathfrak{I}^{\prime}(x_{\vec{1}}, x_{\vec{2}})&=\sum_{\alpha, \beta \in \mathbb{N}_{2}^{n}}\frac{(-1)^{n-|\alpha| -|\beta|}}{2^{n}} \mathfrak{I}(x_{\alpha}, x_{\beta})\\
			&= \sum_{|F|=0}^{n} \sum_{\xi \in \mathbb{N}_{2}^{n-|F|}} (-1)^{n-|F|}2^{|F|-n} \mathfrak{I}(x_{\vec{1}_{F} + \xi_{F^{c}}},x_{\vec{2}_{F} + \xi_{F^{c}}} ). 
		\end{split}
	\end{equation}
	Note that if $|F|< n $, then $\mathfrak{I}(x_{\vec{1}_{F} + \xi_{F^{c}}},x_{\vec{2}_{F} + \xi_{F^{c}}} )$ depends on a maximum of $n-1$ among the $n$ variables of either $x_{\vec{1}}$ or $x_{\vec{2}}$, hence due to Equation \eqref{integmu0n}, for every $\mu \in \mathcal{M}_{n}( \mathds{X}_{n})$ 
	$$
	\int_{\mathds{X}_{n}}\int_{\mathds{X}_{n}}\mathfrak{I}(x_{\vec{1}_{F} + \xi_{F^{c}}},x_{\vec{2}_{F} + \xi_{F^{c}}} )d     \mu (x_{\vec{1}})d\mu(x_{\vec{2}})=0, \quad \xi \in \mathbb{N}_{2}^{n-|F|}, \quad |F|< n,
	$$
	which concludes the equality 
	$$
	\int_{\mathds{X}_{n}}\int_{\mathds{X}_{n}}(-1)^{n}\mathfrak{I}^{\prime} (u, v )d\mu(u)d\mu(v)=\int_{\mathds{X}_{n}}\int_{\mathds{X}_{n}}(-1)^{n}\mathfrak{I} (u, v )d\mu(u)d\mu(v).
	$$
	If at least one of the coordinates of $x_{\vec{1}}$ and $x_{\vec{2}}$ is equal then $ \mu^{n}_{n}[x_{\vec{1}}, x_{\vec{2}}]$ is the zero measure, consequently $\mathfrak{I}^{\prime}(x_{\vec{1}}, x_{\vec{2}})=0$.\\
	If $\mathfrak{I}(x_{\vec{1}}, x_{\vec{2}})=0$ whenever at least one coordinate of $x_{\vec{1}}$ and $x_{\vec{2}}$ is equal then for $\alpha + \beta \neq \vec{3}$, we obtain $\mathfrak{I}(x_{\alpha}, x_{\beta})=0$, thus
	$$
	\mathfrak{I}^{\prime} (x_{\vec{1}}, x_{\vec{2}} )= \frac{1}{2^{n}} \left[\sum_{\alpha \in \mathbb{N}_{2}^{n}}(-1)^{n-|\alpha|-|\vec{3}-\alpha|}\right ] \mathfrak{I}(x_{\alpha}, x_{\vec{3}-\alpha}) = \mathfrak{I} (x_{\vec{1}}, x_{\vec{2}} ).
	$$    \end{proof}   

We define the extended diagonal $\Delta_{n-1}^{n}$ of $\mathds{X}_{n}$ as the set
$$
\Delta_{n-1}^{n} :=\{(x_{\vec{1}}, x_{\vec{2}}) \in \mathds{X}_{n} \times \mathds{X}_{n}, \quad |\{ i, \quad x_{i}^{1}= x_{i}^{2}\}| \geq 1 \},
$$
that is, the set of pairs for which at least one of its coordinates is equal. The assumption that an PDI$_{n}$ kernel is zero at the extended diagonal $\Delta_{n-1}^{n}$, or in other words that  $\mathfrak{I}= \mathfrak{I}^{\prime}$,  simplifies several results, for instance in Theorem \ref{PDIngeometryrkhs} and Theorem \ref{integPDIn}, and as shown in the previous lemma, does not change the value of the double integration we aim to analyze.

A direct and frequently used consequence of this hypothesis is that a PDI$_{n}$ kernel $\mathfrak{I}$ that is zero at the extended diagonal $\Delta_{n-1}^{n}$ is a nonnegative function. To that end, take arbitrary $x_{\vec{1}}, x_{\vec{2}} \in \mathds{X}_{n}$ then
\begin{equation}\label{nonnegativity}
	0\leq \int_{\mathds{X}_{n} }\int_{\mathds{X}_{n} }(-1)^{n}\mathfrak{I}(u,v)d\mu^{n}_{n}[x_{\vec{1}}, x_{\vec{2}}](u)\mu^{n}_{n}[x_{\vec{1}}, x_{\vec{2}}](v)= 2^{n}\mathfrak{I}(x_{\vec{1}}, x_{\vec{2}}).
\end{equation}

Next lemma connects the concept of PDI$_{n}$ kernels in $\mathds{X}_{n}$ with the one of smaller order.

\begin{lemma}\label{fixedkern}Let $\mathfrak{I}: \mathds{X}_{n} \times \mathds{X}_{n} \to \mathbb{R}$ be a PDI$_{n}$ kernel which is zero at the extended diagonal $\Delta_{n-1}^{n}$ of $\mathds{X}_{n}$. Then, for every $F \subset \{1, \ldots, n\}$ with $1\leq |F|< n$ and $\lambda \in \mathcal{M}_{n-|F|}(\mathds{X}_{F^{c}} ) $, the kernel
	\begin{equation*}\label{eqcndkern}
		\mathfrak{I}_{\lambda}( x_{\vec{1}_{F}}, x_{\vec{2}_{F}}):=    (-1)^{n-|F|}\int_{\mathds{X}_{F^{c}}}\int_{\mathds{X}_{F^{c}}}\mathfrak{I}(( x_{\vec{1}_{F}},u),(x_{\vec{2}_{F}},v))d\lambda(u)d\lambda(v), 
	\end{equation*}
	is PDI$_{|F|}$ on $\mathds{X}_{F}$ that is zero in the extended diagonal $\Delta_{|F|-1}^{|F|}$ of $\mathds{X}_{F}$. In particular, for every fixed $x_{\vec{3}_{F^{c}}}, x_{\vec{4}_{F^{c}}} \in \mathds{X}_{F^{c}}$ the kernel
	$$
	( x_{\vec{1}_{F}}, x_{\vec{2}_{F}}) \to \mathfrak{I}( x_{\vec{1}_{F}+ \vec{3}_{F^{c}}},x_{\vec{2}_{F}+ \vec{4}_{F^{c}}}) \in \mathbb{R}
	$$
	is PDI$_{|F|}$ on $\mathds{X}_{F}$ that is zero in the extended diagonal $\Delta_{|F|-1}^{|F|}$ of $\mathds{X}_{F}$.
\end{lemma}

\begin{proof} Indeed, if $\mu \in \mathcal{M}_{|F|}(\mathds{X}_{F})$, we get that $\mu \times \lambda \in \mathcal{M}_{n}(\mathds{X}_{n} )$, and then
	\begin{align*}
		\int_{\mathds{X}_{F}}\int_{\mathds{X}_{F}}&    (-1)^{|F|}\mathfrak{I}_{\lambda}( x_{\vec{1}_{F}}, x_{\vec{2}_{F}})d\mu(x_{\vec{1}_{F}})d\mu(x_{\vec{2}_{F}})\\
		&= \int_{\mathds{X}_{n}}\int_{\mathds{X}_{n}}    (-1)^{n}\mathfrak{I}( z,w)d(\mu\times\lambda)(z)d(\mu\times\lambda)(w)\geq 0.
	\end{align*}    
	If any of the $|F|$ coordinates of $x_{\vec{1}_{F}}, x_{\vec{2}_{F}} \in \mathds{X}_{F}$ is equal, then $\mathfrak{I}(( x_{\vec{1}_{F}},u),(x_{\vec{2}_{F}},v))=0$ for any $u,v \in \mathds{X}_{F^{c}}$, because $\mathfrak{I}$ is zero at the extended diagonal $\Delta_{n-1}^{n}$, hence $\mathfrak{I}_{\lambda}( x_{\vec{1}_{F}}, x_{\vec{2}_{F}})=0$.\\
	For the second part, take $\lambda=\bigtimes_{i \in F^{c}}[ \delta(x_{i}^{4})-\delta(x_{i}^{3}) ]$, and note that
	\begin{align*}
		\mathfrak{I}_{\lambda}( x_{\vec{1}_{F}}, x_{\vec{2}_{F}})&=\sum_{\alpha, \beta \in \mathbb{N}_{2}^{F^{c}}}    (-1)^{n-|F|}(-1)^{|\alpha| + |\beta|}\mathfrak{I}( x_{\vec{1}_{F} +(\vec{2} + \alpha)_{F^{c}} },x_{\vec{2}_{F}+(\vec{2} + \beta)_{F^{c}}})\\
		&=2^{n-|F|}\mathfrak{I}( x_{\vec{1}_{F}+ \vec{3}_{F^{c}}},x_{\vec{2}_{F}+ \vec{4}_{F^{c}}}).
\end{align*}\end{proof}    

Now we provide a geometric interpretation of PDI$_{n}$ kernels, by connecting them to PD kernels, by generalizing Equation \eqref{Kgamma} for an arbitrary $n$.

\begin{lemma}\label{PDIntoPDn} Let $\mathfrak{I} : \mathds{X}_{n} \times \mathds{X}_{n} \to \mathbb{R} $ be an $n$-symmetric kernel and a fixed $x_{\vec{0}} \in \mathds{X}_{n} $. The kernel $ K^{\mathfrak{I}} : \mathds{X}_{n} \times \mathds{X}_{n} \to \mathbb{R} $ defined as
	$$
	K^{\mathfrak{I}}(x_{\vec{1}}, x_{\vec{2}}):= \int_{\mathds{X}_{n} }\int_{\mathds{X}_{n} }(-1)^{n}\mathfrak{I}(u,v)d\mu^{n}_{n}[x_{\vec{1}}, x_{\vec{0}}](u)\mu^{n}_{n}[x_{\vec{2}}, x_{\vec{0}}](v) 
	$$
	is PD if and only if $\mathfrak{I}$ is PDI$_{n}$.            
\end{lemma}

\begin{proof} Suppose that $\mathfrak{I}$ is PDI$_{n}$, then for arbitrary points $z_{1}, \ldots, z_{m} \in \mathds{X}_{n} $ and scalars $c_{1}, \ldots, c_{m} \in \mathbb{R}$ 
	\begin{align*}
		\sum_{i,j=1}^{m}c_{i}c_{j}&K^{\mathfrak{I}}(z_{i}, z_{j})\\
		&= \int_{\mathds{X}_{n} }\int_{\mathds{X}_{n} }(-1)^{n}\mathfrak{I}(u,v)d\left [\sum_{i=1}^{m}c_{i}\mu^{n}_{n}[z_{i}, x_{\vec{0}}] \right ](u)d\left [\sum_{j=1}^{m}c_{j}\mu^{n}_{n}[z_{j}, x_{\vec{0}}]\right ](v) \geq 0
	\end{align*}    
	because $\mathcal{M}_{n}( \mathds{X}_{n}) $ is a vector space.\\
	Conversely, if $K^{\mathfrak{I}}$ is PD, let $x^{1}_{i}, \dots, x^{n}_{i} \in X_{i}$, $1\leq i \leq n$ and scalars $c_{\alpha} \in \mathbb{R}$, $\alpha \in \mathbb{N}^{n}_{m}$ that satisfies the restrictions in Definition \ref{PDIn} (or equivalently, $    \sum_{\alpha \in \mathbb{N}_{m}^{n}}c_{\alpha}\delta_{x_{\alpha}} \in \mathcal{M}_{n}( \mathds{X}_{n})$), then 
	\begin{align*}
		0 &\leq \sum_{\alpha, \beta \in \mathbb{N}_{m}^{n}}c_{\alpha}c_{\beta}K^{\mathfrak{I}}(x_{\alpha},x_{\beta})\\
		&= \int_{\mathds{X}_{n} }\int_{\mathds{X}_{n} }(-1)^{n}\mathfrak{I}(u,v)d\left [\sum_{\alpha \in \mathbb{N}_{m}^{n}}c_{\alpha}\mu^{n}_{n}[x_{\alpha}, x_{\vec{0}}] \right ](u)d\left [\sum_{\beta \in \mathbb{N}_{m}^{n}}c_{\beta}\mu^{n}_{n}[x_{\beta}, x_{\vec{0}}]\right ](v).
	\end{align*}
	However
	$$
	\sum_{\alpha \in \mathbb{N}_{m}^{n}}c_{\alpha}\mu^{n}_{n}[x_{\alpha}, x_{\vec{0}}]=\sum_{\alpha \in \mathbb{N}_{m}^{n}}c_{\alpha} \delta_{x_{\alpha}}, 
	$$
	because by Equation \eqref{integmu0n} for any function $f:\mathds{X}_{n} \to \mathbb{R} $
	\begin{align*}
		\int_{\mathds{X}_{n}}f(u)&d\left [\sum_{\alpha \in \mathbb{N}_{m}^{n}}c_{\alpha}\mu^{n}_{n}[x_{\alpha}, x_{\vec{0}}]\right ](u) = \sum_{\alpha \in \mathbb{N}_{m}^{n}}c_{\alpha} \left [ f(x_{\alpha}) + \sum_{|F|\leq n-1}(-1)^{n-|F|} f(x_{ \alpha_{F}} ) \right] \\
		&= \sum_{\alpha \in \mathbb{N}_{m}^{n}}c_{\alpha} f(x_{\alpha}) = \int_{\mathds{X}_{n}}f(u)d\left [\sum_{\alpha \in \mathbb{N}_{m}^{n}}c_{\alpha}\delta_{x_{\alpha}} \right ](u),
	\end{align*}
	as for every fixed $|F|\leq n-1$, we get that $\sum_{\alpha \in \mathbb{N}_{m}^{n}}c_{\alpha} f(x_{ \alpha_{F}})=0$ by Equation \eqref{integmu0n}. Then 
	\begin{equation}\label{contaPDIntoPDn}
		\sum_{\alpha, \beta \in \mathbb{N}_{m}^{n}}c_{\alpha}c_{\beta}(-1)^{n}\mathfrak{I}(x_{\alpha},x_{\beta}) =\sum_{\alpha, \beta \in \mathbb{N}_{m}^{n}}c_{\alpha}c_{\beta}K^{\mathfrak{I}}(x_{\alpha},x_{\beta}) \geq 0.
\end{equation}\end{proof}            

We emphasize that the kernel $K^{\mathfrak{I}}$ depends on the choice of the element $x_{\vec{0}}$, which we omit to simplify the notation, however, the equality in Equation \eqref{contaPDIntoPDn} is independent of this choice. In the special case that $\mathfrak{I}$ is the Kronecker product of $n$ CND kernels, then $K^{\mathfrak{I}}$ is the Kronecker product of the $n$ relative PD kernels using the same point $x_{\vec{0}} \in \mathds{X}_{n}$. 

The explicit expression for $K^{\mathfrak{I}}$ is
\begin{equation}\label{newrepkI}
	K^{\mathfrak{I}} (x_{\vec{1}}, x_{\vec{2}}):= (-1)^{n} \sum_{i,j=0}^{n}(-1)^{i+j}\sum_{ |F|=i}\sum_{ |\mathcal{F}|=j}\mathfrak{I} (x_{\vec{1}_{F}}, x_{\vec{2}_{\mathcal{F}}}),
\end{equation}
and by a similar argument as the one in Lemma \ref{PDInsimpli}, if $\mathfrak{I}$ is $n$-symmetric we have that 
\begin{equation}\label{newrepkI2}
	K^{\mathfrak{I}} (x_{\vec{1}}, x_{\vec{1}})= 2^{n}\mathfrak{I} (x_{\vec{1}}, x_{\vec{0}}).
\end{equation}

As a consequence of Lemma \ref{PDIntoPDn} we can obtain another geometrical interpretation for PDI$_{n}$ kernels by connecting them to PD kernels by using the RKHS of the related positive definite kernel, which is a generalization of Equation \eqref{geogamma}. We use the notation $\mathcal{H}_{\mathfrak{I}}$ for the RKHS of the corresponding PD kernel $K^{\mathfrak{I}}$.

\begin{theorem}\label{PDIngeometryrkhs} Let $\mathfrak{I}: \mathds{X}_{n} \times \mathds{X}_{n} \to \mathbb{R}$ be an $n$-symmetric kernel which is zero at the extended diagonal $\Delta_{n-1}^{n}$, a fixed $x_{\vec{0}} \in \mathds{X}_{n}$ and the positive definite kernel $K^{\mathfrak{I}}: \mathds{X}_{n}\times \mathds{X}_{n} \to \mathbb{R}$ defined in Lemma \ref{PDIntoPDn}. The following equality is satisfied
	$$
	\left (    \left \|\sum_{\alpha \in \mathbb{N}_{2}^{n}} (-1)^{|\alpha|} K_{x_{\alpha}}^{\mathfrak{I}} \right \|_{\mathcal{H}_{\mathfrak{I}}}\right )^{2}= \sum_{\alpha, \beta \in \mathbb{N}_{2}^{n}} (-1)^{|\alpha| + |\beta|}K^{\mathfrak{I}}(x_{\alpha}, x_{\beta})= 2^{n}\mathfrak{I}(x_{\vec{1}}, x_{\vec{2}}).
	$$
\end{theorem} 

\begin{proof} The first equality is a consequence of the inner product in $\mathcal{H}_{\mathfrak{I}}$. For the second, the scalars $c_{\alpha}=(-1)^{|\alpha|}$, $\alpha \in \mathbb{N}_{2}^{n}$, satisfy the restrictions in Definition \ref{PDIn}, hence by Equation \eqref{contaPDIntoPDn} and the fact that $\mathfrak{I}$ is zero at the extended diagonal $\Delta_{n-1}^{n}$, we have that
	\begin{align*}
		\sum_{\alpha, \beta \in \mathbb{N}_{2}^{n}} (-1)^{|\alpha|}(-1)^{|\beta|} K^{\mathfrak{I}}(x_{\alpha}, x_{\beta} )&=    (-1)^{n}\sum_{\alpha, \beta \in \mathbb{N}_{2}^{n}} (-1)^{|\alpha|}(-1)^{|\beta|} \mathfrak{I}(x_{\alpha}, x_{\beta} )\\
		&= (-1)^{n}\sum_{\alpha \in \mathbb{N}_{2}^{n}} (-1)^{|\alpha|}(-1)^{|\vec{3} - \alpha|} \mathfrak{I}(x_{\alpha}, x_{\vec{3} - \alpha} )\\
		&=\sum_{\alpha \in \mathbb{N}_{2}^{n}} \mathfrak{I}(x_{\vec{1}}, x_{\vec{2}} )= 2^{n}\mathfrak{I}(x_{\vec{1}}, x_{\vec{2}} ).
\end{align*}\end{proof}

An immediate and useful inequality for the results in Subsection \ref{Integrabilityrestrictionsn} and a generalization of Equation \eqref{CNDINEQ} is the following. 

\begin{corollary}\label{ineqPDIngeometryrkhs}Let $\mathfrak{I}: \mathds{X}_{n} \times \mathds{X}_{n} \to \mathbb{R}$ be an $n$-symmetric kernel which is zero at the extended diagonal $\Delta_{n-1}^{n}$. Then, the following inequalities are satisfied
	$$
	\mathfrak{I}(x_{\vec{1}}, x_{\vec{2}}) \leq \left (\sum_{\alpha \in \mathbb{N}_{2}^{n}}\sqrt{\mathfrak{I}(x_{\alpha}, x_{\vec{0}})} \right )^{2} \leq 2^{n} \sum_{\alpha \in \mathbb{N}_{2}^{n}}\mathfrak{I}(x_{\alpha}, x_{\vec{0}}) 
	$$
	for every $x_{\vec{0}}, x_{\vec{1}}, x_{\vec{2}} \in \mathds{X}_{n}$.
\end{corollary} 

\begin{proof}
	Indeed, by using the PD kernel related to the element $x_{\vec{0}} \in \mathds{X}_{n}$ of Lemma \ref{PDIntoPDn}, the triangle inequality and Equation \eqref{newrepkI2} we get that
	$$
	\left \|\sum_{\alpha \in \mathbb{N}_{2}^{n}} (-1)^{|\alpha|} K_{x_{\alpha}}^{\mathfrak{I}} \right \|_{\mathcal{H}_{\mathfrak{I}}}\leq \sum_{\alpha \in \mathbb{N}_{2}^{n}} \left \| K_{x_{\alpha}}^{\mathfrak{I}} \right \|_{\mathcal{H}_{\mathfrak{I}}}=\sum_{\alpha \in \mathbb{N}_{2}^{n}} \sqrt{K^{\mathfrak{I}}(x_{\alpha},x_{\alpha} )} =2^{n/2}\sum_{\alpha \in \mathbb{N}_{2}^{n}} \sqrt{\mathfrak{I}(x_{\alpha},x_{\vec{0}} )},
	$$
	thus obtaining the first inequality by using Theorem \ref{PDIngeometryrkhs}. For the last inequality, since for every real numbers $|ab|\leq (a^{2} + b^{2})/2$ we have that
	$$
	\left (\sum_{\alpha \in \mathbb{N}_{2}^{n}}\sqrt{\mathfrak{I}(x_{\alpha}, x_{\vec{0}})} \right )^{2}= \sum_{\alpha, \beta \in \mathbb{N}_{2}^{n} }\sqrt{\mathfrak{I}(x_{\alpha}, x_{\vec{0}})}\sqrt{\mathfrak{I}(x_{\beta}, x_{\vec{0}})} \leq 2^{n} \sum_{\alpha \in \mathbb{N}_{2}^{n}}\mathfrak{I}(x_{\alpha}, x_{\vec{0}}). 
	$$\end{proof}

Surprisingly, for $n\geq 2$, there is no direct analogue of Schoenberg's characterization in Equation \eqref{schoenmetriccond} that connects PDI$_{n}$ kernels with PD kernels.

\begin{lemma} Let $n\geq 2$, $\mathfrak{I}: \mathds{X}_{n}\times\mathds{X}_{n} \to [0, \infty) $ be a PDI$_{n}$ kernel which is zero at the extended diagonal $\Delta_{n-1}^{n}$ and also a function $f: [0, \infty) \to \mathbb{R}$. The kernel
	$$
	f(\mathfrak{I}(x_{\vec{1}}, x_{\vec{2}})), \quad x_{\vec{1}}, x_{\vec{2}} \in \mathds{X}_{n}
	$$
	is positive definite if and only if this is a constant kernel.
\end{lemma}

\begin{proof}
	Indeed, pick $x_{\vec{1}}$ and $x_{\vec{2}}$ for which all of its coordinates are different. Since the kernel is positive definite, the interpolation matrix at the $2^{n}$ points $x_{\alpha}$, $\alpha \in \mathbb{N}_{2}^{n}$, is 
	$$
	A:=[f(\mathfrak{I}(x_{\alpha}, x_{\beta}))]_{\alpha, \beta}= [f(0)]_{\alpha, \beta}+ [(f(c)- f(0))\delta_{\vec{3}}]_{\alpha, \beta} 
	$$
	where $c := \mathfrak{I}(x_{\alpha}, x_{\beta})$ for every $\alpha +\beta = \vec{3}$. However, for scalars $v_{\alpha}= (-1)^{\alpha_{1}}$ and $u_{\alpha}= (-1)^{\alpha_{1} + \alpha_{2}}$, which satisfy that $\sum_{\alpha \in \mathbb{N}_{2}^{n} } v_{\alpha} = \sum_{\alpha \in \mathbb{N}_{2}^{n} } u_{\alpha}=0 $, we have that
	\begin{align*}
		\sum_{\alpha, \beta \in \mathbb{N}_{2}^{n}}v_{\alpha}v_{\beta} [f(0) + (f(c)- f(0))\delta_{\vec{3}}]&= \sum_{\alpha+ \beta= \vec{3}}v_{\alpha}v_{\beta} [ (f(c)- f(0))]\\
		&= \sum_{\alpha \in \mathbb{N}_{2}^{n}}(-1)^{\alpha_{1}}(-1)^{3-\alpha_{1}} [ (f(c)- f(0))]\\
		&= -2^{n}[ (f(c)- f(0))],
	\end{align*}
	and similarly
	\begin{align*}
		\sum_{\alpha, \beta \in \mathbb{N}_{2}^{n}}&u_{\alpha}u_{\beta} [f(0) + (f(c)- f(0))\delta_{\vec{3}}]\\
		&= \sum_{\alpha \in \mathbb{N}_{2}^{n}}(-1)^{\alpha_{1}+\alpha_{2}}(-1)^{3-\alpha_{1} + 3-\alpha_{2}} [ (f(c)- f(0))]= 2^{n}[ (f(c)- f(0))].
	\end{align*}
	Since by the hypothesis the matrix $A$ is positive semidefinite, these two relations imply that $f(c)=f(0)$.\end{proof}

\subsection{Integrability restrictions}\label{Integrabilityrestrictionsn}   

In this subsection, we prove the technical results regarding the description of the continuous probability measures that can be compared using a PDI$_{n}$ kernel. First, we make a few observations regarding the use of PD kernels for testing the independence of two variables.  

Let $K: [X\times Y] \times [X \times Y] \to \mathbb{R}$ be a continuous PD kernel. As mentioned in Theorem \ref{initialextmmddominio}, for measures $\lambda, \eta \in \mathfrak{M}_{\sqrt{K}}(X\times Y)$, then
$$
\int_{X\times Y}\int_{X\times Y}K((x_{1},y_{1}),(x_{2},y_{2}))d\lambda(x_{1},y_{1})d\eta(x_{2},y_{2})= \langle K_{\lambda}, K_{\eta}\rangle, 
$$
and the set of measures that satisfies these properties is a vector space. As expected by measure theory, if we want to use this kernel to analyze independence in two variables, the fact that a probability $P  \in \mathfrak{M}_{\sqrt{K}}(X\times Y)$ does not immediately imply that $P_{X}\times P_{Y}  \in \mathfrak{M}_{\sqrt{K}}(X\times Y)$, hence we must add this hypothesis. However, even this does not lead to satisfactory results, as the integrals 
\[
\int_{X}\sqrt{K((x_{1},y),(x_{1},y))}dP_{X}(x_{1}), \quad \int_{Y}\sqrt{K((x,y_{1}),(x,y_{1}))}dP_{Y}(y_{1})
\]
may not exist for a fixed but arbitrary $x \in X$, $y \in Y$. Regarding this last point, we emphasize that for a Kronecker product of kernels, such as in HSIC or Distance Covariance, this issue is solved on the literature by using non-degenerate probabilities and, implicitly, by the structure of this Kronecker product. 

The concept that will allow us to pursue these ideas for an arbitrary PDI$_{n}$ kernel (but it can also be used on a PD kernel on $\mathds{X}_{n}$) is the following.

\begin{definition}\label{FIP} A function $f: \mathds{X}_{n} \to \mathbb{R}$ is \textbf{Fully Integrable by Partitions (FIP)} with respect to a measure $\mu \in \mathfrak{M}(\mathds{X}_{n})$ if for every $x_{\vec{2}} \in \mathds{X}_{n}$, partition $\pi $ and subset $F$ of $\{1, \ldots, n\}$, it holds that
	$$
	\int_{\mathds{X}_{n}}\left |f(x_{\vec{1}_{F} + \vec{2}_{F^{c}}}) \right |d|\mu|_{\pi}(x_{\vec{1}})< \infty.
	$$
\end{definition}

There is a redundancy in the integrability restrictions in Definition \ref{FIP} and on subsequent uses of it, because we only need to check the partitions of the subset $F$ and not on the whole set $\{1,\ldots, n\}$. It is presented in this way to simplify the terminology. Hence, the amount of integrability restrictions for a fixed $x_{\vec{2}} \in \mathds{X}_{n}$ is (including the case $F=\emptyset$) 
$$
\sum_{j=0}^{n}\binom{n}{j}B_{j}= B_{n+1},
$$ 
while the number that appears in Definition \ref{FIP} is $2^{n}B_{n}$.

We develop the theory regarding FIP integrals separately, on Appendix \ref{FIPintegrals} as they are of importance on their own. The remaining results in this section are obtained through them. Other reasons for using this approach are presented in Section \ref{futurework}.

By Definition \ref{FIP}, a function $T: \mathds{X}_{n} \times \mathds{Y}_{m} \to \mathbb{R}$ is FIP with respect to $\mu\times \nu \in \mathfrak{M}(\mathds{X}_{n} \times \mathds{Y}_{m})$ if for every $x_{\vec{2}} \in \mathds{X}_{n}$, $y_{\vec{2}} \in \mathds{Y}_{m}$, partition $\pi $ and subset $F$ of $\{1, \ldots, n\}$, and partition $\sigma $ and subset $G$ of $\{1, \ldots, m\}$, it holds that
$$
\int_{\mathds{X}_{n} \times \mathds{Y}_{m}}|T(x_{\vec{1}_{F} + \vec{2}_{F^{c}}}, y_{\vec{1}_{G} + \vec{2}_{G^{c}}} )| d |\mu|_{\pi} (x_{\vec{1}})d |\nu|_{\sigma} (y_{\vec{1}}) < \infty.
$$

\begin{lemma}\label{probintegPDIn} Let $\mathfrak{I}: \mathds{X}_{n} \times \mathds{X}_{n} \to \mathbb{R}$ be a continuous PDI$_{n}$ kernel that is zero at the extended diagonal $\Delta_{n-1}^{n}$ of $\mathds{X}_{n}$. Then, the following conditions are equivalent for a measure $ \mu \in \mathfrak{M}(\mathds{X}_{n})$
	\begin{enumerate}
		\item[$(i)$] The function 
		$$
		(x_{\vec{1}}, x_{\vec{2}}) \in \mathds{X}_{n}\times \mathds{X}_{n} \to \mathfrak{I}(x_{\vec{1} },x_{\vec{2}}) \in \mathbb{R}
		$$
		is FIP with respect to $\mu\times \mu \in \mathfrak{M}(\mathds{X}_{n}\times \mathds{X}_{n})$.
		\item[$(ii)$] There exists an element $x_{\vec{3}} \in \mathds{X}_{n}$ such that
		$$ 
		x_{\vec{1}} \in \mathds{X}_{n} \to \mathfrak{I}(x_{\vec{1} },x_{\vec{3}}) \in \mathbb{R}
		$$
		is FIP with respect to $\mu $.
		\item[$(iii)$] For every element $x_{\vec{3}} \in \mathds{X}_{n}$, the function 
		$$ 
		x_{\vec{1}} \in \mathds{X}_{n} \to \mathfrak{I}(x_{\vec{1} },x_{\vec{3}}) \in \mathbb{R}
		$$
		is FIP with respect to $\mu $.
		\item[$(iv)$] For whichever $x_{\vec{0}} \in \mathds{X}_{n}$ that is used to define $K^{\mathfrak{I}}$ defined in Lemma \ref{PDIntoPDn}, the function 
		$$ 
		x_{\vec{1}} \in \mathds{X}_{n} \to K^{\mathfrak{I}}(x_{\vec{1}}, x_{\vec{1}}) \in \mathbb{R}
		$$
		is FIP with respect to $\mu $.
	\end{enumerate}   
\end{lemma}

\begin{proof} Suppose, without loss of generality, that $|\mu|$ is a probability that we denote as $P$. Also, all functions involved are nonnegative by Equation \eqref{nonnegativity}.\\
	Suppose that relation $(i)$ is valid. By the comment made before the statement of this lemma, for any $x_{\vec{3}}, x_{\vec{4}} \in \mathds{X}_{n}$, subset $F$ and partition $\pi$ of $\{1, \ldots, n \}$
	$$
	\int_{\mathds{X}_{n}}\mathfrak{I}(x_{\vec{1}_{F}+ \vec{4}_{F^{c}} },x_{\vec{3}})dP_{\pi}(x_{\vec{1}})= \int_{\mathds{X}_{n} \times \mathds{X}_{n} }\mathfrak{I}(x_{\vec{1}_{F}+ \vec{4}_{F^{c}} },x_{\vec{3}})dP_{\pi}(x_{\vec{1}})dP_{\{1,\ldots, n\}} (x_{\vec{2}}) < \infty, 
	$$
	thus relation $(iii)$ is valid, and in particular relation $(ii)$ is valid. \\
	If relation $(ii)$ is valid for a fixed $x_{\vec{3}} \in \mathds{X}_{n}$, then by Corollary \ref{ineqPDIngeometryrkhs}, for every $x_{\vec{2}} \in \mathds{X}_{n}$ we have that
	\begin{equation}\label{probintegPDIneq1}
		0 \leq \mathfrak{I}(x_{\vec{1}},x_{\vec{2}}) \leq 2^{n} \sum_{\alpha \in \mathbb{N}_{2}^{n}}\mathfrak{I}(x_{\alpha},x_{\vec{3}})= 2^{n}\sum_{|\mathcal{F}|=0}^{n}\mathfrak{I}(x_{\vec{1}_{\mathcal{F}}+ \vec{2}_{\mathcal{F}^{c}} },x_{\vec{3}}). 
	\end{equation}
	As all functions $ x_{\vec{1}} \to \mathfrak{I}(x_{\vec{1}_{\mathcal{F}}+ \vec{2}_{\mathcal{F}^{c}} },x_{\vec{3}})$ are FIP with respect to $P$ by Lemma \ref{propertiesFIP}, we obtain that $ x_{\vec{1}} \to \mathfrak{I}(x_{\vec{1} },x_{\vec{2}})$ is FIP with respect to $P$, concluding that relation $(iii)$ is valid.\\
	Relation $(iv)$ and relation $(iii)$ are equivalent by Equation \eqref{newrepkI2} and setting $x_{\vec{0}} = x_{\vec{3}}$.\\
	If relation $(iii)$ is valid, by Equation \eqref{probintegPDIneq1}, to prove that relation $(i)$ is valid is sufficient to prove that 
	$$
	(x_{\vec{1}}, x_{\vec{2}}) \in \mathds{X}_{n}\times \mathds{X}_{n} \to     \mathfrak{I}(x_{\vec{1}_{\mathcal{F}}+ \vec{2}_{\mathcal{F}^{c}} },x_{\vec{3}}) \in \mathbb{R}
	$$
	is FIP with respect to $P\times P$ for every subset $\mathcal{F}$ of $\{1, \ldots, n\}$, which follows as all integrals are special cases of the integrals 
	$$
	\int_{\mathds{X}_{m}}\mathfrak{I}(x_{\vec{5}_{F}+ \vec{6}_{F^{c}} },x_{\vec{3}})dP_{\pi}(x_{\vec{5}})
	$$
	that appear on the hypothesis that $ x_{\vec{1}} \to \mathfrak{I}(x_{\vec{1} },x_{\vec{3}})$ is FIP with respect to $P$ (the term $x_{\vec{6}}$ is the constant term in Definition \ref{FIP}).\end{proof}

In particular, Theorem \ref{estimativa} is the special case of Lemma \ref{probintegPDIn} when $n=1$. For a kernel $\mathfrak{I}$ that satisfies the requirements of Lemma \ref{probintegPDIn} we define the sets
\begin{equation}\label{intsetsn}
	\begin{split}
		&\mathfrak{M}[\mathfrak{I}]:=\{ \mu, \quad \mu \in \mathfrak{M}(\mathds{X}_{n}) \text{ satisfies Lemma \ref{probintegPDIn}} \},\\
		&\mathfrak{M}_{n}[\mathfrak{I}]:= \mathfrak{M}_{n}(\mathds{X}_{n}) \cap \mathfrak{M}[\mathfrak{I}].\\
	\end{split}
\end{equation}

With this approach we still cannot define the concept of PDI$_{n}$-Characteristic using a single vector space such as the CND-Characteristic property in Theorem \ref{estimativa}, as we cannot guarantee, in general, that the set $\mathfrak{M}_{n}[\mathfrak{I}]$ is a vector space. The same would occur if we have proved Lemma \ref{probintegPDIn} for a PD kernel. Below we show an example.

\begin{example}\label{exkerneldiag} Let $X= Y= \mathbb{N}$, consider the PD kernel 
	\[
	K((n,m),(a,b)):= \frac{nmab}{n+a+m+b}= \int_{(0, \infty)}nme^{-ar-mr}abe^{-ar -br}dr,
	\] 
	and the probabilities $P(n,m) = \zeta(5/2)^{-1} \zeta(4)^{-1}n^{-5/2}m^{-4} $ and $Q(n,m) := P(m,n) $ in $\mathbb{N}^{2}$, where $\zeta$ is the Riemann Zeta function. We claim that $f(n,m):=K((n,m),(n,m))$ is FIP with respect to $P$ and $Q$, but is not FIP with respect to $(P+Q)/2$.\\
	Indeed,
	$$
	2\zeta(5/2) \zeta(4)    \sum_{n,m=1}^{\infty} f(n,m)P(n,m) = \sum_{n,m=1}^{\infty} \frac{1}{n^{1/2}m^{2}(n+m)}\leq \sum_{n,m=1}^{\infty}\frac{1}{n^{3/2}m^{2}}< \infty,
	$$
	and since $P=P_{1}\times P_{2}$, we also have that 
	$$
	\sum_{n,m=1}^{\infty} f(n,m)P_{1}(n)P_{2}(m) < \infty.
	$$   
	Also, for every fixed $m \in \mathbb{N}$
	$$
	2\zeta(5/2)\sum_{n=1}^{\infty} f(n,m)P_{1}(n)= \sum_{n =1 }^{\infty}\frac{m^{2}}{n^{1/2}(n+m)}\leq \sum_{n=1 }^{\infty}\frac{m^{2}}{n^{3/2}}= m^{2} \zeta(3/2),
	$$
	and to conclude, for every fixed $n \in \mathbb{N}$
	$$
	2\zeta(4)\sum_{m=1}^{\infty} f(n,m)P_{2}(m)= \sum_{m =1 }^{\infty}\frac{n^{2}}{m^{2}(n+m)}\leq \sum_{m=1 }^{\infty}\frac{n^{2}}{m^{3}}= n^{2}\zeta( 3).
	$$   
	The same properties hold for the probability $Q$, as the function $f$ is symmetric in the sense that $f(n,m)=f(m,n)$ and the definition of $Q$. However, $f$ is not integrable with respect to 
	$$
	\left (\frac{P+Q}{2} \right )_{1}\times \left (\frac{P+Q}{2} \right )_{2}.
	$$     
	To prove this, it is sufficient to show that $f$ is not integrable with respect to $P_{1}\times Q_{2}$. Indeed, our aim is to prove that 
	$$
	2\zeta(5/2)^{2}\sum_{n,m=1}^{\infty} f(n,m)P_{1}(n)Q_{2}(m) = \sum_{n,m=1}^{\infty}\frac{1}{n^{1/2}m^{1/2}(n+m)},
	$$
	diverges. After the change of variables $n+m=k $ and $n=l$, this double sum is equivalent to
	$$
	\sum_{k=2}^{\infty}\sum_{l=1}^{k-1}\frac{1}{l^{1/2}(k-l)^{1/2}k},
	$$
	but, for $0\leq l \leq k$ we have that $l^{1/2}(k-l)^{1/2} \leq k/2$, thus 
	\begin{align*}
		\sum_{k=2}^{\infty}\sum_{l=1}^{k-1}\frac{1}{l^{1/2}(k-l)^{1/2}k} &\geq     \sum_{k=2}^{\infty}\sum_{l=1}^{k-1}\frac{2}{k^{2}}= \sum_{k=2}^{\infty}\frac{2(k-1)}{k^{2}}=\infty,
	\end{align*}
	because the harmonic series diverges.
\end{example}

We overturn this issue by adapting to the FIP integrals a concept that appeared in Theorem 3.8 in \cite{guella2023}, in the context of $n=2$.

\begin{definition}\label{margdeli} Let $P \in \mathfrak{M}(\mathds{X}_{n})$ be a probability and $\mu \in \mathfrak{M}(\mathds{X}_{n})$. We say that $\mu$ is \textbf{marginally delimited} by $P$, if there exists an $M\geq 1 $ such that for any $F \subset \{1, \ldots, n\}$ with $|F|\leq n-1$
	it holds that the measure
	$$
	M \sum_{\pi \in \mathscr{P}(F)} (P_{F})_{\pi} - |\mu|_{F} \text{ is nonnegative}, 
	$$
	where $\mathscr{P}(F)$ stands for the set of partitions of the set $F$.
\end{definition}

We need only to verify this definition on subsets of size $n-1$, as the other cases would be a consequence of it, but the arguments of a few results are simplified when using the full definition. 

\begin{lemma} \label{integPDIn0}Let $\mathfrak{I}: \mathds{X}_{n} \times \mathds{X}_{n} \to \mathbb{R}$ be a continuous $n$-symmetric PDI$_{n}$ kernel that is zero at the extended diagonal $\Delta_{n-1}^{n}(\mathds{X}_{n})$. Consider the sets
	\begin{align*}
		&\mathcal{P}[\mathfrak{I}]:=\{ Q, \quad Q \text{ is a probability and satisfies Lemma \ref{probintegPDIn}} \},\\
		&\mathfrak{M}[\mathfrak{I}, P]:=\{ \mu \in \mathfrak{M}[\mathfrak{I}] \text{ and $\mu$ is marginally delimited by $P \in \mathcal{P}[\mathfrak{I}]$} \},\\
		&\mathfrak{M}_{n }[\mathfrak{I}, P]:=\mathfrak{M}[\mathfrak{I}, P] \cap \mathfrak{M}_{n}(\mathds{X}_{n}).
	\end{align*}
	Then, we have that: 
	\begin{enumerate}
		\item [$(i)$] If $P \in \mathcal{P}[\mathfrak{I}]$ then every linear combination $\sum_{\pi} a_{\pi}P_{\pi}$ is an element of $ \mathfrak{M}[\mathfrak{I}]$.     
		\item [$(ii)$] If $P \in \mathcal{P}[\mathfrak{I}]$ and $\mu \in \mathfrak{M} (\mathds{X}_{n})$ is a measure for which there exists a constant $M \geq 0$ such that the measure $MP -|\mu|$ is nonnegative, then $\mu \in \mathcal{M}[\mathfrak{I}, P]$.
		\item [$(iii)$] For any fixed $P \in \mathcal{P}[\mathfrak{I}]$, the sets $\mathfrak{M}_{n }[\mathfrak{I}, P]$ and $\mathfrak{M}[\mathfrak{I}, P]$ are vector spaces. 
		\item [$(iv)$] If $P \in \mathcal{P}[\mathfrak{I}]$, the Streitberg $\Sigma[P]$ and the Lancaster $\Lambda[P]$ interactions are elements of $\mathfrak{M}_{n }[\mathfrak{I}, P]$.
		\item [$(v)$] If $\mu \in \mathfrak{M}[\mathfrak{I}]$, then $|\mu| \in \mathfrak{M}[\mathfrak{I}]$. In particular, a nonzero $\mu$ is marginally delimited by the probability $|\mu|/ |\mu|( \mathds{X}_{n}).$ 
	\end{enumerate}
\end{lemma}

\begin{proof} 
	Given a linear combination $\sum_{\pi} a_{\pi}P_{\pi}$, note that for every partition $\pi^{\prime} $ of $\{1, \ldots, n \}$ the measure $(\sum_{\pi} |a_{\pi}|P_{\pi})_{\pi^{\prime}}$ also is a linear combination $\sum_{\pi} b_{\pi}P_{\pi}$. Since for every subset $F$ of $\{1, \ldots, n\}$ and for every elements $x_{\vec{3}}, x_{\vec{4}} \in \mathds{X}_{n}$ 
	$$
	0\leq     \int_{\mathds{X}_{n}}\mathfrak{I}(x_{\vec{1}_{F} + \vec{3}_{F^{c}}},x_{\vec{4}})dP_{\pi}(x_{\vec{1}}) < \infty,
	$$
	and since $|\sum_{\pi} a_{\pi}P_{\pi}| \leq \sum_{\pi} |a_{\pi}|P_{\pi}$, we conclude that 
	\begin{align*}
		0 &\leq \int_{\mathds{X}_{n}}\mathfrak{I}(x_{\vec{1}_{F} + \vec{3}_{F^{c}}},x_{\vec{4}})d \left | \sum_{\pi} a_{\pi}P_{\pi}\right |_{\pi^{\prime}}(x_{\vec{1}})
		\leq  \int_{\mathds{X}_{n}}\mathfrak{I}(x_{\vec{1}_{F} + \vec{3}_{F^{c}}},x_{\vec{4}})d \left | \sum_{\pi} |a_{\pi}|P_{\pi}\right |_{\pi^{\prime}}(x_{\vec{1}})\\
		& \leq \sum_{\pi } b_{\pi}\int_{\mathds{X}_{n}}\mathfrak{I}(x_{\vec{1}_{F} + \vec{3}_{F^{c}}},x_{\vec{4}})dP_{\pi}(x_{\vec{1}}) < \infty.
	\end{align*}
	For relation $(ii)$, we obtain that $\mu$ is marginally delimited by $P$ by the inequality in the hypothesis. To prove that $\mu \in \mathcal{M}[\mathfrak{I}]$, suppose, without loss of generality, that $M\geq 1$. Let $\pi=\{L_{1}, \ldots, L_{|\pi|}\} $ be a partition of the set $\{1, \ldots, n\}$. By the hypothesis, for every $A_{i} \in \mathscr{B}(X_{i})$
	\begin{align*}
		M^{n}P_{\pi} (\prod_{i=1}^{n}A_{i})&=M^{n} \prod_{\ell =1}^{|\pi|}\left [P(\left [\prod_{i \in L_{\ell}}A_{i}\right ] \times \left [\prod_{i \notin L_{\ell}}X_{i} \right ] ) \right ]\\
		&\geq M^{n-|\pi|} \prod_{\ell =1}^{|\pi|}\left [|\mu|(\left [\prod_{i \in L_{\ell}}A_{i} \right ] \times \left [\prod_{i \notin L_{\ell}}X_{i} \right ] ) \right ] \geq |\mu|_{\pi}(\prod_{i=1}^{n}A_{i}).
	\end{align*}
	Thus, for every subset $F$ of $\{1, \ldots, n\}$ and for every elements $x_{\vec{3}}, x_{\vec{4}} \in \mathds{X}_{n}$ 
	$$
	\int_{\mathds{X}_{n}}\mathfrak{I}(x_{\vec{1}_{F} + \vec{3}_{F^{c}}},x_{\vec{4}})d|\mu|_{\pi}(x_{\vec{1}}) \leq M^{n}\int_{\mathds{X}_{n}}\mathfrak{I}(x_{\vec{1}_{F} + \vec{3}_{F^{c}}},x_{\vec{4}})dP_{\pi}(x_{\vec{1}})< \infty.
	$$
	To prove relation $(iii)$, note that $\mathfrak{M}[\mathfrak{I}, P]$ is a vector space by Lemma \ref{FIPmarvecspace} and relation $(iii)$ in Lemma \ref{probintegPDIn}, while $\mathfrak{M}_{n}[\mathfrak{I}, P]$ is just the intersection of two vector spaces.\\
	To conclude, relation $(v)$ is a direct consequence of the definition, as $|(|\mu|)|_{\pi} = |\mu|_{\pi}$, and the marginally delimited property occurs for the constant $M \geq |\mu|( \mathds{X}_{n})$.\end{proof}   

The following result is a version of the famous kernel mean embedding for PDI$_{n}$ kernels.  

\begin{theorem}\label{integPDIn} Let $\mathfrak{I}: \mathds{X}_{n} \times \mathds{X}_{n} \to \mathbb{R}$ be a continuous $n$-symmetric PDI$_{n}$ kernel that is zero at the extended diagonal $\Delta_{n-1}^{n}(\mathds{X}_{n})$. Then for any $\lambda \in  \mathfrak{M}_{n}[\mathfrak{I}]$ and for whichever $x_{\vec{0}} \in \mathds{X}_{n}$ is used to define $K^{\mathfrak{I}}$    
	\begin{align*}
		\int_{\mathds{X}_{n}}\int_{\mathds{X}_{n}}(-1)^{n}\mathfrak{I}(u, v)d\lambda(u)d\lambda(v)&= \int_{\mathds{X}_{n}}\int_{\mathds{X}_{n}}K^{\mathfrak{I}}(u, v)d\lambda(u)d \lambda(v)\\
		&= \langle K^{\mathfrak{I}}_{\lambda}, K^{\mathfrak{I}}_{\lambda} \rangle _{\mathcal{H}_{\mathfrak{I}}} \geq 0.
	\end{align*}    
	In particular, for any fixed $P \in \mathcal{P}[\mathfrak{I}]$ the bilinear function 
	$$
	(\lambda, \eta) \in \mathfrak{M}_{n}[\mathfrak{I}, P]\times \mathfrak{M}_{n}[\mathfrak{I}, P] \to \int_{\mathds{X}_{n}}\int_{\mathds{X}_{n}}(-1)^{n}\mathfrak{I}(u, v)d\lambda(u)d \eta(v) \in \mathbb{R} , 
	$$    
	is well defined and is a semi-inner product because
	\begin{align*}
		\int_{\mathds{X}_{n}}\int_{\mathds{X}_{n}}(-1)^{n}\mathfrak{I}(u, v)d\lambda(u)d\eta(v)&= \int_{\mathds{X}_{n}}\int_{\mathds{X}_{n}}K^{\mathfrak{I}}(u, v)d\lambda(u)d \eta(v)\\
		&= \langle K^{\mathfrak{I}}_{\lambda}, K^{\mathfrak{I}}_{\eta} \rangle _{\mathcal{H}_{\mathfrak{I}}},
	\end{align*}
	for whichever $x_{\vec{0}} \in \mathds{X}_{n}$ is used to define $K^{\mathfrak{I}}$.
\end{theorem}

\begin{proof} The first statement will be a direct consequence of the subsequent statement by relation $(v)$ in Lemma \ref{integPDIn0}.\\ 
	Let $x_{\vec{0}} \in \mathds{X}_{n}$ be arbitrary and consider the PD kernel $K^{\mathfrak{I}}$ related to it, whose explicit expression is given in Equation \eqref{newrepkI}.\\ 
	Since $\lambda, \eta \in \mathfrak{M}[ \mathfrak{I}, P]$ then $ |\lambda|, |\eta| \in \mathfrak{M}[ \mathfrak{I}]$ and these two measures are also marginally delimited by $P$. Thus, by relation $(iii)$ in Lemma \ref{integPDIn0} we obtain that $|\lambda | + |\eta| \in \mathfrak{M}[ \mathfrak{I},P]$. As a direct consequence of this fact, we obtain that all kernels that appear on the right hand side of Equation \eqref{newrepkI} are in $L^{1}((|\lambda| + |\eta|)\times (|\lambda| +|\eta|))$.\\ 
	By using Equation \eqref{integmu0n} in Equation \eqref{newrepkI}, we get that 
	$$
	\int_{\mathds{X}_{n}}\int_{\mathds{X}_{n}}(-1)^{n}\mathfrak{I}(u, v)d\lambda(u)d \eta(v)= \int_{\mathds{X}_{n}}\int_{\mathds{X}_{n}}K^{\mathfrak{I}}(u, v)d\lambda (u)d\eta(v),
	$$
	because with the exception of the term $\mathfrak{I}(x_{\vec{1}},x_{\vec{2}})$, all the other terms that appear on the right hand side of Equation \eqref{newrepkI} are zero.\\
	The third equality is a direct consequence of the kernel mean embedding in Theorem \ref{initialextmmddominio} because $2^{n}\mathfrak{I}(x_{\vec{1}}, x_{\vec{0}})=K^{\mathfrak{I}}(x_{\vec{1}}, x_{\vec{1}}) \in L^{1}(|\lambda|+|\eta|)$.
\end{proof}    

\begin{definition}\label{PDIn-Characteristic} Let $\mathfrak{I}: \mathds{X}_{n} \times \mathds{X}_{n} \to \mathbb{R}$ be a continuous $n$-symmetric PDI$_{n}$ kernel that is zero at the extended diagonal $\Delta_{n-1}^{n}(\mathds{X}_{n})$.    We say that $\mathfrak{I}$ is PDI$_{n}$-Characteristic if for every nonzero $\lambda \in  \mathfrak{M}_{n}[\mathfrak{I}]$
	\[
	\int_{\mathds{X}_{n}}\int_{\mathds{X}_{n}}(-1)^{n}\mathfrak{I}(u, v)d\lambda(u)d \lambda(v)>0.
	\]
	Equivalently, $\mathfrak{I}$ is PDI$_{n}$-Characteristic if for any fixed $P \in \mathcal{P}[\mathfrak{I}]$, the bilinear function 
	$$
	(\lambda, \eta) \in \mathfrak{M}_{n}[\mathfrak{I}, P]\times \mathfrak{M}_{n}[ \mathfrak{I}, P] \to \int_{\mathds{X}_{n}}\int_{\mathds{X}_{n}}(-1)^{n}\mathfrak{I}(u, v)d\lambda(u)d \eta(v), 
	$$    
	is an inner product.
\end{definition}

In particular, for a PDI$_{n}$-Characteristic kernel, we can determine whether the Streitberg $\Sigma[P]$ or the Lancaster $\Lambda[P]$ interactions are zero for any probability $P \in \mathcal{P}[\mathfrak{I}]$, by using the double integration of Theorem \ref{integPDIn}.     

The example of a PDI$_{n}$ kernel by taking a Kronecker product of $n$ CND kernels  has several additional properties, including others in Section \ref{KroneckerproductsofPDIkernels}.

\begin{corollary}\label{integmarginkroeprodndim2} Let $\gamma_{i}: X_{i} \times X_{i} \to \mathbb{R}$, $1\leq i \leq n$, be continuous CND metrizable kernels that are zero at the diagonal. The following assertions are equivalent for a measure $\mu \in \mathfrak{M}(\mathds{X}_{n})$ 
	\begin{enumerate}
		\item[$(i)$] $\prod_{i \in F}\gamma_{i} \in L^{1}(|\mu|\times |\mu|)$, for any $F \subset \{1, \ldots, n\}$.
		\item[$(ii)$] The functions $\prod_{i \in F}\gamma_{i}(\cdot, x_{i}) \in L^{1}(|\mu|)$ for any $F \subset \{1, \ldots, n\}$ and a fixed $x \in \mathds{X}_{n}$.
		\item[$(iii)$] The functions $\prod_{i \in F}\gamma_{i}(\cdot, x_{i}) \in L^{1}(|\mu|)$ for any $F \subset \{1, \ldots, n\}$ and for every $x \in \mathds{X}_{n}$.
		\item[$(iv)$] The measure $\mu$ satisfies the requirements of Lemma \ref{probintegPDIn}.
	\end{enumerate}
	Further, the set of measures that satisfies these relations is a vector space.\end{corollary}   

\begin{proof}     
	Suppose, without loss of generality, that $\mu$ is a probability that we denote as $P$. We prove that each relation in the statement of this corollary is equivalent to the same statement in Lemma \ref{probintegPDIn}. We focus on relation $(ii)$, as the others are proved similarly.\\
	Indeed, let $x=(x_{1}, \ldots, x_{n}) \in \mathds{X}_{n}$ and a probability $P$ that satisfies relation $(ii)$ of the corollary. Define $x_{\vec{3}}=x$ and let an arbitrary $x_{\vec{2}} \in \mathds{X}_{n}$. Since for every subset $F$ of $\{1,\ldots, n\}$ the function $\prod_{i \in F}\gamma_{i}(\cdot, x_{i}^{3}) \in L^{1}(P)$ (which is equivalent to $\prod_{i \in F}\gamma_{i}(\cdot, x_{i}^{3}) \in L^{1}(P_{F})$), we obtain that for every partition $\pi$ of $\{1, \ldots, n\}$ it holds that $\prod_{i \in F}\gamma_{i}(\cdot, x_{i}^{3}) \in L^{1}(P_{\pi})$, thus, relation $(ii)$ in the corollary implies relation $(ii)$ in Lemma \ref{probintegPDIn}, as the remaining term $\prod_{i \in F^{c}}\gamma_{i}(x_{i}^{2}, x_{i}^{3})$ is a constant. \\
	For the converse, let $P$ be a probability that satisfies relation $(ii)$ with a fixed element $x_{\vec{3}} \in \mathds{X}_{n}$ in Lemma \ref{probintegPDIn}. Then, by the hypothesis, for any subset $F$ of $\{1, \ldots, n\}$ consider the partition $\pi=\{F, F^{c}\}$ and an $x_{\vec{2}} \in \mathds{X}_{n}$ for which $ x_{i}^{2}\neq x_{i}^{3}$ for every $1\leq i \leq n$, thus
	$$
	\int_{\mathds{X}_{n}}[\times_{i=1}^{n}\gamma_{i}](x_{\vec{1}_{F} + \vec{2}_{F^{c}}},x_{\vec{3}})d[P_{F}\times P_{F^{c}}]d(x_{\vec{1}})< \infty.
	$$   
	But, for every $ i \in F^{c}$ the kernel $\gamma_{i}$ is metrizable, so $\gamma_{i}(x_{i}^{2}, x_{i}^{3}) \neq 0$, thus
	\begin{align*}
		\left [\prod_{i \in F^{c}} \gamma_{i}(x_{i}^{2}, x_{i}^{3}) \right ]&\int_{\mathds{X}_{n}}\prod_{i \in F} \gamma_{i}(x_{i}^{1}, x_{i}^{3})dP(x_{\vec{1}})\\
		&=    \int_{\mathds{X}_{n}}[\times_{i=1}^{n}\gamma_{i}](x_{\vec{1}_{F} + \vec{2}_{F^{c}}},x_{\vec{3}})d[P_{F}\times P_{F^{c}}]d(x_{\vec{1}})< \infty,
	\end{align*}
	which concludes the converse. The fact that the set of measures that satisfies these relations is a vector space is a direct consequence of relation $(iii)$.
\end{proof}

In particular, in the case of a Kronecker product of CND kernels, the concept of a PDI$_{n}$-Characteristic kernel is simpler. Instead of working with an infinite amount of inner products on different vector spaces, we may work with only one inner product on the vector space of measures that satisfies Corollary \ref{integmarginkroeprodndim2}, see Corollary \ref{lastkronprod}. 

Also, Corollary \ref{integmarginkroeprodndim2} still holds if we assume that the CND kernels $\gamma_{i}: X_{i} \times X_{i} \to \mathbb{R}$ are bounded at the diagonal instead of being zero at it. This occurs because if $M\geq |\gamma_{i}(x_{i}, x_{i})|$, for every $1\leq i \leq n$ and $x_{i}\in X_{i}$, then
$$
0 \leq |\gamma_{i}(x_{i}^{1}, x_{i}^{2}) -\gamma_{i}(x_{i}^{1}, x_{i}^{1})/2-\gamma_{i}(x_{i}^{2}, x_{i}^{2})/2 | \leq |\gamma_{i}(x_{i}^{1}, x_{i}^{2})| + M
$$
$$
0\leq |\gamma_{i}(x_{i}^{1}, x_{i}^{2})| \leq |\gamma_{i}(x_{i}^{1}, x_{i}^{2}) -\gamma_{i}(x_{i}^{1}, x_{i}^{1})/2-\gamma_{i}(x_{i}^{2}, x_{i}^{2})/2 | + M
$$   
so, for each of the $4$ relations, we may, without loss of generality, assume that each $\gamma_{i}$ is zero at the diagonal.

\section{Generalized Independence tests embedded in Hilbert spaces}\label{IndependencetestsembeddedinHilbertspaces}

In this section, we generalize the concept of a PDI$_{n}$ kernel on a set $\mathds{X}_{n}$, by introducing a parameter $k \in \{0, \ldots, n\}$, where the case $k=0$ are PD kernels in $\mathds{X}_{n}$, $k=1$ are the CND kernels in $\mathds{X}_{n}$ and $k=n$ are PDI$_{n}$ kernels in $\mathds{X}_{n}$ presented in Section \ref{Positivedefiniteindependentkernelsofordern}.

\begin{definition}\label{PDI2}Let $n \in \mathbb{N}$ and $0\leq k\leq n$. An $n$-symmetric kernel $\mathfrak{I}: \mathds{X}_{n} \times \mathds{X}_{n} \to \mathbb{R}$ is a positive definite independent kernel of order $k$ (PDI$_{k}$) in $\mathds{X}_{n}$ if, for every $\mu \in \mathcal{M}_{k}( \mathds{X}_{n})$, it satisfies
	$$
	\int_{\mathds{X}_{n} }    \int_{\mathds{X}_{n} }(-1)^{k}\mathfrak{I}(u,v)d\mu(u)d\mu(v) \geq 0.
	$$    
	If the previous inequality is an equality only when $\mu$ is the zero measure in $\mathcal{M}_{k}( \mathds{X}_{n})$, we say that $\mathfrak{I}$ is a strictly positive definite independent kernel of order $k$ (SPDI$_{k}$) in $\mathds{X}_{n}$.
\end{definition}

If $\mathfrak{I}$ is a PDI$_{k}$ kernel in $\mathds{X}_{n}$, then $(-1)^{k^{\prime} - k}\mathfrak{I}$ is a PDI$_{k^{\prime}}$ kernel in $\mathds{X}_{n}$ for any $k\leq k^{\prime} \leq n$ due to the inclusion in Equation \eqref{inclumeas}, similarly, this also holds for the strict case.

The class of PDI$_{1}$ kernels in $\mathds{X}_{n}$ is slightly different from the class of CND kernels in $\mathds{X}_{n}$, the former being more restrictive solely because we assume the $n$-symmetry of the kernel (the same issue occurs in the case $k=0$). % A technical solution to this issue for $k\geq 1$ can be done if we replace the requirement of $n$-symmetry in Definition \ref{PDI2} by the weak property that for any partition $\pi=\{F_{1}, \ldots, F_{k}\}$ of $\{1, \ldots, n\}$ with $|\pi| =k$, the induced kernel on $\prod_{i=1}^{k}Y_{i}$ is $k$-symmetric (as in Definition \ref{PDIn}), where $Y_{i}:= \prod_{j \in F_{i}}X_{j}$. Under those requirements, if $k\geq 2$ this definition is equivalent to $n$-symmetry and if $k=1$ is equivalent to standard symmetry as in the definition of a CND kernel.

%Unlike the case $k=n$, where an interesting class of examples for PDI$_{n}$ kernels in $\mathds{X}_{n}$ arises from taking the Kronecker product of $n$ CND kernels, the situation is more involved for other values of $k$. 

For  $ k-1 < a < k$, the following Bernstein-type equality can be proved by induction on the derivatives (see page 78  in \cite{Berg1984} and Lemma $4.1$ in \cite{Guella2022}) 
\[
t^{a}=(-1)^{k}\frac{(a)_{k}}{\Gamma(k-a)}\int_{(0,\infty)} (e^{-tr} - \omega_{k}(rt))\frac{1+r}{r^{k}} \left [\frac{r^{k-a-1}}{1+r}\right ]d(r),
\]	 
where $(a)_{k}:=(a-k+1)\ldots(a) $ is the standard Pochhammer symbol and $\omega_{k} (s):= \sum_{l=0}^{k-1}(-1)^{l}s^{l}/l!$.  From this relation we can obtain an interesting example of PDI$_{k}$ kernel, as if  $\gamma_{i}: X_{i} \times X_{i}\to \mathbb{R}$, $1\leq i \leq n$, are CND kernels that are zero on its diagonal, the kernel
\begin{equation}\label{exPDIK}
	(x_{\vec{1}}, x_{\vec{2}}) \to \left ( \sum_{i=1}^{n}\gamma_{i}(x_{i}^{1}, x_{i}^{2})  \right )^{a},
\end{equation}
is PDI$_{k}$, because for any $\mu \in \mathcal{M}(\mathds{X}_{n})$ it holds that
\begin{align*}
	&\int_{\mathds{X}_{n}}\int_{\mathds{X}_{n}}(-1)^{k}\left ( \sum_{i=1}^{n}\gamma_{i}(x_{i}^{1}, x_{i}^{2})  \right )^{a}d\mu(x_{\vec{1}})d\mu(x_{\vec{2}})\\
	&= \frac{(a)_{k}}{\Gamma(k-a)}\int_{(0,\infty)} \left [ \int_{\mathds{X}_{n}}\int_{\mathds{X}_{n}} e^{-\gamma_{1}(x_{1}^{1}, x_{1}^{2}) } \ldots e^{-\gamma_{n}(x_{n}^{1}, x_{n}^{2}) }  d\mu(x_{\vec{1}})d\mu(x_{\vec{2}}) \right ] \frac{1+r}{r^{k}} \left [\frac{r^{k-a-1}}{1+r}\right ]d(r) \geq 0, 
\end{align*}
where the double integrals  involving the terms of the polynomial $\omega$ vanished by Equation \eqref{integmu0n}.

For another example, let $1\leq k \leq |F| \leq n-1$ for some $F \subset \{1, \ldots, n\}$, and let $\mathfrak{L}: \mathds{X}_{F} \times \mathds{X}_{F} \to \mathbb{R}$ be an $|F|$-symmetric kernel that is PDI$_{k}$ in $\mathds{X}_{F}$. Then $\mathfrak{I}(x_{\vec{1}}, x_{\vec{2}}):= \mathfrak{L}(x_{\vec{1}_{F}}, x_{\vec{2}_{F}})$ is PDI$_{k}$ in $\mathds{X}_{n}$ because for every $\mu \in \mathcal{M}_{k}( \mathds{X}_{n})$, the measure 
$$
\mu_{F}\left(\prod_{i \in F} A_{i} \right) := \mu\left(\left[\prod_{i \in F} A_{i}\right] \times \left[\prod_{i \in F^{c}} X_{i} \right] \right),
$$
belongs to $\mathcal{M}_{k}( \mathds{X}_{F})$. Note, however, that this kernel is never SPDI$_{k}$. To see this, take an arbitrary nonzero measure $\eta \in \mathcal{M}_{n}( \mathds{X}_{n})$ and observe that the double integral of Definition \ref{PDI2} is always zero with respect to this measure. More generally, even a positive linear combination of all possible $k \leq |F| \leq n-1$, where $F \subset \{1, \ldots, n\}$, is PDI$_{k}$ but not an SPDI$_{k}$. In the notation of the previous example, we can obtain that 
$$
(x_{\vec{1}}, x_{\vec{2}}) \to \left ( \sum_{i=1}^{n}\gamma_{i}(x_{i}^{1}, x_{i}^{2})  \right )^{k},
$$
is PDI$_{k}$ but not SPDI$_{k}$. A characterization of when an arbitrary Kronecker product of kernels is SPDI$_{k}$ is presented in Section \ref{KroneckerproductsofPDIkernels}. 

We generalize Lemma \ref{PDInsimpli} to PDI$_{k}$ kernels from a different perspective, using the measure $\mu_{k}^{n}$ defined in Equation \eqref{measorderk}.

\begin{lemma}\label{PDI2simpli} Let $\mathfrak{I}: \mathds{X}_{n} \times \mathds{X}_{n} \to \mathbb{R}$ be an $n$-symmetric kernel. Consider the kernel $\mathfrak{I}^{\prime}: \mathds{X}_{n} \times \mathds{X}_{n} \to \mathbb{R}$ 
	\begin{align*}
		\mathfrak{I}^{\prime} (x_{\vec{1}}, x_{\vec{2}} ):= \frac{1}{2}\int_{\mathds{X}_{n}}    \mathfrak{I}(x_{\vec{1}}, y )d     \mu^{n}_{k}[x_{\vec{2}}, x_{\vec{1}}] (y) + \frac{1}{2}\int_{\mathds{X}_{n}}    \mathfrak{I}(x_{\vec{2}}, y )d    \mu^{n}_{k}[x_{\vec{1}}, x_{\vec{2}}](y). 
	\end{align*}
	Then, for any $\mu \in \mathcal{M}_{k}( \mathds{X}_{n})$ 
	$$ 
	\int_{\mathds{X}_{n}}\int_{\mathds{X}_{n}}(-1)^{k}\mathfrak{I}^{\prime} (u, v )d\mu(u)d\mu(v)= \int_{\mathds{X}_{n}}\int_{\mathds{X}_{n}}(-1)^{k}\mathfrak{I} (u, v )d\mu(u)d\mu(v)
	$$
	If at least $n-k+1$ coordinates of $x_{\vec{1}}$ and $x_{\vec{2}}$ are equal, then $\mathfrak{I}^{\prime}(x_{\vec{1}}, x_{\vec{2}})=0$.\\
	If $\mathfrak{I}(x_{\vec{1}}, x_{\vec{2}})=0$ whenever at least $n-k+1$ coordinates of $x_{\vec{1}}$ and $x_{\vec{2}}$ are equal, then $\mathfrak{I}= \mathfrak{I}^{\prime}$.
\end{lemma}

\begin{proof}
	Due to the definition of $\mu^{n}_{k}[x_{\vec{2}}, x_{\vec{1}}]$, for every fixed $x_{\vec{1}} \in \mathds{X}_{n}$, the function 
	$$
	x_{\vec{2}} \in \mathds{X}_{n} \to \int_{\mathds{X}_{n}}    \mathfrak{I}(x_{\vec{1}}, y )d    [ \mu^{n}_{k}[x_{\vec{2}}, x_{\vec{1}}]-\delta_{x_{\vec{2}}}](y) \in \mathbb{R}
	$$
	is a linear combination of functions that only depend on at most $k-1$ of its $n$ variables; hence, due to Equation \eqref{integmu0n}, for every $\mu \in \mathcal{M}_{k}( \mathds{X}_{n})$ 
	$$
	\int_{\mathds{X}_{n}}\left [ \int_{\mathds{X}_{n}}\mathfrak{I}(x_{\vec{1}}, y )d    [ \mu^{n}_{k}[x_{\vec{2}}, x_{\vec{1}}] \right ]d\mu(x_{\vec{2}})=     \int_{\mathds{X}_{n}}\mathfrak{I}(x_{\vec{1}}, x_{\vec{2}} ) d\mu(x_{\vec{2}}),
	$$
	and similarly for $\mu^{n}_{k}[x_{\vec{1}}, x_{\vec{2}}] $. Then 
	$$ 
	\int_{\mathds{X}_{n}}\int_{\mathds{X}_{n}}(-1)^{k}\mathfrak{I}^{\prime} (u, v )d\mu(u)d\mu(v)=\int_{\mathds{X}_{n}}\int_{\mathds{X}_{n}}(-1)^{k}\mathfrak{I} (u, v )d\mu(u)d\mu(v).
	$$
	If at least $n-k+1$ coordinates of $x_{\vec{1}}$ and $x_{\vec{2}}$ are equal, then by the comment after Equation \eqref{measorderk}, both $\mu^{n}_{k}[x_{\vec{2}}, x_{\vec{1}}]$ and $\mu^{n}_{k}[x_{\vec{1}}, x_{\vec{2}}]$ are the zero measure.\\
	If $\mathfrak{I}(x_{\vec{1}}, x_{\vec{2}})=0$ whenever at least $n-k+1$ coordinates of $x_{\vec{1}}$ and $x_{\vec{2}}$ are equal, then 
	$$
	\int_{\mathds{X}_{n}}\mathfrak{I}(x_{\vec{1}}, y )d    [ \mu^{n}_{k}[x_{\vec{2}}, x_{\vec{1}}]-\delta_{x_{\vec{2}}}](y)=\sum_{j=0}^{k-1}(-1)^{k-j}\binom{n-j-1}{n-k}\sum_{|F|=j}\mathfrak{I}(x_{\vec{1}},x_{\vec{1} + \vec{1}_{F}} )=0 
	$$
	because all terms $\mathfrak{I}(x_{\vec{1}},x_{\vec{1} + \vec{1}_{F}} )$ are zero, and thus $\mathfrak{I}= \mathfrak{I}^{\prime}$.
\end{proof}    

\begin{remark}\label{pdiksymmetryissues}\begin{enumerate}
		\item We cannot guarantee that the kernel $\mathfrak{I}^{\prime}$ defined in Lemma \ref{PDI2simpli} is $n$-symmetric except when $k=n$. 
		\item We cannot guarantee that the kernel 
		$$
		(x_{\vec{1}}, x_{\vec{2}} ) \in \mathds{X}_{n}\times \mathds{X}_{n} \to \int_{\mathds{X}_{n}}\int_{\mathds{X}_{n}}    \mathfrak{I}(u, v )d     \mu^{n}_{k}[x_{\vec{2}}, x_{\vec{1}}] (u)d     \mu^{n}_{k}[x_{\vec{2}}, x_{\vec{1}}] (v) \in \mathbb{R}
		$$
		is PDI$_{k}$ except when $k=1$ or $n=k$. The reason is that if we analyze the explicit expression for this kernel, some of its terms depend on more than $k$ variables of $x_{\vec{1}}$ and of $x_{\vec{2}}$ simultaneously.    
		\item  The difference between the definition of $\mathfrak{I}^{\prime}$ in Lemma \ref{PDInsimpli} (for the case $n=k$) and the one in Lemma \ref{PDI2simpli} is a kernel  in $\mathds{X}_{n}$ that is zero whenever at least one of the coordinates of $x_{\vec{1}}$ and $x_{\vec{2}}$ is equal, and whose double integral against any $\mu\in\mathcal{M}_k(\mathds{X}_n)$ is zero.
	\end{enumerate}
\end{remark}

Now, we present another symmetry property that will lead to better behavior of the kernel presented in Lemma \ref{PDI2simpli}, taking into account Remark \ref{pdiksymmetryissues}.

\begin{definition}\label{completensymmetry} An $n$-symmetric kernel $\mathfrak{I}: \mathds{X}_{n} \times \mathds{X}_{n} \to \mathbb{R}$ is called complete $n$-symmetric if for any $F \subset \{1, \dots,n\}$ and $x_{\vec{1}}, x_{\vec{2}}, x_{\vec{3}}, x_{\vec{4}} \in \mathds{X}_{n} $
	$$
	\mathfrak{I} (x_{\vec{1}_{F} + \vec{3}_{F^{c}}}, x_{\vec{2}_{F} + \vec{3}_{F^{c}}} )=\mathfrak{I} (x_{\vec{1}_{F} + \vec{4}_{F^{c}}}, x_{\vec{2}_{F} + \vec{4}_{F^{c}}} ).
	$$    
\end{definition}

For $n=3$, this implies $\mathfrak{I}((x_1, y_1, z), (x_2, y_2, z)) = \mathfrak{I}((x_1, y_1, w), (x_2, y_2, w))$ for any $z,w \in X_{3}$, and so forth.

As an example, consider $g:[0, \infty)^{n} \to \mathbb{R}$ and symmetric kernels $ \gamma_{i}: X_{i}\times X_{i} \to [0, \infty)$, $1\leq i \leq n$ that are constant on the diagonal. Then $g(\gamma_{1}, \ldots, \gamma_{n})$ is a complete $n$-symmetric kernel in $\mathds{X}_{n}$. The radial PDI$_{k}$ kernels presented in Section \ref{Bernsteinfunctionsofordern} are complete $n$-symmetric, as are the PDI$_{n}$ kernels that are zero at the extended diagonal $\Delta_{n-1}^{n}$. 

\begin{corollary}\label{complnsimzerodiag} If the kernel $\mathfrak{I}: \mathds{X}_{n} \times \mathds{X}_{n} \to \mathbb{R}$ is complete $n$-symmetric, then the kernel $\mathfrak{I}^{\prime}$ defined in Lemma \ref{PDI2simpli} is also complete $n$-symmetric, and in particular, $\mathfrak{I}$ is PDI$_{k}$ if and only if $\mathfrak{I}^{\prime}$ is PDI$_{k}$.\\ Further, if $k=n$, the kernel $\mathfrak{I}^{\prime}$ in Lemma \ref{PDInsimpli} and the one in Lemma \ref{PDI2simpli} are the same. 
\end{corollary}

\begin{proof} The explicit expression for $\mathfrak{I}^{\prime}$ is
	\begin{align*}
		\mathfrak{I}^{\prime}& (x_{\vec{1}}, x_{\vec{2}} )= \mathfrak{I}(x_{\vec{1}}, x_{\vec{2}} )\\
		& + \frac{1}{2} \sum_{j=0}^{k-1}(-1)^{k-j}\binom{n-j-1}{n-k}\sum_{|F|=j}\left [\mathfrak{I} (x_{\vec{1}}, x_{\vec{1}_{F^{c}}+ \vec{2}_{F}} ) +\mathfrak{I} (x_{\vec{2}}, x_{\vec{2}_{F^{c}}+ \vec{1}_{F}} ) \right ].
	\end{align*}
	First, we prove that $    \mathfrak{I}^{\prime}$ is $n$-symmetric. Indeed, by the previous expression
	\begin{align*}
		&\mathfrak{I}^{\prime} (x_{ \alpha}, x_{\vec{3}-\alpha} )= \mathfrak{I}(x_{ \alpha}, x_{\vec{3}-\alpha})\\
		& + \frac{1}{2} \sum_{j=0}^{k-1}(-1)^{k-j}\binom{n-j-1}{n-k}\sum_{|F|=j}\left [\mathfrak{I} (x_{\alpha}, x_{\alpha_{F^{c}}+ (\vec{3} -\alpha)_{F}} ) +\mathfrak{I} (x_{ \vec{3} -\alpha}, x_{(\vec{3} -\alpha)_{F^{c}}+ \alpha_{F}} ) \right ].   
	\end{align*}
	By the $n$-symmetry of $\mathfrak{I}$, we have that $\mathfrak{I}(x_{ \alpha}, x_{\vec{3}-\alpha}) = \mathfrak{I}(x_{\vec{1}}, x_{\vec{2}} )$ and, by changing the coordinates of $F$:
	\begin{align*}
		&\mathfrak{I} (x_{\alpha}, x_{\alpha_{F^{c}}+ (\vec{3} -\alpha)_{F}} )= \mathfrak{I} (x_{\vec{1}_{F} + \alpha_{F^{c}}}, x_{ \vec{2}_{F} +\alpha_{F^{c}}} ),\\
		& \mathfrak{I} (x_{ \vec{3} -\alpha}, x_{(\vec{3} -\alpha)_{F^{c}}+ \alpha_{F}} )=\mathfrak{I} (x_{\vec{1}_{F}+ (\vec{3} -\alpha)_{F^{c}}}, x_{\vec{2}_{F} +(\vec{3} -\alpha)_{F^{c}}} ). 
	\end{align*}
	By the complete $n$-symmetry property, we have that 
	$$
	\mathfrak{I} (x_{\vec{1}_{F} + \alpha_{F^{c}}}, x_{ \vec{2}_{F} +\alpha_{F^{c}}} )=\mathfrak{I} (x_{\vec{1}}, x_{ \vec{2}_{F} +\vec{1}_{F^{c}}} ), \quad \mathfrak{I} (x_{\vec{1}_{F}+ (\vec{3} -\alpha)_{F^{c}}}, x_{\vec{2}_{F} +(\vec{3} -\alpha)_{F^{c}}} ) = \mathfrak{I} (x_{\vec{2}}, x_{\vec{2}_{F^{c}}+ \vec{1}_{F}} ),
	$$
	hence $    \mathfrak{I}^{\prime} (x_{ \alpha}, x_{\vec{3}-\alpha} )= \mathfrak{I}^{\prime} (x_{\vec{1}}, x_{\vec{2}} )$. Now, we prove that $\mathfrak{I}^{\prime} $ is complete $n$-symmetric. For arbitrary $L \subset \{1, \dots,n\}$ and $x_{\vec{1}}, x_{\vec{2}}, x_{\vec{3}}, x_{\vec{4}} \in \mathds{X}_{n} $, we have that
	\begin{align*}
		\mathfrak{I}^{\prime}& (x_{\vec{1}_{L} + \vec{3}_{L^{c}}}, x_{\vec{2}_{L} + \vec{3}_{L^{c}}} )= \mathfrak{I}(x_{\vec{1}_{L} + \vec{3}_{L^{c}}}, x_{\vec{2}_{L} + \vec{3}_{L^{c}}} )\\
		& + \frac{1}{2} \sum_{j=0}^{k-1}(-1)^{k-j}\binom{n-j-1}{n-k}\sum_{|F|=j} \mathfrak{I} (x_{\vec{1}_{L} + \vec{3}_{L^{c}}}, x_{\vec{1}_{L\cap F^{c}}+\vec{3}_{L^{c}\cap F^{c} } + \vec{2}_{L\cap F} +\vec{3}_{L^{c}\cap F}} )\\
		& +\frac{1}{2} \sum_{j=0}^{k-1}(-1)^{k-j}\binom{n-j-1}{n-k}\sum_{|F|=j}\mathfrak{I} (x_{\vec{2}_{L} + \vec{3}_{L^{c}}},x_{\vec{2}_{L\cap F^{c}}+\vec{3}_{L^{c}\cap F^{c} } + \vec{1}_{L\cap F} +\vec{3}_{L^{c}\cap F}}).
	\end{align*} 
	Since $\mathfrak{I}(x_{\vec{1}_{L} + \vec{3}_{L^{c}}}, x_{\vec{2}_{L} + \vec{3}_{L^{c}}} )= \mathfrak{I}(x_{\vec{1}_{L} + \vec{4}_{L^{c}}}, x_{\vec{2}_{L} + \vec{4}_{L^{c}}} )$ and
	$$
	\mathfrak{I} (x_{\vec{1}_{L} + \vec{3}_{L^{c}}}, x_{\vec{1}_{L\cap F^{c}}+\vec{3}_{L^{c}\cap F^{c} } + \vec{2}_{L\cap F} +\vec{3}_{L^{c}\cap F}} )= \mathfrak{I} (x_{\vec{1}_{L} + \vec{4}_{L^{c}}}, x_{\vec{1}_{L\cap F^{c}}+\vec{4}_{L^{c}\cap F^{c} } + \vec{2}_{L\cap F} +\vec{4}_{L^{c}\cap F}} ),
	$$
	and similarly for the other term, we obtain that $\mathfrak{I}^{\prime} (x_{\vec{1}_{L} + \vec{3}_{L^{c}}}, x_{\vec{2}_{L} + \vec{3}_{L^{c}}} )=\mathfrak{I}^{\prime} (x_{\vec{1}_{F} + \vec{4}_{F^{c}}}, x_{\vec{2}_{F} + \vec{4}_{F^{c}}} )$. \\
	To conclude, suppose that $k=n$. By Equation \eqref{simplidoublen} and the hypothesis of complete $n$-symmetry, the kernel $\mathfrak{I}^{\prime}$ in Lemma \ref{PDInsimpli} can be written as
	\begin{align*}
		\mathfrak{I}^{\prime}(x_{\vec{1}}, x_{\vec{2}})&= \sum_{|F|=0}^{n} \sum_{\xi\in \mathbb{N}_{2}^{n-|F|}} (-1)^{n-|F|}2^{|F|-n} \mathfrak{I}(x_{\vec{1}_{F} + \xi_{F^{c}}},x_{\vec{2}_{F} + \xi_{F^{c}}} )\\
		&=\sum_{|F|=0}^{n} \sum_{\xi \in \mathbb{N}_{2}^{n-|F|}} (-1)^{n-|F|}2^{|F|-n} \left [ \frac{1}{2}\mathfrak{I}(x_{\vec{1} },x_{\vec{2}_{F} + \vec{1}_{F^{c}}} ) + \frac{1}{2}\mathfrak{I}(x_{\vec{1}_{F} + \vec{2}_{F^{c}}},x_{\vec{2}} ) \right ]\\
		&=\sum_{|F|=0}^{n} (-1)^{n-|F|}\left [ \frac{1}{2}\mathfrak{I}(x_{\vec{1} },x_{\vec{2}_{F} + \vec{1}_{F^{c}}} ) + \frac{1}{2}\mathfrak{I}(x_{\vec{1}_{F} + \vec{2}_{F^{c}}},x_{\vec{2}} ) \right ],
	\end{align*}
	and the latter is the kernel $    \mathfrak{I}^{\prime}$ defined in Lemma \ref{PDI2simpli}.
\end{proof}

Note that a PDI$_{n}$ kernel that is zero at the extended diagonal $\Delta_{n-1}^{n}$ is necessarily complete $n$-symmetric. 

In essence, Lemma \ref{PDI2simpli} states that it is convenient to assume that a complete $n$-symmetric PDI$_{k}$ kernel $\mathfrak{I}: \mathds{X}_{n} \times \mathds{X}_{n} \to \mathbb{R}$ is zero on the extended diagonal 
\begin{equation}\label{exdia1def}
	\Delta_{k-1}^{n} :=\{(x_{\vec{1}}, x_{\vec{2}}) \in \mathds{X}_{n} \times \mathds{X}_{n}, \quad |\{ i, \quad x_{i}^{1}= x_{i}^{2}\}| \geq n-k+1 \},
\end{equation}        
where, if $k=0$, then $\Delta_{-1}^{n}= \emptyset$ and $\Delta_{0}^{n}$ is the standard diagonal. They satisfy the following inclusion relations:
$$
\Delta_{-1}^{n} \subset \Delta_{0}^{n} \subset \ldots \subset \Delta_{n-1}^{n}.
$$
From now on, we will assume the restriction that $\mathfrak{I}$ is complete $n-$symmetric according to its necessity.

From a PDI$_{k}$ kernel, we can obtain several other kernels of order $a \leq k$ by fixing a few coordinates, in a manner similar to Lemma \ref{fixedkern}.

\begin{lemma}\label{projectionkernels}    Let $\mathfrak{I}: \mathds{X}_{n} \times \mathds{X}_{n} \to \mathbb{R}$ be an $n$-symmetric PDI$_{k}$ kernel. If $F \subset \{1, \ldots, n \}$, $|F| \geq 1$, the kernel $\mathfrak{I}_{\lambda} : \mathds{X}_{F} \times \mathds{X}_{F} \to \mathbb{R} $, defined as
	\begin{equation}
		\mathfrak{I}_{\lambda}(x_{\vec{1}_{F}}, x_{\vec{2}_{F}}):= \int_{\mathds{X}_{F^{c}}}\int_{\mathds{X}_{F^{c}}}(-1)^{k - a}\mathfrak{I}((x_{\vec{1}_{F}},u ) ),(x_{\vec{2}_{F}}, v) )d\lambda(u)d\lambda(v),
	\end{equation}
	is PDI$_{a}$ in $\mathds{X}_{F}$ for any $ \lambda \in \mathcal{M}_{b}(\mathds{X}_{F^{c}})$ such that $a+b \geq k$ and with the set size  restrictions. Further, if $\mathfrak{I}$ is complete $n$-symmetric, then $\mathfrak{I}_{\lambda}$ is complete $|F|$-symmetric.
\end{lemma}

\begin{proof}Indeed, by using in order  Lemma   \ref{inclusionmeasures2indexkroen} , Lemma \ref{inclusionmeasures2index} and Equation \eqref{inclumeas}, we have that if $0\leq a \leq |F|$ and $0\leq b \leq n-|F|$ then
	$$
	\mathcal{M}_{a}(\mathds{X}_{F }) \times \mathcal{M}_{b}(\mathds{X}_{F^{c}}) \subset 	\mathcal{M}_{a,b}(\mathds{X}_{F }, \mathds{X}_{F^{c}}) \subset     \mathcal{M}_{a+b}(\mathds{X}_{n }) \subset \mathcal{M}_{k}(\mathds{X}_{n }),
	$$
	and the conclusion follows directly. The complete $|F|$-symmetry is immediate and the proof is omitted. 
\end{proof}

How the  products of vector spaces of measures work, and how they interact with Kronecker products of PDI kernels, is presented in Section \ref{KroneckerproductsofPDIkernels}. 

In order to obtain a geometrical interpretation of PDI$_{k}$ kernels, we connect them to PD kernels in a similar way as Lemma \ref{PDIntoPDn} using the measure $\mu_{k}^{n}[x_{\vec{1}}, x_{\vec{2}}]$ defined in Equation \eqref{measorderk}.

\begin{lemma}\label{pdi2topdn} Let $\mathfrak{I} : \mathds{X}_{n} \times \mathds{X}_{n} \to \mathbb{R} $ be an $n$-symmetric kernel and a fixed $x_{\vec{0}} \in \mathds{X}_{n} $. The kernel $ K^{\mathfrak{I}} : \mathds{X}_{n} \times \mathds{X}_{n} \to \mathbb{R} $ defined as
	$$
	K^{\mathfrak{I}}(x_{\vec{1}}, x_{\vec{2}}):= \int_{\mathds{X}_{n} }\int_{\mathds{X}_{n} }(-1)^{k}\mathfrak{I}(u,v)d\mu^{n}_{k}[x_{\vec{1}}, x_{\vec{0}}](u)d\mu^{n}_{k}[x_{\vec{2}}, x_{\vec{0}}](v) 
	$$
	is PD if and only if $\mathfrak{I}$ is PDI$_{k}$. Further, for every $\eta \in \mathcal{M}_{k}( \mathds{X}_{n})$, we have that
	$$
	\int_{\mathds{X}_{n} }\int_{\mathds{X}_{n}}    K^{\mathfrak{I}}(u, v)d\eta(u)d\eta(v) = \int_{\mathds{X}_{n} }\int_{\mathds{X}_{n} }(-1)^{k}\mathfrak{I}(u,v)d\eta(u)d\eta(v).
	$$
\end{lemma}

\begin{proof} Suppose that $\mathfrak{I}$ is PDI$_{k}$, then for arbitrary points $z_{1}, \ldots, z_{m} \in \mathds{X}_{n} $ and scalars $d_{1}, \ldots, d_{m} \in \mathbb{R}$ 
	\begin{align*}
		& \sum_{i,j=1}^{m}d_{i}d_{j}K^{\mathfrak{I}}(z_{i}, z_{j})\\
		&= \int_{\mathds{X}_{n} }\int_{\mathds{X}_{n} }(-1)^{k}\mathfrak{I}(u,v)d\left [\sum_{i=1}^{m}d_{i}\mu^{n}_{k}[z_{i}, x_{\vec{0}} ] \right ](u)d\left [\sum_{j=1}^{m}d_{j}\mu^{n}_{k}[z_{j}, x_{\vec{0}} ]\right ](v) \geq 0,
	\end{align*}
	because $\mathcal{M}_{k}( \mathds{X}_{n}) $ is a vector space.\\
	Conversely, if $K^{\mathfrak{I}}$ is PD, let $x^{1}_{i}, \dots, x^{m}_{i} \in X_{i}$, $1\leq i \leq n$ and scalars $c_{\alpha} \in \mathbb{R}$, $\alpha \in \mathbb{N}^{n}_{m}$ such that $\sum_{ \alpha \in \mathbb{N}_{m}^{n}}c_{\alpha}\delta_{x_{\alpha}} \in \mathcal{M}_{k}( \mathds{X}_{n})$, then 
	\begin{align*}
		0 &\leq \sum_{\alpha, \beta \in \mathbb{N}_{m}^{n}}c_{\alpha}c_{\beta}K^{\mathfrak{I}}(x_{\alpha},x_{\beta})\\
		&= \int_{\mathds{X}_{n} }\int_{\mathds{X}_{n} }(-1)^{k}\mathfrak{I}(u,v)d\left [\sum_{\alpha \in \mathbb{N}_{m}^{n}}c_{\alpha}\mu^{n}_{k}[x_{\alpha}, x_{\vec{0}}] \right ](u)d\left [\sum_{\beta \in \mathbb{N}_{m}^{n}}c_{\beta}\mu^{n}_{k}[x_{\beta}, x_{\vec{0}}] \right ](v).
	\end{align*}
	However
	$$
	\sum_{\alpha \in \mathbb{N}_{m}^{n}}c_{\alpha}\mu^{n}_{k}[x_{\alpha}, x_{\vec{0}}]=\sum_{\alpha \in \mathbb{N}_{m}^{n}}c_{\alpha} \delta_{x_{\alpha}}, 
	$$
	because for any function $f:\mathds{X}_{n} \to \mathbb{R} $
	\begin{align*}
		\int_{\mathds{X}_{n}}f(u)&d\left [\sum_{\alpha \in \mathbb{N}_{m}^{n}}c_{\alpha} \mu^{n}_{k}[x_{\alpha}, x_{\vec{0}}]\right ](u)\\
		&= \sum_{\alpha \in \mathbb{N}_{m}^{n}}c_{\alpha} \left [ f(x_{\alpha}) + \sum_{j=0}^{k-1}(-1)^{k-j}\binom{n-j-1}{n-k}\sum_{|F|=j}f(x_{\alpha_{F} }) \right] \\
		&= \sum_{\alpha \in \mathbb{N}_{m}^{n}}c_{\alpha} f(x_{\alpha}) = \int_{\mathds{X}_{n}}f(u)d\left [\sum_{\alpha \in \mathbb{N}_{m}^{n}}c_{\alpha}\delta_{x_{\alpha}} \right ](u),
	\end{align*}
	which again occurs by Equation \eqref{integmu0n}, because the functions 
	$$
	(u_{1}, \ldots, u_{n}) \in \mathds{X}_{n} \to f((u_{F}, x_{\vec{0}_{F^{c}}})) \in \mathbb{R}, \quad 0\leq |F| \leq k-1
	$$
	only depend on $|F|\leq k-1$ among the $n$ variables and
	$$
	\int_{\mathds{X}_{n}}f((u_{F}, x_{\vec{0}_{F^{c}}}))d\left [ \sum_{\alpha \in \mathbb{N}_{m}^{n}}c_{\alpha} \delta_{x_{\alpha}} \right ](u)= \sum_{\alpha \in \mathbb{N}_{m}^{n}}c_{\alpha} f(x_{\alpha_{F} }).
	$$\end{proof}   

A direct consequence of the proof of Lemma \ref{pdi2topdn} is that for every $\eta \in \mathcal{M}_{k}( \mathds{X}_{n})$ 
\begin{equation}\label{contaPDIkntoPD}
	\begin{split}
		\int_{\mathds{X}_{n}}\int_{\mathds{X}_{n}} (-1)^{k}\mathfrak{I}(x_{\vec{1} },x_{\vec{2} } ) d\lambda(x_{\vec{1} })d\eta(x_{\vec{2} }) &=     \int_{\mathds{X}_{n}}\int_{\mathds{X}_{n}} K^{\mathfrak{I}}(x_{\vec{1} },x_{\vec{2} } ) d\lambda(x_{\vec{1} })d\eta(x_{\vec{2} })\\
		&=\langle K^{\mathfrak{I}}_{\lambda}, K^{\mathfrak{I}}_{\eta} \rangle_{K^{\mathfrak{I}}}. 
	\end{split}
\end{equation}

Inspired by the distance covariance generalization done in Section 6 in \cite{guella2023}, we can prove a stronger property than the one in Lemma \ref{pdi2topdn} with a very similar argument, as the kernel 
$$
K^{\mathfrak{I}}((x_{\vec{1}}, x_{\vec{3}}), (x_{\vec{2}}, x_{\vec{4}})):= \int_{\mathds{X}_{n} }\int_{\mathds{X}_{n} }(-1)^{k}\mathfrak{I}(u,v)d\mu^{n}_{k}[x_{\vec{1}}, x_{\vec{3}}](u)\mu^{n}_{k}[x_{\vec{2}}, x_{\vec{4}}](v) 
$$
is PD in $\mathds{X}_{n} \times \mathds{X}_{n}$ if and only if $\mathfrak{I}$ is PDI$_{k}$ in $\mathds{X}_{n}$. 

The explicit expression for $    K^{\mathfrak{I}}$ in Lemma \ref{pdi2topdn} is 

\begin{equation}\label{expleqkIk}
	\begin{split}
		K^{\mathfrak{I}}(x_{\vec{1}}, x_{\vec{2}}):=& (-1)^{k}\mathfrak{I}(x_{\vec{1}}, x_{\vec{2}})+ \sum_{i=0}^{k-1}(-1)^{i}\binom{n-i-1}{n-k}\sum_{|F|=i} \mathfrak{I}(x_{\vec{1}_{F}}, x_{\vec{2}})\\
		+&\sum_{j=0}^{k-1}(-1)^{j} \binom{n-j-1}{n-k} \sum_{|\mathcal{F}|=j} \mathfrak{I}(x_{\vec{1}}, x_{\vec{2}_{\mathcal{F}}})\\
		+&\sum_{i,j=0}^{k-1}(-1)^{k+i+j}\binom{n-i-1}{n-k}\binom{n-j-1}{n-k}\sum_{|F|=i}\sum_{|\mathcal{F}|=j} \mathfrak{I}(x_{\vec{1}_{F}}, x_{\vec{2}_{\mathcal{F}}}).
	\end{split}
\end{equation}

Unfortunately, the geometrical interpretation for PDI$_{k}$ kernels defined in $\mathds{X}_{n}$ by using the RKHS of the related positive definite kernel $    K^{\mathfrak{I}}$ (in a similar way as Theorem \ref{PDIngeometryrkhs} or Equation \eqref{geogamma}) gets more complicated as the codimension $n-k$ increases. Our results are presented in Appendix \ref{geoint}. 

Unlike Equation  \eqref{newrepkI2}, the value of $K^{\mathfrak{I}}$ in Equation \eqref{expleqkIk} does not have a simple expression at the diagonal of this kernel. However, we can present an expression for a specific scenario that is enough to obtain the integrability restrictions equivalences in Corollary \ref{integrestrik}. First, we prove another representation for the measure $\mu_{k}^{n}$.

\begin{lemma}\label{muknsimplified}
	For every $1\leq k \leq n$ and $x_{\vec{1}},x_{\vec{2}}\in \mathds{X}^{n}$, we have
	$$
	\mu_{k}^{n}[x_{\vec{1}},x_{\vec{2}}] = \sum_{|F|\geq k }\mu_{|F|}^{|F|}[x_{\vec{1}_{F}},x_{\vec{2}_{F}}]
	\times \delta(x_{\vec{2}_{F^{c}}}) .
	$$
\end{lemma}

\begin{proof}
	Indeed,
	$$
	\mu_{|F|}^{|F|}[x_{\vec{1}_{F}},x_{\vec{2}_{F}}]
	\times \delta(x_{\vec{2}_{F^{c}}})= \left [ \bigtimes_{i\in F}(\delta_{x^{1}_{i}}-\delta_{x^{2}_{i}}) \right ]\times \delta(x_{\vec{2}_{F^{c}}})=\sum_{G\subset F} (-1)^{|F|-|G|}\,
	\delta_{x_{\vec{1}_{G} +\vec{2}_{G^{c}}}}.
	$$
	Thus 
	\begin{align*}
		\sum_{|F|\geq k }\mu_{|F|}^{|F|}[x_{\vec{1}_{F}},x_{\vec{2}_{F}}]
		\times \delta(x_{\vec{2}_{F^{c}}}) &= \sum_{|F|\geq k } \left [\sum_{G\subset F} (-1)^{|F|-|G|} \delta_{x_{\vec{1}_{G} +\vec{2}_{G^{c}}}} \right ]\\
		&= \sum_{|G|=0}^{n}\left [\sum_{F }^{G\subset F, |F|\geq k} (-1)^{|F|-|G| } \right ]\delta_{x_{\vec{1}_{G} +\vec{2}_{G^{c}}}}.      
	\end{align*}
	To obtain this number, if $|G|=\ell$, then
	\begin{align*}
		\sum_{F }^{G\subset F, |F|\geq k} (-1)^{|F|-\ell }&= \sum_{j=\max{\{k,\ell\}}}^{n} (-1)^{j-\ell }\sum_{|F|= j, G\subset F }1=\sum_{j=\max{\{k,\ell\}}}^{n}(-1)^{j-\ell } \binom{n-\ell}{j-\ell}.   
	\end{align*}
	When $\ell=n$, this sum is equal to $1$. When $k \leq \ell \leq n-1$, this sum is equal to $0$ because
	$$
	\sum_{j=\ell}^{n}(-1)^{j-\ell}\binom{n-\ell}{j-\ell}= \sum_{s=0}^{n-\ell}(-1)^{s}\binom{n-\ell}{s} = (1-1)^{n-\ell}=0,
	$$
	and if  $\ell \leq k-1$ 
	$$
	\sum_{j=k}^{n} (-1)^{j-\ell}\binom{n-\ell}{j-\ell} = \sum_{s=k-\ell}^{n-\ell}(-1)^{s}\binom{n-\ell}{s} =(-1)^{k-\ell}\binom{n-\ell-1}{k-1},
	$$
	which concludes the proof.
\end{proof}

\begin{corollary}\label{corexpleqkIkcaseG}
	For any $x_{\vec{1}} \in \mathds{X}_{n} $ and $G \subset \{1, \ldots, n \}$ with $|G|=k$, the kernel $K^{\mathfrak{I}}$ of Equation \eqref{expleqkIk} for a PDI$_{k}$ kernel that is zero at the extended diagonal $\Delta_{k-1}^{n}$ satisfies 
	\begin{equation}\label{expleqkIkcaseG}
		K^{\mathfrak{I}}(x_{\vec{1}_{G}}, x_{\vec{1}_{G}}) = 2^{k}\mathfrak{I}(x_{\vec{1}_{G}}, x_{\vec{0}} ).
	\end{equation}
\end{corollary}

\begin{proof} Indeed, by Lemma \ref{muknsimplified}
	$$
	\mu_{k}^{n}[x_{\vec{1}_{G}}, x_{\vec{0}}] = \sum_{|F|\geq k }\mu_{|F|}^{|F|}[x_{\vec{1}_{F\cap G}},x_{\vec{0}_{F}}]
	\times \delta(x_{\vec{0}_{F^{c}}}).
	$$
	If $F \cap G \subsetneq F$, then $\{i \in F, \quad (x_{\vec{1}_{F\cap G}})_{i}= (x_{\vec{0}_{F}})_{i} \} \neq \emptyset$, hence by the comment after Equation \eqref{measorderk}, we have $\mu_{|F|}^{|F|}[x_{\vec{1}_{F\cap G}},x_{\vec{0}_{F}}]=0$. As this occurs on any set $F$ such that $|F| \geq k$ and which is not $G$, we obtain  $ \mu_{k}^{n}[x_{\vec{1}_{G}}, x_{\vec{0}}] = \mu_{k}^{k}[x_{\vec{1}_{ G}},x_{\vec{0}_{G}}]
	\times \delta(x_{\vec{0}_{G^{c}}})$. From this, we conclude 
	\begin{align*}
		K^{\mathfrak{I}}(x_{\vec{1}_{G}},x_{\vec{1}_{G}})&= \int_{\mathds{X}_{n} }\int_{\mathds{X}_{n} }(-1)^{k}\mathfrak{I}(u,v)d\mu_{k}^{n}[x_{\vec{1}_{G}}, x_{\vec{0}}](u)d\mu_{k}^{n}[x_{\vec{1}_{G}}, x_{\vec{0}}](v)\\
		&=\sum_{\alpha, \beta \in \mathbb{N}_{1}^{0,k}}(-1)^{k+|\alpha| + |\beta|}\mathfrak{I}(x_{\alpha_{G}}, x_{\beta_{G}})\\
		&= \sum_{\alpha \in \mathbb{N}_{1}^{0,k}}\mathfrak{I}(x_{\alpha_{G}}, x_{ (\vec{1} - \alpha)_{G}}) = 2^{k}\mathfrak{I}(x_{\vec{1}_{G}}, x_{\vec{0}}),
	\end{align*}
	where the equalities in the last line are due to the fact that $\mathfrak{I}$ is zero at the extended diagonal and the $n$-symmetry.\end{proof}

The next result is a partial version of Corollary \ref{growthpdiknvar} for an arbitrary PDI$_{k}$ kernel. 

\begin{theorem}\label{ineqpdik}Let $\mathfrak{I}: \mathds{X}_{n} \times \mathds{X}_{n} \to \mathbb{R}$ be a complete $n$-symmetric PDI$_{k}$ kernel that is zero at the extended diagonal $\Delta_{k-1}^{n}$. Then there exists a constant $C_{n,k}> 0$ for which
	\begin{equation}\label{ineqpdiknprin}
		|\mathfrak{I}(x_{\vec{1}}, x_{\vec{2}})| \leq     C_{n,k} \sum_{|F|=k}\mathfrak{I}(x_{\vec{1}_{F}+ \vec{3}_{F^{c}}}, x_{\vec{2}_{F}+ \vec{3}_{F^{c}}}), \quad x_{\vec{1}},x_{\vec{2}}, x_{\vec{3}} \in \mathds{X}_{n}.
	\end{equation}
	In particular, the only complete $n$-symmetric PDI$_{k}$ kernel that is zero at the extended diagonal $\Delta_{k}^{n}$ is the zero kernel.
\end{theorem}

\begin{proof}The second part is a direct consequence of the inequality, because we would have that for any $ |F|=k $ and $x_{\vec{1}},x_{\vec{2}}, x_{\vec{3}} \in \mathds{X}_{n}$  it occur that $\mathfrak{I}(x_{\vec{1}_{F}+ \vec{3}_{F^{c}}}, x_{\vec{2}_{F}+ \vec{3}_{F^{c}}})=0$.\\
	The rest of the proof is done by induction on $n\geq k$. The case $n=k$ is immediate as we can define $C_{n,n}=1$ and then  both sides have the same value. Suppose then that it holds for all values of $n \in \{k, \ldots, m-1\}$, and we prove that it also holds for $n=m$.\\
	By Lemma \ref{projectionkernels}, for every $G \subset \{1, \ldots, m\}$, $k\leq |G|\leq m-1$, and fixed $x_{\vec{3}} \in \mathds{X}_{n}$, the kernel
	$$
	(x_{\vec{1}_{G}},x_{\vec{2}_{G}} ) \in \mathds{X}_{G}\times \mathds{X}_{G} \to \mathfrak{I}(x_{\vec{1}_{G}+ \vec{3}_{G^{c}}}, x_{\vec{2}_{G}+ \vec{3}_{G^{c}}}),
	$$
	is a complete $|G|$-symmetric PDI$_{k}$ kernel that is zero at the extended diagonal $\Delta_{k-1}^{|G|}$. In particular, the terms on the right-hand side of Equation \eqref{ineqpdiknprin} are nonnegative. By the hypothesis, for every such $G$, we have that 
	$$
	|\mathfrak{I}(x_{\vec{1}_{G}+ \vec{3}_{G^{c}}}, x_{\vec{2}_{G}+ \vec{3}_{G^{c}}})| \leq     C_{|G|,k} \sum_{\mathcal{F} \subset G, |\mathcal{F}|=k}\mathfrak{I}(x_{\vec{1}_{\mathcal{F}}+ \vec{3}_{\mathcal{F}^{c}}}, x_{\vec{2}_{\mathcal{F}}+ \vec{3}_{\mathcal{F}^{c}}}), \quad x_{\vec{1}},x_{\vec{2}}, x_{\vec{3}} \in \mathds{X}_{m}.
	$$
	Fix an arbitrary $\mathcal{G} \subset \{1, \ldots, m\}$ with $|\mathcal{G}|=k$, and consider the measures $\lambda^{\mathcal{G}}:= \delta(x_{\vec{1}_{\mathcal{G}^{c}}}) + \delta(x_{\vec{2}_{\mathcal{G}^{c}}})$, $\lambda^{\prime, \mathcal{G}}:= \delta(x_{\vec{1}_{\mathcal{G}^{c}}}) - \delta(x_{\vec{2}_{\mathcal{G}^{c}}}) \in \mathcal{M}(\mathds{X}_{\mathcal{G}^{c}})$ and $\mu =\bigtimes_{i \in \mathcal{G}}(\delta_{x_{i}^{1}}- \delta_{x_{i}^{2}}) \in \mathcal{M}_{k}(\mathds{X}_{\mathcal{G}}) $. Then $\lambda^{ \mathcal{G}} \times \mu$ and $\lambda^{\prime, \mathcal{G}}\times \mu$ are elements of $\mathcal{M}_{k}(\mathds{X}_{n})$. Hence
	\begin{align*}
		0& \leq \int_{\mathds{X}_{n}}\int_{\mathds{X}_{n}}(-1)^{k}\mathfrak{I}(u,v ) d[\mu \times \lambda^{\mathcal{G}}](u)d[\mu \times \lambda^{\mathcal{G}}](v)\\
		&=2^{k}\mathfrak{I}(x_{\vec{1} },x_{\vec{2}_{\mathcal{G}}+\vec{1}_{\mathcal{G}^{c}} } ) + 2^{k}\mathfrak{I}(x_{\vec{1}_{\mathcal{G}}+\vec{2}_{\mathcal{G}^{c}} },x_{\vec{2} } ) +2\left [    \sum_{\alpha, \beta \in \mathbb{N}_{2}^{\mathcal{G}}}(-1)^{k+ |\alpha|+|\beta|}\mathfrak{I}( x_{\alpha_{\mathcal{G}}+ \vec{1}_{\mathcal{G}^{c}}}, x_{\beta_{\mathcal{G}}+\vec{2}_{\mathcal{G}^{c}}} )\right],
	\end{align*}
	because the kernel $\mathfrak{I}$ is zero at the extended diagonal $\Delta_{k-1}^{n}$. Using the same approach on the measure     $\lambda^{\prime, \mathcal{G}}$ and comparing the inequalities, we obtain that
	\begin{align*}
		&\left |    \sum_{\alpha, \beta \in \mathbb{N}_{2}^{\mathcal{G}}}(-1)^{ |\alpha|+|\beta|}\mathfrak{I}( x_{\alpha_{\mathcal{G}}+ \vec{1}_{\mathcal{G}^{c}}}, x_{\beta_{\mathcal{G}}+\vec{2}_{\mathcal{G}^{c}}} )\right | \leq 2^{k-1}\mathfrak{I}(x_{\vec{1}},x_{\vec{2}_{\mathcal{G}}+\vec{1}_{\mathcal{G}^{c}} } ) + 2^{k-1}\mathfrak{I}(x_{\vec{1}_{\mathcal{G}}+\vec{2}_{\mathcal{G}^{c}} },x_{\vec{2} } ).
	\end{align*}
	However, the kernel $\mathfrak{I}$ is complete $n$-symmetric, so for an arbitrary but fixed $x_{\vec{3}} \in \mathds{X}_{n}$
	$$
	\mathfrak{I}(x_{\vec{1} },x_{\vec{2}_{\mathcal{G}}+\vec{1}_{\mathcal{G}^{c}} } ) = \mathfrak{I}(x_{\vec{1}_{\mathcal{G}}+\vec{2}_{\mathcal{G}^{c}} },x_{\vec{2} } ) =\mathfrak{I}(x_{\vec{1}_{\mathcal{G}}+\vec{3}_{\mathcal{G}^{c}} },x_{\vec{2}_{\mathcal{G}}+\vec{3}_{\mathcal{G}^{c}} } ). 
	$$
	For every $H \subset \mathcal{G}$, it holds that
	$$
	\sum_{\xi \in \mathbb{N}_{2}^{|\mathcal{G}|-|H|}} \left [\sum_{\varsigma \in \mathbb{N}_{2}^{|H|}}(-1)^{|\varsigma_{H}+\xi_{\mathcal{G}\setminus{H} }|} (-1)^{|(\vec{3}-\varsigma)_{H}+\xi_{\mathcal{G}\setminus{H} }|} \right ]= 2^{k}(-1)^{|H|},
	$$ 
	hence, by  Equation \eqref{simplidouble} and the complete $n$-symmetry we have that 
	\begin{align*}
		&\sum_{\alpha, \beta \in \mathbb{N}_{2}^{\mathcal{G}}}(-1)^{ |\alpha|+|\beta|}\mathfrak{I}( x_{\alpha_{\mathcal{G}}+ \vec{1}_{\mathcal{G}^{c}}}, x_{\beta_{\mathcal{G}}+\vec{2}_{\mathcal{G}^{c}}} )=2^{k}\sum_{H \subset \mathcal{G}} (-1)^{|H|} \mathfrak{I}(x_{\vec{1}_{H \cup \mathcal{G}^{c}}+ \vec{3}_{\mathcal{G} \setminus{H}}},x_{\vec{2}_{H \cup\mathcal{G}^{c}}+ \vec{3}_{\mathcal{G} \setminus{H}}} ).
	\end{align*}
	In the previous sum, when $H=\mathcal{G}$, we have the term $2^{k} \mathfrak{I}(x_{\vec{1}},x_{\vec{2}})$, hence, by the triangle inequality
	\begin{align*}
		&|\mathfrak{I}(x_{\vec{1}}, x_{\vec{2}})|\\
		&\leq    2^{-k} \left |    \sum_{\alpha, \beta \in \mathbb{N}_{2}^{\mathcal{G}}}(-1)^{|\alpha|+|\beta|}\mathfrak{I}( x_{\alpha_{\mathcal{G}}+ \vec{1}_{\mathcal{G}^{c}}}, x_{\beta_{\mathcal{G}}+\vec{2}_{\mathcal{G}^{c}}} )\right | + \sum_{H \subsetneq \mathcal{G}} \left | \mathfrak{I}(x_{\vec{1}_{H \cup \mathcal{G}^{c}}+ \vec{3}_{\mathcal{G} \setminus{H}}},x_{\vec{2}_{H \cup\mathcal{G}^{c}}+ \vec{3}_{\mathcal{G} \setminus{H}}} )\right |\\
		&\leq \mathfrak{I}(x_{\vec{1}_{\mathcal{G}}+\vec{3}_{\mathcal{G}^{c}} },x_{\vec{2}_{\mathcal{G}}+\vec{3}_{\mathcal{G}^{c}} } ) + \sum_{H \subsetneq \mathcal{G}} C_{|H\cup \mathcal{G}^{c}|, k} \left [ \sum_{\mathcal{F} \subset H\cup \mathcal{G}^{c}, |\mathcal{F}|=k}\mathfrak{I}(x_{\vec{1}_{\mathcal{F}}+ \vec{3}_{\mathcal{F}^{c}}}, x_{\vec{2}_{\mathcal{F}}+ \vec{3}_{\mathcal{F}^{c}}}) \right ]\\
		&=\mathfrak{I}(x_{\vec{1}_{\mathcal{G}}+\vec{3}_{\mathcal{G}^{c}} },x_{\vec{2}_{\mathcal{G}}+\vec{3}_{\mathcal{G}^{c}} } ) + \sum_{|\mathcal{F}|=k}\left [ \sum_{H}^{\mathcal{F} \cap \mathcal{G}\subset H \subsetneq \mathcal{G}} C_{|H\cup \mathcal{G}^{c}|, k}\right ]\mathfrak{I}(x_{\vec{1}_{\mathcal{F}}+ \vec{3}_{\mathcal{F}^{c}}}, x_{\vec{2}_{\mathcal{F}}+ \vec{3}_{\mathcal{F}^{c}}}),
	\end{align*}
	which concludes the proof.
\end{proof}

We do not need a precise estimation for $C_{n,k}$, as the inequality is used to obtain integrability properties for $\mathfrak{I}$. It remains elusive whether the left-hand inequality of the main equation in Corollary \ref{growthpdiknvar} can be obtained for an arbitrary PDI$_{k}$ kernel $\mathfrak{I}$ in $\mathds{X}_{n}$.

In Theorem \ref{finalgeomkn}, and more explicitly in Example \ref{examplecoefi},   we show that $\mathfrak{I}$ is a sum of squares, and that there are several possibilities to describe $\mathfrak{I}$ in this way. However, we could prove that when either the codimension $n-k=1$ (see Subsection \ref{cod1}) or when the codimension $n-k=2$ (see Subsection \ref{cod2}), there is a nonnegative combination of sums of squares that describes $\mathfrak{I}$. As a by-product, in these two specific cases, we obtain that 
\begin{equation} 
	B_{n,k}    \sum_{|F|=k}\mathfrak{I}(x_{\vec{1}_{F}+ \vec{3}_{F^{c}}}, x_{\vec{2}_{F}+ \vec{3}_{F^{c}}})\leq \mathfrak{I}(x_{\vec{1}}, x_{\vec{2}}) \leq C_{n,k}    \sum_{|F|=k}\mathfrak{I}(x_{\vec{1}_{F}+ \vec{3}_{F^{c}}}, x_{\vec{2}_{F}+ \vec{3}_{F^{c}}}), 
\end{equation}
for every $x_{\vec{1}},x_{\vec{2}}, x_{\vec{3}} \in \mathds{X}_{n}$ and some nonnegative constant $B_{n,k}$. We conjecture that the results in Subsections \ref{cod1} and \ref{cod2} can be generalized for any $0\leq k \leq n$.

\subsection{Integrability restrictions}\label{Integrabilityrestrictionsnk}

In this subsection, we prove the technical results regarding the description of the continuous probability measures that can be compared using a continuous complete $n$-symmetric PDI$_{k}$ kernel, and provide a Kernel Mean Embedding result for them. Due to the inequality of Theorem \ref{ineqpdik}, they are often a consequence of the results proved in Subsection \ref{Integrabilityrestrictionsn}. 

\begin{corollary}\label{integrestrik} Let $n>k$, $\mathfrak{I}: \mathds{X}_{n} \times \mathds{X}_{n} \to \mathbb{R}$ be a continuous complete $n$-symmetric PDI$_{k}$ kernel that is zero at the extended diagonal $\Delta_{k-1}^{n}$ of $\mathds{X}_{n}$. Then, the following conditions are equivalent for a measure $ \mu \in \mathfrak{M}(\mathds{X}_{n})$:
	\begin{enumerate}
		\item[$(i)$] The function 
		$$
		(x_{\vec{1}}, x_{\vec{2}}) \in \mathds{X}_{n}\times \mathds{X}_{n} \to \mathfrak{I}(x_{\vec{1} },x_{\vec{2}}) \in \mathbb{R}
		$$
		is FIP with respect to $\mu\times \mu \in \mathfrak{M}(\mathds{X}_{n}\times \mathds{X}_{n})$.
		\item[$(ii)$] There exists an $x_{\vec{3}} \in \mathds{X}_{n}$ such that
		$$ 
		x_{\vec{1}} \in \mathds{X}_{n} \to \mathfrak{I}(x_{\vec{1} },x_{\vec{3}}) \in \mathbb{R}
		$$
		is FIP with respect to $\mu \in \mathfrak{M}(\mathds{X}_{n})$.
		\item[$(iii)$] For every $x_{\vec{3}} \in \mathds{X}_{n}$ 
		$$ 
		x_{\vec{1}} \in \mathds{X}_{n} \to \mathfrak{I}(x_{\vec{1} },x_{\vec{3}}) \in \mathbb{R}
		$$
		is FIP with respect to $\mu \in \mathfrak{M}(\mathds{X}_{n})$.
		\item[$(iv)$] For whichever $x_{\vec{0}} \in \mathds{X}_{n}$ that is used to define $K^{\mathfrak{I}}$ defined in Lemma \ref{pdi2topdn}, the function 
		$$ 
		x_{\vec{1}} \in \mathds{X}_{n} \to K^{\mathfrak{I}}(x_{\vec{1}}, x_{\vec{1}}) \in \mathbb{R}
		$$
		is FIP with respect to $\mu \in \mathfrak{M}(\mathds{X}_{n})$.
		\item[$(v)$] For any $F \subset \{1, \ldots, n\}$ with $|F|=k$, the PDI$_{k}$ kernel $$
		\mathfrak{I}_{F} : \mathds{X}_{F} \times \mathds{X}_{F} \to \mathbb{R}, \quad     \mathfrak{I}_{F}(x_{\vec{1}_{F}}, x_{\vec{2}_{F}}):=\mathfrak{I} (x_{\vec{1}_{F} + \vec{4}_{F^{c}}}, x_{\vec{2}_{F} + \vec{4}_{F^{c}}}),
		$$
		satisfies the equivalences of Lemma \ref{probintegPDIn} for the measure $|\mu|_{F} \in \mathfrak{M}(\mathds{X}_{F})$ (which is independent of the choice of $x_{\vec{4}} \in \mathds{X}_{n}$).
	\end{enumerate}   
\end{corollary}
\begin{proof}
	If relation $(v)$ is satisfied, then by Equation \eqref{ineqpdiknprin} and relation $(iii)$ in Lemma \ref{propertiesFIP}, we obtain that each one of the equivalences $(i)$, $(ii)$ and $(iii)$ occurs. Conversely, if any of the relations $(i)$, $(ii)$ and $(iii)$ holds, we obtain relation $(v)$ by its respective relation in Lemma \ref{probintegPDIn}, due to relation $(ii)$ in Lemma \ref{propertiesFIP}.\\
	To conclude, if relation $(iv)$ holds, then by Equation \eqref{expleqkIkcaseG}, we have that $x_{\vec{1}} \to \mathfrak{I}(\vec{1}_{G}, x_{\vec{0}})$ is also FIP with respect to $\mu$ for any $G$ with $|G|=k$, and Equation \eqref{ineqpdiknprin} concludes that relation $(i)$ is satisfied. On the other hand, if relation $(i)$ is satisfied, then relation $(iv)$ occurs by Equation \eqref{expleqkIk}. \end{proof}

The proof of the next lemma is very similar to that of Lemma \ref{integPDIn0}, so its proof is omitted. 

\begin{lemma}\label{integPDInk} Let $n>k$, $\mathfrak{I}: \mathds{X}_{n} \times \mathds{X}_{n} \to \mathbb{R}$ be a continuous complete $n$-symmetric PDI$_{k}$ kernel that is zero at the extended diagonal $\Delta_{k-1}^{n}$ of $\mathds{X}_{n}$. Consider the sets
	\begin{align*}
		\mathcal{P}[\mathfrak{I}]:=&\{ Q, \quad Q \text{ is a probability and satisfies Corollary \ref{integrestrik}} \},\\
		\mathfrak{M}[\mathfrak{I}]:=&\{ \mu, \quad \mu \in \mathfrak{M}(\mathds{X}_{n}) \text{ satisfies Corollary \ref{integrestrik}} \},\\
		\mathfrak{M}_{k}[\mathfrak{I}]:=&\mathfrak{M}[\mathfrak{I}]\cap \mathfrak{M}_{k}(\mathds{X}_{n}),\\
		\mathfrak{M}[ \mathfrak{I}, P]:=&\{ \mu \in \mathfrak{M}[\mathfrak{I}]\cap \mathfrak{M}(\mathds{X}_{n}), \quad \text{ $\mu$ is marginally delimited by $P \in \mathcal{P}[\mathfrak{I}]$} \},\\
		\mathfrak{M}_{k }[ \mathfrak{I}, P]:=& \mathfrak{M}[\mathfrak{I},P]\cap \mathfrak{M}_{k}(\mathds{X}_{n}).
	\end{align*}
	Then, we have that    
	\begin{enumerate}
		\item [$(i)$] If $P \in \mathcal{P}[\mathfrak{I}]$, then every linear combination $\sum_{\pi} a_{\pi}P_{\pi}$ is an element of $     \mathfrak{M}[ \mathfrak{I}, P]$.   
		\item [$(ii)$] If $P \in \mathcal{P}[\mathfrak{I}]$ and $\mu \in \mathfrak{M} (\mathds{X}_{n})$ is a measure for which there exists a constant $M \geq 0$ such that the measure $MP -|\mu|$ is nonnegative, then $\mu \in \mathfrak{M}[\mathfrak{I}, P]$.
		\item [$(iii)$] For any fixed $P \in \mathcal{P}[\mathfrak{I}]$, the sets $\mathfrak{M}_{k }[\mathfrak{I}, P]$ and $\mathfrak{M}[\mathfrak{I}, P]$ are vector spaces. 
		\item [$(iv)$] If $P \in \mathcal{P}[\mathfrak{I}]$, the generalized Lancaster interaction $\Lambda_{k}^{n}[P]$ is an element of $\mathfrak{M}_{k }[\mathfrak{I}, P]$. 
		\item [$(v)$] If $\mu \in \mathfrak{M}[\mathfrak{I}]$, then $|\mu| \in \mathfrak{M}[\mathfrak{I}]$. In particular, a nonzero $\mu$ is marginally delimited by the probability $|\mu|/ |\mu|( \mathds{X}_{n}).$ 
	\end{enumerate}
\end{lemma}

The following result is a version of the kernel mean embedding for PDI$_{k}$ kernels.

\begin{theorem}\label{integPDIkn}Let $n>k$, $\mathfrak{I}: \mathds{X}_{n} \times \mathds{X}_{n} \to \mathbb{R}$ be a continuous complete $n$-symmetric PDI$_{k}$ kernel that is zero at the extended diagonal $\Delta_{k-1}^{n}$ of $\mathds{X}_{n}$.  Then for any $\lambda \in  \mathfrak{M}_{k}[\mathfrak{I}]$ and for whichever $x_{\vec{0}} \in \mathds{X}_{n}$ is used to define $K^{\mathfrak{I}}$    
	\begin{align*}
		\int_{\mathds{X}_{n}}\int_{\mathds{X}_{n}}(-1)^{k}\mathfrak{I}(u, v)d\lambda(u)d\lambda(v)&= \int_{\mathds{X}_{n}}\int_{\mathds{X}_{n}}K^{\mathfrak{I}}(u, v)d\lambda(u)d \lambda(v)\\
		&= \langle K^{\mathfrak{I}}_{\lambda}, K^{\mathfrak{I}}_{\lambda} \rangle _{\mathcal{H}_{\mathfrak{I}}} \geq 0.
	\end{align*}    
	In particular, for any fixed $P \in \mathcal{P}[\mathfrak{I}]$, the bilinear function 
	$$
	(\lambda, \eta) \in \mathfrak{M}_{k}[ \mathfrak{I}, P]\times \mathfrak{M}_{k}[\mathfrak{I}, P] \to \int_{\mathds{X}_{n}}\int_{\mathds{X}_{n}}(-1)^{k}\mathfrak{I}(u, v)d\lambda(u)d \eta(v) 
	$$    
	is well defined and is a semi-inner product because
	\begin{align*}
		\int_{\mathds{X}_{n}}\int_{\mathds{X}_{n}}(-1)^{k}\mathfrak{I}(u, v)d\lambda(u)d\eta(v)&= \int_{\mathds{X}_{n}}\int_{\mathds{X}_{n}}K^{\mathfrak{I}}(u, v)d\lambda(u)d \eta(v)\\
		&= \langle K^{\mathfrak{I}}_{\lambda}, K^{\mathfrak{I}}_{\eta} \rangle _{\mathcal{H}_{\mathfrak{I}}}.
	\end{align*}
\end{theorem}

\begin{proof}         Let $x_{\vec{0}} \in \mathds{X}_{n}$ be arbitrary and consider the PD kernel $K^{\mathfrak{I}}$ related to it, whose explicit expression is given in Equation \eqref{expleqkIk}.\\ 
	Since $\lambda, \eta \in \mathfrak{M}[ \mathfrak{I},P]$, then $ |\lambda|, |\eta| \in \mathfrak{M}[ \mathfrak{I}]$, and these two measures are also marginally delimited by $P$. Thus, by relation $(iii)$ in Lemma \ref{integPDInk}, we obtain that $|\lambda | + |\eta| \in \mathfrak{M}[ \mathfrak{I},P]$. As a direct consequence of this fact, we obtain that all kernels that appear on the right-hand side of Equation \eqref{expleqkIk} are in $L^{1}(|\lambda|\times |\eta|)$, because they are elements of $L^{1}((|\lambda| + |\eta|)\times (|\lambda| +|\eta|))$.\\ 
	By Equation \eqref{integmu0n}, it holds
	$$
	\int_{\mathds{X}_{n}}\int_{\mathds{X}_{n}}(-1)^{k}\mathfrak{I}(u, v)d\lambda(u)d \eta(v)= \int_{\mathds{X}_{n}}\int_{\mathds{X}_{n}}K^{\mathfrak{I}}(u, v)d\lambda (u)d\eta(v),
	$$
	because, with the exception of the term $\mathfrak{I}(x_{\vec{1}},x_{\vec{2}})$, all the other terms that appear on the right-hand side of Equation \eqref{expleqkIk} are zero.\\
	The third equality is a direct consequence of the kernel mean embedding in Theorem \ref{initialextmmddominio} because $K^{\mathfrak{I}}(x_{\vec{1}}, x_{\vec{1}}) \in L^{1}(|\lambda|+|\eta|)$, as due to complete $n$-symmetry, all terms on the right-hand side of Equation \eqref{expleqkIk} for when $x_{\vec{1}}= x_{\vec{2}}$ are in $L^{1}(|\lambda|+|\eta|)$.\end{proof}

\begin{definition}\label{PDIkn-Characteristic} Let $\mathfrak{I}: \mathds{X}_{n} \times \mathds{X}_{n} \to \mathbb{R}$ be a continuous $n$-symmetric PDI$_{k}$ kernel that is zero at the extended diagonal $\Delta_{k-1}^{n}(\mathds{X}_{n})$.  We say that $\mathfrak{I}$ is PDI$_{k}$-Characteristic if for every nonzero $\lambda \in  \mathfrak{M}_{k}[\mathfrak{I}]$
	\[
	\int_{\mathds{X}_{n}}\int_{\mathds{X}_{n}}(-1)^{k}\mathfrak{I}(u, v)d\lambda(u)d \lambda(v)>0.
	\]
	Equivalently, $\mathfrak{I}$ is    PDI$_{k}$-Characteristic if for any fixed $P \in \mathcal{P}[\mathfrak{I}]$, the bilinear function 
	$$
	(\lambda, \eta) \in \mathfrak{M}_{k}[\mathfrak{I}, P]\times \mathfrak{M}_{k}[ \mathfrak{I}, P] \to \int_{\mathds{X}_{n}}\int_{\mathds{X}_{n}}(-1)^{k}\mathfrak{I}(u, v)d\lambda(u)d \eta(v) 
	$$    
	is an inner product.
\end{definition}

\section{Kronecker products of PDI kernels}\label{KroneckerproductsofPDIkernels}
The key property of distance covariance, as defined in \cite{Szekely2007, Lyons2013} and other references, is that $\gamma$ and $\varsigma$ being CND-Characteristic kernels is necessary and sufficient for their Kronecker product to be used as an independence test. In \cite{guella2023}, it is proved that this Kronecker product can do even more: it is PDI$_{2}$-Characteristic.  

In this section, we characterize how this property behaves in general by analyzing the properties of the Kronecker product of PDI kernels. We first address the discrete case before detailing the continuous case and its necessary integrability restrictions in  Subsection \ref{continuouspdikkron}.

In order to understand them, we define a new class of subspaces of $\mathcal{M}( \mathds{X}_{n} \times \mathds{Y}_{m})$. Under the restriction that $0\leq a \leq n$ and $0\leq b \leq m$, we define 
\[
\mathcal{M}_{a,b}( \mathds{X}_{n},\mathds{Y}_{m}):= \mathcal{M}_{a, 0}( \mathds{X}_{n},\mathds{Y}_{m})\cap \mathcal{M}_{0, b}( \mathds{X}_{n},\mathds{Y}_{m}),
\] 
\begin{equation*}
	\begin{split}
		\mathcal{M}_{a, 0}( \mathds{X}_{n},\mathds{Y}_{m}):&= \{ \mu \in \mathcal{M}( \mathds{X}_{n} \times \mathds{Y}_{m}), \quad \mu\left(\left[\prod_{i=1}^{n}A_{i}\right] \times \left[\prod_{j=1}^{m}B_{j}\right] \right)=0,\\
		& \text{ when } |\{ i, \quad A_{i}=X_{i} \}| \geq n-a+1, \text{ for arbitrary choices of } B_{j} \}.
	\end{split}
\end{equation*}
\begin{equation*}
	\begin{split}
		\mathcal{M}_{0, b}( \mathds{X}_{n},\mathds{Y}_{m}):&= \{ \mu \in \mathcal{M}( \mathds{X}_{n} \times \mathds{Y}_{m}), \quad \mu\left(\left[\prod_{i=1}^{n}A_{i}\right] \times \left[\prod_{j=1}^{m}B_{j}\right] \right)=0,\\
		& \text{ when } |\{ j, \quad B_{j}=Y_{j} \}| \geq m-b+1, \text{ for arbitrary choices of } A_{i} \}.
	\end{split}
\end{equation*}

We point out the immediate equalities
$$
\mathcal{M}_{0,0}( \mathds{X}_{n},\mathds{Y}_{m})= \mathcal{M}_{0 }( \mathds{X}_{n}\times \mathds{Y}_{m}), \quad \mathcal{M}_{n,m}( \mathds{X}_{n},\mathds{Y}_{m})= \mathcal{M}_{n+m }( \mathds{X}_{n}\times \mathds{Y}_{m}),
$$
and, similar to Equation \eqref{inclumeas}, if $c\geq a$ and $d \geq b$, then 
\begin{equation}\label{inclumeasdoubleindex}
	\mathcal{M}_{c,d}( \mathds{X}_{n},\mathds{Y}_{m}) \subset \mathcal{M}_{a,b}( \mathds{X}_{n},\mathds{Y}_{m}).
\end{equation}

Analogous to the measures in $\mathcal{M}_{k}(\mathds{X}_{n})$, the technical property that we frequently use for a measure $\mu \in \mathcal{M}_{a, b}( \mathds{X}_{n},\mathds{Y}_{m})$ is that if $T: \mathds{X}_{n}\times \mathds{Y}_{m} \to \mathbb{R}$ either only depends on at most $a-1$ of the $n$ variables of $\mathds{X}_{n}$ or only depends on at most $b-1$ of the $m$ variables of $\mathds{Y}_{m}$, then 
\begin{equation}\label{integmuklnm}
	\int_{\mathds{X}_{n} \times \mathds{Y}_{m}}T(x, y)d\mu(x, y)=0.
\end{equation}

Before showing the main results, we detail two lemmas that will be needed. We use the notation that for $r \in \mathbb{R}$, $\operatorname{sgn}(r)=1$ for $r>0$, $\operatorname{sgn}(r)=-1$ for $r<0$, and $\operatorname{sgn}(0)=0$.

\begin{lemma}\label{inclusionmeasures2index} For any $0\leq k \leq n+m$, the following inclusion is satisfied:
	\begin{equation}\label{abeq1}
		\begin{split}   
			&\mathcal{M}_{a^{\prime}, b^{\prime \prime }}( \mathds{X}_{n},\mathds{Y}_{m}) \subset    \mathcal{M}_{k}( \mathds{X}_{n} \times \mathds{Y}_{m}) \subset \mathcal{M}_{a^{\prime}, b^{\prime}}( \mathds{X}_{n},\mathds{Y}_{m}),\\
			&     \mathcal{M}_{a^{\prime \prime }, b^{\prime}}( \mathds{X}_{n},\mathds{Y}_{m}) \subset    \mathcal{M}_{k}( \mathds{X}_{n} \times \mathds{Y}_{m}) \subset \mathcal{M}_{a^{\prime}, b^{\prime}}( \mathds{X}_{n},\mathds{Y}_{m}),
		\end{split} 
	\end{equation}
	where $a^{\prime}:=\max(k-m,0) $, $b^{\prime}:=\max(k-n,0) $, $a^{\prime \prime }:= \min(n,k)$, and $b^{\prime \prime }:= \min(m,k)$.\\
	Also, for any $0\leq a \leq n$ and $0\leq b \leq m$, the following inclusion is satisfied:
	\begin{equation}\label{abeq2}
		\mathcal{M}_{k^{\prime}}( \mathds{X}_{n} \times \mathds{Y}_{m}) \subset \mathcal{M}_{a, b}( \mathds{X}_{n},\mathds{Y}_{m}) \subset \mathcal{M}_{a+ b}( \mathds{X}_{n} \times \mathds{Y}_{m}),
	\end{equation}
	where $k^{\prime}:=\max((a+m)\operatorname{sgn}(a), (n+b)\operatorname{sgn}(b))$ and $\operatorname{sgn}(c)= 0$ when $c=0$.
\end{lemma}        

\begin{proof}First, we prove the second inclusion in Equation \eqref{abeq1}.
	By the definition of $\mathcal{M}_{a^{\prime},b^{\prime}}( \mathds{X}_{n},\mathds{Y}_{m})$, we prove that $\mathcal{M}_{k}( \mathds{X}_{n} \times \mathds{Y}_{m}) \subset \mathcal{M}_{a^{\prime},0}( \mathds{X}_{n},\mathds{Y}_{m})$, the proof for the inclusion with $b^{\prime} $ being similar. When $a^{\prime} =0$, it is immediate because $ \mathcal{M}_{0,0}( \mathds{X}_{n},\mathds{Y}_{m})= \mathcal{M} ( \mathds{X}_{n}\times \mathds{Y}_{m})$. When $a^{\prime} >0$, we have that $a^{\prime} = k-m$. Hence, for any measure $\mu \in \mathcal{M}_{k}( \mathds{X}_{n} \times \mathds{Y}_{m})$, subsets $A_{i} \subset X_{i}$ and $B_{j} \subset Y_{j}$, we have that $\mu\left(\left[\prod_{i=1}^{n}A_{i}\right] \times \left[\prod_{j=1}^{m}B_{j}\right] \right)=0$ when the number $s:=|\{ j, \quad B_{j}=Y_{j} \}|$ is arbitrary and $r:= |\{ i, \quad A_{i}=X_{i} \}| \geq n- (k-m) +1$ because 
	$$
	r+s \geq r \geq n- (k-m) + 1 = n+m-k+1, 
	$$
	thus $\mu \in \mathcal{M}_{a^{\prime},0}( \mathds{X}_{n},\mathds{Y}_{m})$. From this, we obtain the first inclusion of Equation \eqref{abeq2} as 
	\begin{align*}
		\max(k^{\prime} -m, 0)&=\max((a+m)\operatorname{sgn}(a) -m, (n+b)\operatorname{sgn}(b)-m,0) \geq a,\\
		\max(k^{\prime} -n, 0)&= \max((a+m)\operatorname{sgn}(a) -n, (n+b)\operatorname{sgn}(b)-n,0)\geq b, 
	\end{align*}
	hence, by Equation \eqref{inclumeasdoubleindex},
	$$
	\mathcal{M}_{k^{\prime}}( \mathds{X}_{n} \times \mathds{Y}_{m}) \subset \mathcal{M}_{\max(k^{\prime} -m, 0), \max(k^{\prime} -n, 0)}( \mathds{X}_{n},\mathds{Y}_{m}) \subset \mathcal{M}_{a,b }( \mathds{X}_{n},\mathds{Y}_{m}).
	$$
	Now we prove the second inclusion in Equation \eqref{abeq2}. Indeed, if $\mu \in \mathcal{M}_{a,b}( \mathds{X}_{n}, \mathds{Y}_{m})$, for subsets $A_{i} \subset X_{i}$ and $B_{j} \subset Y_{j}$, define $r:= |\{ i, \quad A_{i}=X_{i} \}| $ and $s:=|\{ j, \quad B_{j}=Y_{j} \}| $. If $r+s\geq n +m -(a+b) +1 $, it must occur that either $r \geq n-a +1$ or $s\geq m-b +1$, but then $\mu\left(\left[\prod_{i=1}^{n}A_{i}\right] \times \left[\prod_{j=1}^{m}B_{j}\right] \right)=0$ because $\mu \in \mathcal{M}_{a,b}( \mathds{X}_{n},\mathds{Y}_{m})$. Thus, we obtain that $\mu \in \mathcal{M}_{a + b}( \mathds{X}_{n} \times \mathds{Y}_{m})$. 
	From this last proven inclusion, we obtain the first inclusion in both lines of Equation \eqref{abeq1}, as
	$$
	a^{\prime \prime} + b^{\prime }= \max(k-n,0) + \min(n,k)= k= \max(k-m,0) + \min(m,k) = a^{\prime } + b^{\prime \prime }.
	$$\end{proof}   

\begin{lemma}\label{inclusionmeasures2indexkroen} Let $0\leq k \leq n+m$, $0\leq a \leq n$, and $0\leq b \leq m$. The following properties hold for nonzero measures $\lambda \in \mathcal{M}( \mathds{X}_{n} )$ and $\eta \in \mathcal{M}( \mathds{Y}_{m})$:
	\begin{enumerate}
		\item[$(i)$] $\lambda \times \eta \in \mathcal{M}_{a,b}( \mathds{X}_{n},\mathds{Y}_{m})$ if and only if $\lambda \in \mathcal{M}_{a }( \mathds{X}_{n} )$ and $\eta \in \mathcal{M}_{b }( \mathds{Y}_{m})$.
		\item[$(ii)$] If $\lambda \times \eta \in \mathcal{M}_{k}( \mathds{X}_{n} \times \mathds{Y}_{m})$, then $\lambda \in \mathcal{M}_{a^{\prime} }( \mathds{X}_{n} ) $ and $\eta \in \mathcal{M}_{b^{\prime} }( \mathds{Y}_{m})$. 
		\item[$(iii)$]     If $\lambda \in \mathcal{M}_{a^{\prime} }( \mathds{X}_{n} )$ and $\eta \in \mathcal{M}_{b^{\prime \prime }}( \mathds{Y}_{m})$, then $\lambda \times \eta \in \mathcal{M}_{k}( \mathds{X}_{n} \times \mathds{Y}_{m})$.
		\item[$(iv)$] If $\lambda \in \mathcal{M}_{a^{\prime \prime } }( \mathds{X}_{n} )$ and $\eta \in \mathcal{M}_{b^{\prime}}( \mathds{Y}_{m})$, then $\lambda \times \eta \in \mathcal{M}_{k}( \mathds{X}_{n} \times \mathds{Y}_{m})$.
	\end{enumerate}   
	where $a^{\prime}:=\max(k-m,0) $, $b^{\prime}:=\max(k-n,0) $, $a^{\prime \prime }:= \min(n,k)$, and $b^{\prime \prime }:= \min(m,k)$.\end{lemma}   

\begin{proof} 
	For subsets $A_{i} \subset X_{i}$ and $B_{j} \subset Y_{j}$, define $r:= |\{ i, \quad A_{i}=X_{i} \}| $ and $s:=|\{ j, \quad B_{j}=Y_{j} \}|$. If $\lambda \in \mathcal{M}_{a }( \mathds{X}_{n} )$, then $\lambda\left(\prod_{i=1}^{n}A_{i}\right)=0$ if $r\geq n-a +1$; analogously, if $\eta \in \mathcal{M}_{b }( \mathds{Y}_{m})$, then $\eta\left(\prod_{j=1}^{m}B_{j} \right)=0$ if $s \geq m-b+1$. Under these hypotheses, we also have that $\lambda\left(\prod_{i=1}^{n}A_{i}\right) \eta\left(\prod_{j=1}^{m}B_{j} \right)=0$ if either $r\geq n-a +1 $ or $s \geq m-b+1$. Conversely, as $\eta$ is nonzero, there must exist subsets $D_{j} \subset Y_{j}$ for which $\eta\left(\prod_{j=1}^{m}D_{j} \right) \neq 0$. Hence, if $\lambda \times \eta \in \mathcal{M}_{a,b}( \mathds{X}_{n},\mathds{Y}_{m})$, for subsets $A_{i} \subset X_{i}$ such that $ |\{ i, \quad A_{i}=X_{i} \}| \geq n-a+1 $, we have that $\lambda\left(\prod_{i=1}^{n}A_{i}\right) \eta\left(\prod_{j=1}^{m}D_{j} \right)=0$, which can only occur if $\lambda\left(\prod_{i=1}^{n}A_{i}\right)=0$. As the proof is similar for the other case, this concludes the proof of relation $(i)$.\\
	The remaining properties are obtained from relation $(i)$ and Equation \eqref{abeq1}.\\
	If $\lambda \times \eta \in \mathcal{M}_{k}( \mathds{X}_{n} \times \mathds{Y}_{m})$, then by Equation \eqref{abeq1} we obtain that $\lambda \times \eta \in \mathcal{M}_{a^{\prime}, b^{\prime}}( \mathds{X}_{n}, \mathds{Y}_{m})$, and relation $(i)$ concludes that $\lambda \in \mathcal{M}_{a^{\prime} }( \mathds{X}_{n} ) $ and $\eta \in \mathcal{M}_{b^{\prime} }( \mathds{Y}_{m})$.\\
	If $\lambda \in \mathcal{M}_{a^{\prime} }( \mathds{X}_{n} )$ and $\eta \in \mathcal{M}_{b^{\prime \prime }}( \mathds{Y}_{m})$, then by relation $(i)$ we obtain that $\lambda \times \eta \in \mathcal{M}_{a^{\prime}, b^{\prime \prime }}( \mathds{X}_{n}, \mathds{Y}_{m})$, and by Equation \eqref{abeq1} we conclude that $\lambda \times \eta \in \mathcal{M}_{k}( \mathds{X}_{n} \times \mathds{Y}_{m})$.\\
	The proof of relation $(iv)$ is similar to the proof of relation $(iii)$, and thus is omitted. 
\end{proof}

With these two results, we are able to prove the following characterization about the Kronecker product of PDI kernels.

\begin{theorem}\label{Kroengeralthm} Let $n,m \in \mathbb{N}$, $0\leq a \leq n$, $0\leq b\leq m$, and $0\leq k \leq n+m$. Given an $n$-symmetric kernel $\mathfrak{I}: \mathds{X}_{n} \times \mathds{X}_{n} \to \mathbb{R}$ and an $m$-symmetric kernel $\mathfrak{L}: \mathds{Y}_{m} \times \mathds{Y}_{m} \to \mathbb{R}$, consider the $(n+m)$-symmetric kernel $\mathfrak{I}\times \mathfrak{L} : [\mathds{X}_{n} \times \mathds{Y}_{m}] \times [\mathds{X}_{n} \times \mathds{Y}_{m}] \to \mathbb{R}$. Then, for  any nonzero $\mu \in \mathcal{M}_{a, b }( \mathds{X}_{n} \times \mathds{Y}_{m})$, we have that 
	$$
	\int_{\mathds{X}_{n} \times \mathds{Y}_{m}}    \int_{\mathds{X}_{n} \times \mathds{Y}_{m}} (-1)^{a+b}\mathfrak{I}(x_{\vec{1}}, x_{\vec{2}}) \mathfrak{L}(y_{\vec{1}}, y_{\vec{2}})d\mu(x_{\vec{1}}, y_{\vec{1}})d\mu(x_{\vec{2}}, y_{\vec{2}}) >0,
	$$
	if and only if for some $\ell \in \{0,1\}$, the kernel $(-1)^{\ell }\mathfrak{I}$ is SPDI$_{a}$ in $\mathds{X}_{n}$ and $(-1)^{\ell }\mathfrak{L}$ is SPDI$_{b}$ in $\mathds{Y}_{m}$.
\end{theorem}

\begin{proof}	For an arbitrary nonzero $\lambda \in \mathcal{M}_{a }( \mathds{X}_{n})$ and a fixed nonzero $\eta \in \mathcal{M}_{b }( \mathds{Y}_{m})$, the nonzero measure $\lambda \times \eta \in \mathcal{M}_{a, b}( \mathds{X}_{n} \times \mathds{Y}_{m}) $. By the hypothesis,
	$$
	\left [    \int_{\mathds{X}_{n} }    \int_{\mathds{X}_{n} } (-1)^{a}\mathfrak{I}(x_{\vec{1}}, x_{\vec{2}}) d\lambda(x_{\vec{1}} )d\lambda(x_{\vec{2}} ) \right ] \left [ \int_{ \mathds{Y}_{m}}    \int_{ \mathds{Y}_{m}} (-1)^{b}\mathfrak{L}(y_{\vec{1}}, y_{\vec{2}})d\eta( y_{\vec{1}})d\eta( y_{\vec{2}}) \right ]>0.
	$$
	Defining
	$$
	(-1)^{\ell}:= \operatorname{sgn} \left [ \int_{ \mathds{Y}_{m}}    \int_{ \mathds{Y}_{m}} (-1)^{b}\mathfrak{L}(y_{\vec{1}}, y_{\vec{2}})d\eta( y_{\vec{1}})d\eta( y_{\vec{2}})\right ] \neq 0, \quad \ell \in \{0,1\}
	$$
	we obtain that $(-1)^{\ell }\mathfrak{I} $ is SPDI$_{a}$ in $\mathds{X}_{n}$. By a similar argument, we obtain that $(-1)^{\ell}\mathfrak{L} $ is SPDI$_{b}$ in $\mathds{Y}_{m}$.\\
	Conversely, suppose that $(-1)^{\ell }\mathfrak{I} $ is SPDI$_{a}$ in $\mathds{X}_{n}$ and $(-1)^{\ell}\mathfrak{L} $ is SPDI$_{b}$ in $\mathds{Y}_{m}$. Without loss of generality, suppose that $\ell =0$. Note that for an arbitrary $x_{\vec{0}} \in \mathds{X}_{n}$, $y_{\vec{0}} \in \mathds{Y}_{m}$, and $\mu \in \mathcal{M}_{a, b}( \mathds{X}_{n} \times \mathds{Y}_{m})$,
	\begin{align*}
		&\int_{\mathds{X}_{n} \times \mathds{Y}_{m}}    \int_{\mathds{X}_{n} \times \mathds{Y}_{m}} (-1)^{a}\mathfrak{I}(x_{\vec{1}}, x_{\vec{2}}) (-1)^{b}\mathfrak{L}(y_{\vec{1}}, y_{\vec{2}})d\mu(x_{\vec{1}}, y_{\vec{1}})d\mu(x_{\vec{2}}, y_{\vec{2}})\\
		&=\int_{\mathds{X}_{n} \times \mathds{Y}_{m}}    \int_{\mathds{X}_{n} \times \mathds{Y}_{m}} K^{\mathfrak{I}}(x_{\vec{1}}, x_{\vec{2}}) (-1)^{b}\mathfrak{L}(y_{\vec{1}}, y_{\vec{2}})d\mu(x_{\vec{1}}, y_{\vec{1}})d\mu(x_{\vec{2}}, y_{\vec{2}})\\
		&=\int_{\mathds{X}_{n} \times \mathds{Y}_{m}}    \int_{\mathds{X}_{n} \times \mathds{Y}_{m}} K^{\mathfrak{I}}(x_{\vec{1}}, x_{\vec{2}}) K^{\mathfrak{L}}(y_{\vec{1}}, y_{\vec{2}})d\mu(x_{\vec{1}}, y_{\vec{1}})d\mu(x_{\vec{2}}, y_{\vec{2}}) \geq 0,
	\end{align*}
	due to Equation \eqref{integmuklnm}, the definition of $ K^{\mathfrak{I}}$ and $ K^{\mathfrak{L}}$ in Lemma \ref{pdi2topdn}, and the fact that the Kronecker product of PD kernels is PD. \\
	It only remains to prove that this double integral is zero only when $\mu$ is the zero measure in $\mathcal{M}_{a, b}( \mathds{X}_{n} \times \mathds{Y}_{m})$.\\
	Due to Equation \eqref{rkhsfunctionalzero} and Theorem \ref{initialextmmddominio}, for every function $g \in \mathcal{H}_{\mathfrak{L} }$, 
	$$
	\int_{\mathds{X}_{n}\times \mathds{Y}_{m}}\int_{\mathds{X}_{n}\times \mathds{Y}_{m}} K^{\mathfrak{I}}(x_{\vec{1}}, x_{\vec{2}}) g(y_{\vec{1}})g(y_{\vec{2}}) d\mu(x_{\vec{1}},y_{\vec{1}})d\mu(x_{\vec{2}},y_{\vec{2}})=0.
	$$
	For a fixed $g \in \mathcal{H}_{\mathfrak{L }} $, we define the measure
	\begin{equation}\label{Kroengeraleq2}
		\mu_{g}(A):= \int_{A \times \mathds{Y}_{m } }g(u)d\mu(x_{\vec{1}},y_{\vec{1}}), \quad A \subset \mathds{X}_{n}, 
	\end{equation}
	which is well defined, finite, and discrete. Also, $    \mu_{g} \in \mathcal{M}_{a}( \mathds{X}_{n})$ because $\mu \in \mathcal{M}_{a, b}( \mathds{X}_{n} \times \mathds{Y}_{m})$ and by its definition. Note that 
	\begin{align*}
		\int_{\mathds{X}_{n}}\int_{\mathds{X}_{n}}(-1)^{a}& \mathfrak{I} (x_{ \vec{1}}, x_{ \vec{2}})d\mu_{g}(x_{\vec{1}})d\mu_{g}(x_{\vec{2}})=\int_{\mathds{X}_{n}}\int_{\mathds{X}_{n}} K^{\mathfrak{I}}(x_{ \vec{1}}, x_{ \vec{2}})d\mu_{g}(x_{\vec{1}})d\mu_{g}(x_{\vec{2}})\\
		&= \int_{\mathds{X}_{n} \times \mathds{Y}_{m}}    \int_{\mathds{X}_{n} \times \mathds{Y}_{m}} K^{\mathfrak{I}}(x_{\vec{1}}, x_{\vec{2}}) g(y_{\vec{1}})g(y_{\vec{2}})d\mu(x_{\vec{1}}, y_{\vec{1}})d\mu(x_{\vec{2}}, y_{\vec{2}}) =0.
	\end{align*}
	However, as $\mu_{g} \in \mathcal{M}_{a} (\mathds{X}_{n})$ and $\mathfrak{I}$ is SPDI$_{a}$ in $\mathds{X}_{n}$, we obtain that $\mu_{g}$ is the zero measure for every $g \in \mathcal{H}_{\mathfrak{L} }$.\\
	For every fixed $A \subset \mathds{X}_{n} $, the following measure satisfies 
	$$
	\mu_{A}(B): = \mu(A\times B) \in \mathcal{M}_{b}(\mathds{Y}_{m} ). 
	$$
	If we pick $g(y_{\vec{1}}) = K^{\mathfrak{L}}(y_{\vec{1}},y_{\vec{2}}) $, for an arbitrary $y_{\vec{2}} \in \mathds{Y}_{m}$, by Equation \eqref{Kroengeraleq2}, we obtain that
	$$
	0 = \mu_{g}(A) = \int_{A \times \mathds{Y}_{m } } K^{\mathfrak{L}}(y_{\vec{1}},y_{\vec{2}})d\mu(x_{\vec{1}},y_{\vec{1}}) = \int_{ \mathds{Y}_{m } } K^{\mathfrak{L}}(y_{\vec{1}},y_{\vec{2}})d\mu_{A}(y_{\vec{1}}), 
	$$
	thus
	$$
	0= \int_{ \mathds{Y}_{m } } \int_{ \mathds{Y}_{m } } K^{\mathfrak{L}}(y_{\vec{1}},y_{\vec{2}})d\mu_{A}(y_{\vec{1}}) d\mu_{A}(y_{\vec{2}}) = \int_{ \mathds{Y}_{m } } \int_{ \mathds{Y}_{m } } (-1)^{b}\mathfrak{L}(y_{\vec{1}},y_{\vec{2}})d\mu_{A}(y_{\vec{1}}) d\mu_{A}(y_{\vec{2}}). 
	$$
	But, $\mathfrak{L}$ is SPDI$_{b}$ in $\mathds{Y}_{m}$, hence $\mu_{A}(B)=0$ for every $A \subset \mathds{X}_{n}$ and $B \subset \mathds{Y}_{m}$, thus concluding that $\mu$ is the zero measure, which concludes the proof.\\ 
\end{proof}
Theorem \ref{Kroengeralthm} makes explicit that the Kronecker product of SPDI kernels implies a behavior on this new spaces $ \mathcal{M}_{a, b }$ and not (directly) on a  the space $ \mathcal{M}_{k}$. From it, we obtain that the distance covariance is a very specific case, because as mentioned at the beginning of this section, $\mathcal{M}_{1,1}( \mathds{X},\mathds{Y})= \mathcal{M}_{2 }( \mathds{X}\times \mathds{Y})$. However, due to the lemmas proved before, we can obtain several interesting consequences.

\begin{corollary}\label{Kroengeral} Let $n,m \in \mathbb{N}$, $0\leq a \leq n$, $0\leq b\leq m$, and $0\leq k \leq n+m$. Given an $n$-symmetric kernel $\mathfrak{I}: \mathds{X}_{n} \times \mathds{X}_{n} \to \mathbb{R}$ and an $m$-symmetric kernel $\mathfrak{L}: \mathds{Y}_{m} \times \mathds{Y}_{m} \to \mathbb{R}$, consider the $(n+m)$-symmetric kernel $\mathfrak{I}\times \mathfrak{L} : [\mathds{X}_{n} \times \mathds{Y}_{m}] \times [\mathds{X}_{n} \times \mathds{Y}_{m}] \to \mathbb{R}$. Then:
	\begin{enumerate}
		
		\item[$(i)$] The kernel $\mathfrak{I}\times \mathfrak{L}$ is SPDI$_{k}$ in $\mathds{X}_{n} \times \mathds{Y}_{m}$ if and only if $(-1)^{a^{\prime}+k +\ell }\mathfrak{I}$ is SPDI$_{a^{\prime}}$ in $\mathds{X}_{n}$ and $(-1)^{b^{\prime}+\ell }\mathfrak{L}$ is SPDI$_{b^{\prime}}$ in $\mathds{Y}_{m}$, where $a^{\prime}:=\max(k-m,0) $ and $b^{\prime}:=\max(k-n,0) $ for some $\ell \in \{0, 1\}$.
		\item [$(ii)$] If $ (-1)^{\ell}\mathfrak{I}$ is SPDI$_{a}$ in $\mathds{X}_{n}$ and $ (-1)^{\ell}\mathfrak{L}$ is SPDI$_{b}$ in $\mathds{Y}_{m}$ for some $\ell \in \{0,1\}$, 
		then $(-1)^{k^{\prime} + a +b}\mathfrak{I}\times \mathfrak{L}$ is SPDI$_{k^{\prime}}$ in $\mathds{X}_{n} \times \mathds{Y}_{m}$, where $k^{\prime}:=\max((a+m)\operatorname{sgn}(a), (n+b)\operatorname{sgn}(b))$. The converse holds when $(a+m)\operatorname{sgn}(a)=(n+b)\operatorname{sgn}(b)$.
		\item[$(iii)$] For $k\leq \min\{n,m\}$, the kernel $\mathfrak{I}\times \mathfrak{L}$ is SPDI$_{k}$ in $\mathds{X}_{n} \times \mathds{Y}_{m}$ if and only if $(-1)^{k }\mathfrak{I} \times \mathfrak{L}$ is SPD in $\mathds{X}_{n} \times \mathds{Y}_{m}$, which is also equivalent to $(-1)^{k +\ell }\mathfrak{I}$ being SPD in $\mathds{X}_{n}$ and $(-1)^{\ell }\mathfrak{L}$ being SPD in $\mathds{Y}_{m}$ for some $\ell \in \{0, 1\}$.
	\end{enumerate}   
\end{corollary}

\begin{proof} To prove relation $(i)$, if $(-1)^{a^{\prime}+k +\ell }\mathfrak{I}$ is SPDI$_{a^{\prime}}$ in $\mathds{X}_{n}$ and $(-1)^{b^{\prime} +\ell }\mathfrak{L}$ is SPDI$_{b^{\prime}}$ in $\mathds{Y}_{m}$, then by Theorem \ref{Kroengeralthm}  
	$$
	\int_{\mathds{X}_{n} \times \mathds{Y}_{m}}    \int_{\mathds{X}_{n} \times \mathds{Y}_{m}} (-1)^{k}\mathfrak{I}(x_{\vec{1}}, x_{\vec{2}}) \mathfrak{L}(y_{\vec{1}}, y_{\vec{2}})d\mu(x_{\vec{1}}, y_{\vec{1}})d\mu(x_{\vec{2}}, y_{\vec{2}}) >0, 
	$$
	for every nonzero $\mu \in \mathcal{M}_{k}( \mathds{X}_{n}\times \mathds{Y}_{m}) \subset \mathcal{M}_{a^{\prime}, b^{\prime}}( \mathds{X}_{n},\mathds{Y}_{m})$,  due to Equation \eqref{abeq1}. \\
	Conversely, if $\mathfrak{I}\times \mathfrak{L}$ is SPDI$_{k}$ in $\mathds{X}_{n} \times \mathds{Y}_{m}$, by the first inclusion in the first line of Equation \eqref{abeq1} and relation $(iii)$ in Lemma \ref{inclusionmeasures2indexkroen}, we find that for any nonzero $\lambda \in \mathcal{M}_{a^{\prime} }( \mathds{X}_{n} )$ and a fixed nonzero $\eta \in \mathcal{M}_{b^{\prime \prime }}( \mathds{Y}_{m})$, it holds that
	$$
	\left [ (-1)^{k}    \int_{\mathds{X}_{n} } \mathfrak{I}(x_{\vec{1}}, x_{\vec{2}}) d\lambda(x_{\vec{1}} )d\lambda(x_{\vec{2}} ) \right ] \left [ \int_{ \mathds{Y}_{m}}    \int_{ \mathds{Y}_{m}} \mathfrak{L}(y_{\vec{1}}, y_{\vec{2}})d\eta( y_{\vec{1}})d\eta( y_{\vec{2}}) \right ]>0.
	$$
	Defining
	$$
	(-1)^{\ell}:= \operatorname{sgn} \left [ \int_{ \mathds{Y}_{m}}    \int_{ \mathds{Y}_{m}} \mathfrak{L}(y_{\vec{1}}, y_{\vec{2}})d\eta( y_{\vec{1}})d\eta( y_{\vec{2}})\right ] \neq 0, \quad \ell \in \{0,1\}
	$$
	we obtain that 
	$$
	\int_{\mathds{X}_{n} } (-1)^{a^{\prime}}\left [(-1)^{k + a^{\prime} + \ell}\mathfrak{I}(x_{\vec{1}}, x_{\vec{2}})\right ] d\lambda(x_{\vec{1}} )d\lambda(x_{\vec{2}} ) >0,
	$$
	for any nonzero $\lambda \in \mathcal{M}_{a^{\prime} }( \mathds{X}_{n} )$. Using the first inclusion in the second line of Equation \eqref{abeq1}, we also find that $(-1)^{b^{\prime}+\ell }\mathfrak{L}$ is SPDI$_{b^{\prime}}$ in $\mathds{Y}_{m}$.\\
	Now we prove relation $(ii)$. By the hypothesis, Theorem \ref{Kroengeralthm}, and the first inclusion in Equation \eqref{abeq2}, for any nonzero $\mu \in \mathcal{M}_{k^{\prime} }( \mathds{X}_{n} \times \mathds{Y}_{m})$ we have that 
	$$
	\int_{\mathds{X}_{n} \times \mathds{Y}_{m}}    \int_{\mathds{X}_{n} \times \mathds{Y}_{m}} (-1)^{k^{\prime}}\left [(-1)^{k^{\prime}+a+b} \mathfrak{I}(x_{\vec{1}}, x_{\vec{2}}) \mathfrak{L}(y_{\vec{1}}, y_{\vec{2}}) \right ]d\mu(x_{\vec{1}}, y_{\vec{1}})d\mu(x_{\vec{2}}, y_{\vec{2}}) >0, 
	$$
	thus the kernel $(-1)^{k^{\prime} + a +b}\mathfrak{I}\times \mathfrak{L}$ is SPDI$_{k^{\prime}}$ in $\mathds{X}_{n} \times \mathds{Y}_{m}$. For the converse, suppose that $(-1)^{k^{\prime} + a +b}\mathfrak{I}\times \mathfrak{L}$ is SPDI$_{k^{\prime}}$ in $\mathds{X}_{n} \times \mathds{Y}_{m}$. If $ (a+m)\operatorname{sgn}(a)= (n+b)\operatorname{sgn}(b)$, then $\max(k^{\prime}-m,0 )=a$ and $\max(k^{\prime}-n,0 )=b$, and the conclusion is a consequence of relation $(i)$.\\
	To conclude, relation $(iii)$ is a direct consequence of relation $(i)$, as by it we obtain that $\mathfrak{I}\times \mathfrak{L}$ is SPDI$_{k}$ in $\mathds{X}_{n} \times \mathds{Y}_{m}$ if and only if $(-1)^{k +\ell }\mathfrak{I}$ is SPDI$_{0}$ (that is, SPD) in $\mathds{X}_{n}$ and $(-1)^{\ell }\mathfrak{L}$ is SPDI$_{0}$ in $\mathds{Y}_{m}$, for some $\ell \in \{0, 1\}$, which is equivalent to $\mathfrak{I}\times \mathfrak{L}$ being SPD in $\mathds{X}_{n} \times \mathds{Y}_{m}$ (see Lemma 7.7 in \cite{Guella2022a}).
\end{proof}

As a direct consequence of Corollary \ref{Kroengeral}, we obtain a characterization of when a general Kronecker product of kernels is SPDI$_{2}$, which are in particular independence tests for discrete probability measures. Surprisingly, there are not many possibilities.

\begin{corollary} \label{classkroenne} Let $n,\ell \geq 2 $ and consider a disjoint family of subsets $F^{1}, \ldots, F^{\ell}$ of $\{1,\ldots, n\}$ whose union is the entire set and $|F^{1}|\geq\ldots \geq |F^{\ell}| $. Given $|F^{i}|$-symmetric kernels $\mathfrak{I}_{i} : \mathds{X}_{F^{i}} \times \mathds{X}_{F^{i}}\to \mathbb{R} $, $1\leq i \leq \ell$, the kernel 
	$$
	\mathfrak{I}(x_{\vec{1}}, x_{\vec{2}}):= \prod_{i=1}^{\ell} \mathfrak{I}_{i}( x_{\vec{ 1}_{F^{i}} }, x_{\vec{2}_{F^{i}}})
	$$
	is SPDI$_{2}$ if and only if 
	\begin{enumerate}
		\item[$(i)$] $\ell >2$: For every $1\leq i \leq \ell$, the kernel $(-1)^{a_{i}}\mathfrak{I}_{i}$ is SPD for some $a_{i}\in \{0,1\}$ such that $\sum_{i=1}^{\ell}a_{i} \in 2\mathbb{N}$.
		\item[$(ii)$] $\ell=2=n $: For every $1\leq i \leq 2$, the kernel $(-1)^{a}\mathfrak{I}_{i}$ is strictly conditionally negative definite for some $a \in \{0,1\}$. 
		\item[$(iii)$] $\ell=2< n $ and $|F^{2}|=1$: The kernel $(-1)^{a +1}\mathfrak{I}_{1}$ is strictly conditionally negative definite and the kernel $(-1)^{a }\mathfrak{I}_{2}$ is SPD for some $a \in \{0,1\}$.
		\item[$(iv)$] $\ell=2< n $ and $|F^{2}|\geq 2$: For every $1\leq i \leq 2$, the kernel $(-1)^{a }\mathfrak{I}_{i}$ is SPD for some $a \in \{0,1\}$.
	\end{enumerate}
\end{corollary}

\begin{proof}
	If $\ell>2$, then $n-|F^{k}|\geq 2$ for any possible $k$; so, if we apply relation $(ii)$ in Theorem \ref{Kroengeral} on $\mathfrak{I}_{i}$ and $\mathfrak{L}=\prod_{\ell \neq i} \mathfrak{I}_{\ell}$, we get that $(-1)^{a_i}\mathfrak{I}_{i}$ should be SPD because $a^{\prime}= \max(2 -\sum_{\ell \neq i}|F^{\ell}|, 0 )=0$, for some $a_{i} \in \{0,1\}$. The sum of these constants $a_{i}$ must be an even number because the Kronecker product of SPD kernels is SPD.\\
	The other $3$ cases are a direct application of relation $(ii)$ in Theorem \ref{Kroengeral}. 
\end{proof}

We emphasize that we cannot affirm in general that for the kernels given in Corollary \ref{classkroenne}, it is equivalent to be an independence test for discrete probability measures in $\mathds{X}_{n}$ and to be an SPDI$_{2}$ kernel. However, we expect that for the most commonly used families of kernels, these two properties are equivalent.

The special case of the Kronecker product of $n$ PD kernels is extensively studied in the literature. When $n \geq 3$, independence tests based on them are denoted as $\ell$-variable Hilbert-Schmidt independence criterion (dHSIC), see \cite{Pfister2018}. Case $(ii)$ is the standard method of distance covariance \cite{Sejdinovic2013a}. Partial proofs of this corollary can be found in \cite{Szabo2018}, but more generally in Theorem 7.1 in \cite{guella2023}, where it is proved that in the case where $n=\ell$, the Kronecker product of kernels is SPDI$_{2}$ if and only if it defines an independence test, and the same equivalence occurs in the continuous case. 

In the next corollary, we completely characterize what type of independence tests can be done on the Kronecker product of $n$ kernels on an $n$-fold Cartesian product.  

\begin{corollary}\label{productnkernels} Let $\gamma_{i}: X_{i} \times X_{i} \to \mathbb{R}$, $1\leq i \leq n$, be symmetric kernels. Define the complete $n$-symmetric kernel
	\begin{align*}
		(x_{\vec{1}}, x_{\vec{2}}) \in \mathds{X}_{n}\times\mathds{X}_{n} \to \vec{\gamma} (x_{\vec{1}}, x_{\vec{2}}) := \prod_{i=1}^{n}\gamma_{i}(x_{i}^{1}, x_{i}^{2}),
	\end{align*}
	the following conditions are equivalent for $2\leq k \leq n$:
	\begin{enumerate}
		\item [$(i)$] The kernel $\vec{\gamma}$ is SPDI$_{k}$.
		\item[$(ii)$] For every discrete probability measure $P $ for which $\Lambda_{k}^{n}[P]$ is not the zero measure: 
		$$
		\int_{\mathds{X}_{n}}\int_{\mathds{X}_{n}} (-1)^{k} \vec{\gamma} (x_{\vec{1}}, x_{\vec{2}})d\Lambda_{k}^{n}[P](x_{\vec{1}})d\Lambda_{k}^{n}[P](x_{\vec{2}})>0.
		$$
		\item[$(ii^{\prime})$] When $n=k$, for every discrete probability measure $P $ for which $\Sigma[P]$ is not the zero measure: 
		$$
		\int_{\mathds{X}_{n}}\int_{\mathds{X}_{n}} (-1)^{k}\vec{\gamma} (x_{\vec{1}}, x_{\vec{2}})d\Sigma[P](x_{\vec{1}})d\Sigma[P](x_{\vec{2}})>0.
		$$
		\item[$(iii)$] The kernel $(-1)^{k}\vec{\gamma}$ is SPD when $n >k $. The kernel $ \vec{\gamma}$ is SPDI$_{n}$ when $n=k$.
		\item[$(iv)$] The kernel $(-1)^{a_{i}}\gamma_{i}$, $1\leq i \leq n$, is SPD for some $a_{i}\in \{0,1\}$ such that $k-\sum_{i}a_{i} \in 2\mathbb{Z}$ when $n>k$. The kernel $(-1)^{a_{i}}\gamma_{i}$, $1\leq i \leq n$, is SCND for some $a_{i}\in \{0,1\}$ such that $\sum_{i}a_{i} \in 2\mathbb{N}$ when $n=k$.   
	\end{enumerate}            
\end{corollary}   

\begin{proof} 
	Relation $(iii)$ and relation $(iv)$ are equivalent by recursively using relation $(ii)$ in Theorem \ref{Kroengeral}. \\
	Relation $(i)$ directly implies relation $(ii)$ and relation $(ii^{\prime})$ by Theorem \ref{generallancaster} and Lemma \ref{exm02xn}.\\
	Relation $(iii)$ implies relation $(i)$ because $ \mathcal{M}_{k}(\mathds{X}_{n}) \subset \mathcal{M}(\mathds{X}_{n}) $. \\
	We conclude by showing that relation $(ii)$ and relation $(ii^{\prime})$ imply relation $(iv)$. Indeed, suppose that $n>k$. By Lemma \ref{genlancastercartesianproduct} and Remark \ref{hanhjordanequivalence}, for every nonzero $\mu_{i} \in \mathcal{M}(X_{i})$ and fixed nonzero $\mu_{j} \in \mathcal{M}_{1}(X_{j})$ for $j\neq i $, the measure $\mu = \times _{l=1}^{n}\mu_{l}$ is a multiple of the Lancaster interaction of some discrete probability measure in $\mathds{X}_{n}$. By the hypothesis,
	\begin{align*}
		(-1)^{k}&\prod_{l=1}^{n}\left [ \int_{ X_{l}}\int_{X_{l}} \gamma_{l}(x_{l}^{1}, x_{l}^{2})d\mu_{l}(x_{l}^{1})d\mu_{l}(x_{l}^{2}) \right ] = \int_{\mathds{X}_{n}}\int_{\mathds{X}_{n}} (-1)^{k} \vec{\gamma} (x_{\vec{1}}, x_{\vec{2}})d\mu(x_{\vec{1}})d\mu(x_{\vec{2}})>0. 
	\end{align*}
	Then, if 
	$$
	a_{i }:= \operatorname{sgn} \left [ (-1)^{k} \prod_{l\neq i}\int_{ X_{l}}\int_{X_{l}} \gamma_{l}(x_{l}^{1}, x_{l}^{2})d\mu_{l}(x_{l}^{1})d\mu_{l}(x_{l}^{2}) \right ] 
	$$     
	we obtain that $(-1)^{a_{i}}\gamma_{i}$ is SPD, as $\mu_{i} \in \mathcal{M}(X_{i})$ is arbitrary. It must hold that $k-\sum_{i}a_{i} \in 2\mathbb{Z}$ because      
	$$
	(-1)^{k}\vec{\gamma} = (-1)^{k- \sum_{i}a_{i}} \left [ \prod_{i=1}^{n} (-1)^{a_{i}}\gamma_{i} \right ] = \prod_{i=1}^{n} (-1)^{a_{i}}\gamma_{i},
	$$
	and the classical result that the Kronecker product of SPD kernels is SPD.\\
	The case where $n=k$ is proved similarly, but we pick all measures $\mu_{i} \in \mathcal{M}_{1}(X_{i})$.
\end{proof}   

Corollary \ref{productnkernels} still holds true in the case $k=0$ or $k=1$, if we remove relations $(ii)$ and $(ii^{\prime})$, which are not defined in these cases. 

We emphasize that for the kernels in Corollary \ref{productnkernels}, being an independence test for the discrete Lancaster Interaction $\Lambda_{k}^{n}$ (or the Streitberg Interaction $\Sigma$) is equivalent to being SPDI$_{k}$. 

\subsection{The continuous case}\label{continuouspdikkron}

In this subsection, we prove a version of Theorem \ref{Kroengeral} for the continuous case. First, we obtain a description for the integrability restrictions for the Kronecker product of PDI kernels.

\begin{lemma}\label{measkroencontrestric} Let $\mathfrak{I} : \mathds{X}_{n} \times \mathds{X}_{n} \to \mathbb{R}$ be a continuous complete $n$-symmetric PDI$_{a}$ kernel that is zero at the extended diagonal $\Delta_{a-1}^{n}$ of $\mathds{X}_{n}$, and let $\mathfrak{L} : \mathds{Y}_{m} \times \mathds{Y}_{m} \to \mathbb{R}$ be a continuous complete $m$-symmetric PDI$_{b}$ kernel that is zero at the extended diagonal $\Delta_{b-1}^{m}$ of $\mathds{Y}_{m}$. The following conditions are equivalent for a measure $\mu \in \mathfrak{M}(\mathds{X}_{n} \times \mathds{Y}_{m})$:
	\begin{enumerate}
		\item [$(i)$]The function 
		$$
		((x_{\vec{1}}, y_{\vec{1}}), (x_{\vec{2}}, y_{\vec{2}})) \in (\mathds{X}_{n}\times \mathds{Y}_{m} ) \times (\mathds{X}_{n} \times \mathds{Y}_{m}) \to \mathfrak{I}(x_{\vec{1} },x_{\vec{2}})\mathfrak{L}(y_{\vec{1}}, y_{\vec{2}}) \in \mathbb{R}
		$$
		is FIP with respect to $\mu\times \mu \in \mathfrak{M}((\mathds{X}_{n} \times \mathds{Y}_{m})\times (\mathds{X}_{n} \times \mathds{Y}_{m}))$.
		\item[$(ii)$] There exists an element $(x_{\vec{3}}, y_{\vec{3}}) \in \mathds{X}_{n}\times \mathds{Y}_{m}$ such that the function 
		$$ 
		(x_{\vec{1}}, y_{\vec{1}}) \in \mathds{X}_{n}\times \mathds{Y}_{m} \to \mathfrak{I}(x_{\vec{1} },x_{\vec{3}})\mathfrak{L}(y_{\vec{1}}, y_{\vec{3}}) \in \mathbb{R}
		$$
		is FIP with respect to $\mu \in \mathfrak{M}(\mathds{X}_{n}\times \mathds{Y}_{m})$.
		\item[$(iii)$] For every element $(x_{\vec{3}}, y_{\vec{3}}) \in \mathds{X}_{n}\times \mathds{Y}_{m}$, the function
		$$ 
		(x_{\vec{1}}, y_{\vec{1}}) \in \mathds{X}_{n}\times \mathds{Y}_{m} \to \mathfrak{I}(x_{\vec{1} },x_{\vec{3}})\mathfrak{L}(y_{\vec{1}}, y_{\vec{3}}) \in \mathbb{R}
		$$
		is FIP with respect to $\mu \in \mathfrak{M}(\mathds{X}_{n}\times \mathds{Y}_{m})$.
		\item[$(iv)$] For whichever $x_{\vec{0}} \in \mathds{X}_{n}$ and $y_{\vec{0}} \in  \mathds{Y}_{m}$ that is used to define $K^{\mathfrak{I}}$ and $K^{\mathfrak{L}}$ defined in Lemma \ref{pdi2topdn}, the function 
		$$ 
		(x_{\vec{1}}, y_{\vec{1}}) \in \mathds{X}_{n} \to K^{\mathfrak{I}}(x_{\vec{1}}, x_{\vec{1}})K^{\mathfrak{L}}(y_{\vec{1}}, y_{\vec{1}}) \in \mathbb{R}
		$$
		is FIP with respect to $\mu \in \mathfrak{M}(\mathds{X}_{n} \times \mathds{Y}_{m})$. 
	\end{enumerate}        
	Additionally, a measure $\mu$ that satisfies these 4 equivalent properties also satisfies the 4 equivalences of Lemma \ref{integrestrik} (or Lemma \ref{probintegPDIn}) for the PDI$_{k^{\prime}}$ kernel $\mathfrak{T}= (\mathfrak{I}\times \mathfrak{L})^{\prime }$ defined in Corollary \ref{complnsimzerodiag}, where $k^{\prime}= \max((a+m)\operatorname{sgn}(a), (n+b)\operatorname{sgn}(b))$. 
\end{lemma}   
\begin{proof}
	We focus the first part of the proof on the case $a=n$ and $b=m$, because the other cases, similar to the proof of Corollary \ref{integrestrik}, rely on Equation \eqref{ineqpdiknprin} and the arguments of Lemma \ref{probintegPDIn}. However, the case $a=n$ and $b=m$ already fits Lemma \ref{probintegPDIn}, because by adapting the converse of relation $(i)$ in Theorem \ref{Kroengeral} to the PDI instead of SPDI setting, we obtain that $\mathfrak{I}\times \mathfrak{L}$ is a PDI$_{n+m}$ kernel in $\mathds{X}_{n} \times \mathds{Y}_{m}$. Also, this kernel is zero at the extended diagonal $\Delta_{n+m-1}^{n+m}$ of $\mathds{X}_{n} \times \mathds{Y}_{m}$. \\
	For the second part, we point out that by relation $(ii)$ in Lemma \ref{propertiesFIP}, if $\mu$ satisfies relation $(i)$, then all functions that appear in the definition of $\mathfrak{T}$ are also FIP with respect to $\mu\times \mu$. \end{proof}

The set of measures that satisfies Lemma \ref{measkroencontrestric} is denoted as $\mathfrak{M}( \mathfrak{I}, \mathfrak{L} )$. Note that $\mathfrak{M}( \mathfrak{I} ) \times \mathfrak{M}( \mathfrak{L} )\subset \mathfrak{M}( \mathfrak{I}, \mathfrak{L} ) $. As usual, we have the sets $\mathfrak{M}_{k}( \mathfrak{I}, \mathfrak{L} )$ and $\mathfrak{M}_{a,b}( \mathfrak{I}, \mathfrak{L} )$.

The inclusion $\mathfrak{M}( \mathfrak{I}, \mathfrak{L} ) \subset \mathfrak{M}( (\mathfrak{I}\times \mathfrak{L} )^{\prime})$, proved in Lemma \ref{measkroencontrestric}, may be  strict. For that, suppose that for $ 1\leq i \leq a $ we have CND kernels $\gamma_{i}^{1}: X_{i} \times X_{i} \to \mathbb{R}$, and for $ 1\leq j \leq b$ we have CND kernels $\gamma_{j}^{2}: Y_{j} \times Y_{j} \to \mathbb{R}$, and all those kernels are zero at the diagonal. By the examples presented after Definition \ref{PDI2}, the complete $n$-symmetric kernel $\mathfrak{I}(x_{\vec{1}}, x_{\vec{2}}) = \prod_{i=1}^{a} \gamma_{i}^{1}(x_{i}^{1}, x_{i}^{2})$ is PDI$_{a}$ and is zero at the extended diagonal $\Delta_{a-1}^{n}$ of $\mathds{X}_{n}$. Similarly, the complete $m$-symmetric kernel $\mathfrak{L}(y_{\vec{1}}, y_{\vec{2}}) = \prod_{j=1}^{b} \gamma_{j}^{2}(y_{j}^{1}, y_{j}^{2})$ is PDI$_{b}$ and is zero at the extended diagonal $\Delta_{b-1}^{m}$ of $\mathds{Y}_{m}$. By adapting relation $(ii)$ in Corollary \ref{Kroengeral}, we can obtain that the kernel $\mathfrak{I} \times \mathfrak{L}$ is PDI$_{k^{\prime}}$ on $X_{n} \times Y_{m}$, where $k^{\prime}=\max((a+m)\operatorname{sgn}(a), (n+b)\operatorname{sgn}(b))$. Naturally, the set $\mathfrak{M}( \mathfrak{I}, \mathfrak{L} )$ depends on the growth of the CND kernels that define it. 

However, we affirm that $(\mathfrak{I}\times \mathfrak{L})^{\prime}$ is the zero kernel, unless $a+b = k^{\prime}$, which only occurs if either $a=b=0$ or $a=n$ and $b=m$, and in these cases $(\mathfrak{I}\times \mathfrak{L})^{\prime} = \mathfrak{I}\times \mathfrak{L}$. Indeed, when $a+ b < k^{\prime} $ (the opposite inequality does not occur), then by Equation \eqref{integmu0n}, for any $\mu \in \mathcal{M}_{k^{\prime}}(\mathds{X}_{n}\times \mathds{Y}_{m}) $
\begin{align*}
	0&= \int_{\mathds{X}_{n} \times \mathds{Y}_{m}}    \int_{\mathds{X}_{n} \times \mathds{Y}_{m}} (-1)^{k}\mathfrak{I}(x_{\vec{1}}, x_{\vec{2}}) \mathfrak{L}(y_{\vec{1}}, y_{\vec{2}})d\mu(x_{\vec{1}}, y_{\vec{1}})d\mu(x_{\vec{2}}, y_{\vec{2}})\\
	& = \int_{\mathds{X}_{n} \times \mathds{Y}_{m}}    \int_{\mathds{X}_{n} \times \mathds{Y}_{m}} (-1)^{k}[\mathfrak{I} \times \mathfrak{L}]^{\prime} ((x_{\vec{1}}, y_{\vec{1}}),(x_{\vec{2}}, y_{\vec{2}}))d\mu(x_{\vec{1}}, y_{\vec{1}})d\mu(x_{\vec{2}}, y_{\vec{2}}), 
\end{align*}
and the conclusion comes from the second part of Theorem \ref{ineqpdik}. We remark that a combinatorial proof using the explicit expression of $[\mathfrak{I} \times \mathfrak{L}]^{\prime}$ in Lemma \ref{PDI2simpli} is also possible. 

Taking into account this behavior and also the consequences of being zero at some extended diagonal (see the comment made before Subsection \ref{Integrabilityrestrictionsnk}), we present a version of Theorem  \ref{Kroengeralthm} and Corollary \ref{Kroengeral} for the continuous case, with the necessary adaptations.

\begin{theorem}\label{Kroenconti} Let $\mathfrak{I} : \mathds{X}_{n} \times \mathds{X}_{n} \to \mathbb{R}$ be a continuous complete $n$-symmetric PDI$_{a}$ kernel that is zero at the extended diagonal $\Delta_{a-1}^{n}$ of $\mathds{X}_{n}$, and let $\mathfrak{L} : \mathds{Y}_{m} \times \mathds{Y}_{m} \to \mathbb{R}$ be a continuous complete $m$-symmetric PDI$_{b}$ kernel that is zero at the extended diagonal $\Delta_{b-1}^{m}$ of $\mathds{Y}_{m}$. 
	\begin{enumerate}
		\item[$(i)$] For any nonzero measure $
		\mu \in \mathfrak{M}_{a,b}( \mathfrak{I}, \mathfrak{L} )$, we have that 
		$$
		\int_{\mathds{X}_{n} \times \mathds{Y}_{m}}    \int_{\mathds{X}_{n} \times \mathds{Y}_{m}} (-1)^{a+b}\mathfrak{I}(x_{\vec{1}}, x_{\vec{2}}) \mathfrak{L}(y_{\vec{1}}, y_{\vec{2}})d\mu(x_{\vec{1}}, y_{\vec{1}})d\mu(x_{\vec{2}}, y_{\vec{2}}) >0,
		$$
		if and only if for some $\ell \in \{0,1\}$, the kernel $(-1)^{\ell }\mathfrak{I}$ is PDI$_{a}$-Characteristic in $\mathds{X}_{n}$ and $(-1)^{\ell }\mathfrak{L}$ is PDI$_{b}$-Characteristic in $\mathds{Y}_{m}$.
		\item[$(ii)$] For every nonzero measure $ \mu \in \mathfrak{M}_{k}( \mathfrak{I}, \mathfrak{L} )$, we have that 
		$$
		\int_{\mathds{X}_{n} \times \mathds{Y}_{m}}    \int_{\mathds{X}_{n} \times \mathds{Y}_{m}} (-1)^{k}\mathfrak{I}(x_{\vec{1}}, x_{\vec{2}}) \mathfrak{L}(y_{\vec{1}}, y_{\vec{2}})d\mu(x_{\vec{1}}, y_{\vec{1}})d\mu(x_{\vec{2}}, y_{\vec{2}}) >0,
		$$
		if and only if for some $\ell \in \{0, 1\}$
		$$
		\int_{\mathds{X}_{n} }    \int_{\mathds{X}_{n} } (-1)^{k +\ell}\mathfrak{I}(x_{\vec{1}}, x_{\vec{2}}) d\mu_{1}(x_{\vec{1}})d\mu_{1}(x_{\vec{2}}) >0,
		$$
		for every nonzero measure $ \mu_{1}\in \mathfrak{M}_{a^{\prime}}( \mathfrak{I} )$, where $a^{\prime}:=\max(k-m,0) $, and 
		$$
		\int_{\mathds{Y}_{m} }    \int_{\mathds{Y}_{m} } (-1)^{ \ell}\mathfrak{L}(y_{\vec{1}}, y_{\vec{2}}) d\mu_{2}(y_{\vec{1}})d\mu_{2}(y_{\vec{2}}) >0,
		$$
		for every nonzero measure $ \mu_{2}\in \mathfrak{M}_{b^{\prime}}( \mathfrak{L} )$, where $b^{\prime}:=\max(k-n,0) $.
		\item [$(iii)$]If $ \mathfrak{I}$ is PDI$_{a}$-Characteristic in $\mathds{X}_{n}$ and $ \mathfrak{L}$ is PDI$_{b}$-Characteristic in $\mathds{Y}_{m}$,     then for every nonzero measure $ \mu \in \mathfrak{M}_{k^{\prime}}( \mathfrak{I}, \mathfrak{L} )$ we have that 
		$$
		\int_{\mathds{X}_{n} \times \mathds{Y}_{m}}    \int_{\mathds{X}_{n} \times \mathds{Y}_{m}} (-1)^{a+b}\mathfrak{I}(x_{\vec{1}}, x_{\vec{2}}) \mathfrak{L}(y_{\vec{1}}, y_{\vec{2}})d\mu(x_{\vec{1}}, y_{\vec{1}})d\mu(x_{\vec{2}}, y_{\vec{2}}) >0,
		$$
		where $k^{\prime}:=\max((a+m)\operatorname{sgn}(a), (n+b)\operatorname{sgn}(b))$.
		\item [$(iv)$] For $k\leq \min\{n,m\}$, we have that 
		$$
		\int_{\mathds{X}_{n} \times \mathds{Y}_{m}}    \int_{\mathds{X}_{n} \times \mathds{Y}_{m}} (-1)^{a+b}\mathfrak{I}(x_{\vec{1}}, x_{\vec{2}}) \mathfrak{L}(y_{\vec{1}}, y_{\vec{2}})d\mu(x_{\vec{1}}, y_{\vec{1}})d\mu(x_{\vec{2}}, y_{\vec{2}}) >0,
		$$
		for every nonzero measure $ \mu \in \mathfrak{M}_{k}( \mathfrak{I}, \mathfrak{L} )$ if and only if for some $\ell \in \{0,1\}$ it occurs that     for every $ \mu_{1}\in \mathfrak{M} ( \mathfrak{I} )$ and $ \mu_{2}\in \mathfrak{M}( \mathfrak{L} )$,
		$$
		\int_{\mathds{X}_{n} }    \int_{\mathds{X}_{n} } (-1)^{k +\ell}\mathfrak{I}(x_{\vec{1}}, x_{\vec{2}}) d\mu_{1}(x_{\vec{1}})d\mu_{1}(x_{\vec{2}}) >0,
		$$
		$$
		\int_{\mathds{Y}_{m} }    \int_{\mathds{Y}_{m} } (-1)^{ \ell}\mathfrak{L}(y_{\vec{1}}, y_{\vec{2}}) d\mu_{2}(y_{\vec{1}})d\mu_{2}(y_{\vec{2}}) >0.
		$$
	\end{enumerate}   
\end{theorem}

\begin{proof} The proof follows the same steps as the proofs of  Theorem  \ref{Kroengeralthm} and Corollary \ref{Kroengeral}, but we must verify if the necessary integrability restrictions are fulfilled.\\
	We focus on relation $(i)$. Note that $ \mathfrak{M}_{a}( \mathfrak{I} ) \times \mathfrak{M}_{b}( \mathfrak{L} )\subset \mathfrak{M}_{a,b}( \mathfrak{I}, \mathfrak{L} ) $. For the converse relation, first, we show that the measure     \begin{equation}\label{Kroengeralconteq2}
		\mu_{g}(A):= \int_{A \times \mathds{Y}_{m } }g(u)d\mu(x_{\vec{1}},y_{\vec{1}}), \quad A \subset \mathscr{B}(\mathds{X}_{n}), 
	\end{equation}
	is an element of $\mathfrak{M}_{a}(\mathfrak{I})$. Indeed, by relation $(iv)$ in Corollary \ref{integrestrik} and Theorem \ref{FIPproduct}, to show that $\mu_{g} \in \mathfrak{M}(\mathfrak{I})$, it is sufficient to prove that the function 
	$$
	(x_{\vec{1}}, y_{\vec{1}}) \to K^{\mathfrak{I}}(x_{\vec{1}}, x_{\vec{1}}) g^{2}( y_{\vec{1}}) \in \mathbb{R}
	$$
	is FIP with respect to $\mu$. By the reproducing property,
	$$
	0 \leq g^{2}( y_{\vec{1}})\leq \left[\|g\|_{\mathcal{H}^{\mathfrak{L}}} \sqrt{K^{\mathfrak{L}}(y_{\vec{1}}, y_{\vec{1}})}\right]^2 = (\|g\|_{\mathcal{H}^{\mathfrak{L}}})^{2} K^{\mathfrak{L}}(y_{\vec{1}}, y_{\vec{1}}), 
	$$
	and relation $(iii)$ in Lemma \ref{propertiesFIP} concludes our claim. Also, $\mu_{g} \in \mathfrak{M}_{a}(\mathds{X}_{n})$ because $\mu \in \mathfrak{M}_{a,b}(\mathds{X}_{n},\mathds{Y}_{m} )$ and by its definition.   
\end{proof}

Note that we cannot affirm that $(-1)^{\ell}\mathfrak{L}$ is ISPD in relation $(iv)$, as we can only analyze its double integral on $\mathfrak{M}(\mathfrak{L})$.   

As described in Corollary \ref{productnkernels} for the discrete case, the continuous case of the Kronecker product of $n$ CND kernels on an $n$-fold Cartesian product has additional properties. For that, we use the integrability restrictions provided in Lemma \ref{integmarginkroeprodndim2}. 

\begin{corollary}\label{lastkronprod} Let $\gamma_{i}: X_{i} \times X_{i} \to \mathbb{R}$, $1\leq i \leq n$, be continuous CND metrizable kernels with bounded diagonal and $2\leq k \leq n $. The set
	\begin{align*}
		\mathfrak{M}_{k}(\vec{\gamma}):=\left\{ \mu \in \mathfrak{M}_{k} ( \mathds{X}_{n} ), \quad \prod_{i\in F}\gamma_{i} \in L^{1}(|\mu|\times |\mu|) \text{ for every } F \subset \{1,\ldots, n\} \right\},
	\end{align*}
	is a vector space and 
	$$
	(\mu, \nu) \to \int_{\mathds{X}_{n}}\int_{\mathds{X}_{n}}
	(-1)^{n}\left (\prod_{i=1}^{n} \gamma_{i} \right )(x_{\vec{1}}, x_{\vec{2}})d\mu(x_{\vec{1}})d\nu(x_{\vec{2}})
	$$
	is a well-defined semi-inner product in $\mathfrak{M}_{k}(\vec{\gamma})$. Further, the following are equivalent:
	\begin{enumerate}
		\item[$(i)$] It is an inner product in $\mathfrak{M}_{k}(\vec{\gamma})$.
		\item[$(ii)$] For every probability $P $ that satisfies Corollary \ref{integmarginkroeprodndim2} for which $\Lambda_{k}^{n}[P]$ is not the zero measure: 
		$$
		\int_{\mathds{X}_{n}}\int_{\mathds{X}_{n}}    (-1)^{n} \left (\prod_{i=1}^{n} \gamma_{i} \right )(x_{\vec{1}}, x_{\vec{2}})d\Lambda_{k}^{n}[P](x_{\vec{1}})d\Lambda_{k}^{n}[P](x_{\vec{2}})>0.
		$$
		\item[$(ii^{\prime})$] When $n=k$, for every probability $P $ that satisfies Corollary \ref{integmarginkroeprodndim2} for which $\Sigma[P]$ is not the zero measure: 
		$$
		\int_{\mathds{X}_{n}}\int_{\mathds{X}_{n}}    (-1)^{n}\left (\prod_{i=1}^{n} \gamma_{i} \right )(x_{\vec{1}}, x_{\vec{2}})d\Sigma[P](x_{\vec{1}})d\Sigma[P](x_{\vec{2}})>0.
		$$
		\item[$(iii)$] All CND kernels $\gamma_{i}$ are CND-Characteristic when $n=k$. The kernel $(-1)^{k}\vec{\gamma}$ is ISPD when $k< n$. 
	\end{enumerate}
\end{corollary}   

\begin{proof} The fact that $\mathfrak{M}_{k}(\vec{\gamma})$ is a vector space is a direct consequence of relation $(iii)$ in Lemma \ref{integmarginkroeprodndim2} and a comment made after the proof of the same lemma.\\ 
	Consider an arbitrary $x_{\vec{0}} \in \mathds{X}_{n}$ and the PD kernel $K^{\mathfrak{I}}$ relative to the PDI$_{n}$ kernel $\vec{\gamma}$ defined in Lemma \ref{PDIntoPDn}. Since
	$$
	K^{\mathfrak{I}}(x_{\vec{1}}, x_{\vec{2}})=\prod_{i=1}^{n}[-\gamma_{i}(x_{i}^{1}, x_{i}^{2}) +\gamma_{i}(x_{i}^{1}, x_{i}^{0}) +\gamma_{i}(x_{i}^{0}, x_{i}^{2}) - \gamma_{i}(x_{i}^{0}, x_{i}^{0})],
	$$
	by the same arguments as Theorem \ref{integPDIn}, we find that for any $\lambda, \eta \in \mathfrak{M}_{k}(\vec{\gamma})$, it holds that $\lambda, \eta$ satisfy the restrictions for the kernel mean embedding of $ K^{\mathfrak{I}}$ and also that     
	$$
	\int_{\mathds{X}_{n}}\int_{\mathds{X}_{n}}    (-1)^{n}\vec{\gamma}(x_{\vec{1}}, x_{\vec{2}})d\lambda(x_{\vec{1}})d\eta(x_{\vec{2}}) = \int_{\mathds{X}_{n}}\int_{\mathds{X}_{n}}     K^{\mathfrak{I}}(x_{\vec{1}}, x_{\vec{2}})d\lambda(x_{\vec{1}})d\eta(x_{\vec{2}}).
	$$
	The equivalences are obtained similarly to the proofs of Corollary \ref{productnkernels}, and thus are omitted.
\end{proof}

Similar to Corollary \ref{productnkernels}, Corollary \ref{lastkronprod} still holds true in the case $k=0$ or $k=1$, if we remove relations $(ii)$ and $(ii^{\prime})$, which are not defined in these cases. 

Explicit examples of ISPD and CND-Characteristic kernels in a general setting were proved in \cite{Guella2022a} and in \cite{Guella2022}.

\section{Discussion and Future Work}\label{futurework}

\textbf{Statistical Estimation:}
The primary contribution of this paper has been to establish the geometric, functional, and measure-theoretic foundations of $\mathrm{PDI}_{k}$ kernels. By formalizing the spaces of interaction measures $\mathfrak{M}_k(\mathds{X}_n)$ and introducing Fully Integrable by Partitions (FIP) integrals, we have provided the analytic machinery necessary to treat generalized independence tests as bounded operators on abstract product spaces. 

However, translating this population-level theory into practical statistical tools requires a rigorous finite-sample asymptotic analysis. Evaluating these interactions in practice involves constructing empirical estimators from finite data. For an i.i.d. sample of size $N$ drawn from a joint distribution $P$, the empirical probability measure $P_N$ can be used to construct an empirical interaction measure $\mu_N = T[P_N]$ (where $T$ is the generalized Lancaster or Streitberg operator or any interaction measure obtained from a partition lattice). The natural empirical test statistic is the $V$-statistic:
\begin{equation}
	\hat{T}_{N} =\int_{\mathds{X}_{n}} \int_{\mathds{X}_{n}}(-1)^{k} \mathfrak{I}(u,v) d\mu_{N}(u) d\mu_{N}(v)  =\int_{\mathds{X}_{n}} \int_{\mathds{X}_{n}} K^{\mathfrak{I}}(u,v) d\mu_{N}(u) d\mu_{N}(v) = \|K^{\mathfrak{I}}_{\mu_N}\|_{\mathcal{H}_{\mathfrak{I}}}^2.
\end{equation}

Developing the exact limiting distributions of $\hat{T}_{N}$ as $N \to \infty$ is left for future research, but our theoretical framework provides the prerequisites for this analysis. Because the test statistic is constructed to be degenerate under the null hypothesis of no interaction ($H_0: \mu = 0$), classical statistical theory \citep{Serfling1980, Lee1990} dictates that $N\hat{T}_N$ should converge in distribution to an infinite weighted sum of independent $\chi^2_1$ variables, $\sum \lambda_j Z_j^2$, where $\{\lambda_j\}$ are the eigenvalues of the associated integral operator. 

Crucially, this limit is only valid if the integral operator is trace-class. While this is immediate for bounded kernels (like the Gaussian kernel) \citep{Gretton2005, Liu2023}, $\mathrm{PDI}_k$ kernels generalizing distance covariance are frequently unbounded. In future work, we anticipate demonstrating that the FIP integrability conditions and the \emph{marginally delimited} property established in this paper provide exactly the required moment bounds to guarantee finite traces for these unbounded operators \citep{Lyons2013}.

\textbf{Further Examples of Interaction Measures:}
While this paper primarily focused on the generalized Lancaster and Streitberg interactions, the mathematical framework surrounding $\mathcal{M}_k(\mathds{X}_n)$ might be compatible with other interaction measures defined via Möbius inversion on partition lattices \citep{Liu2023, Liu2024}. Future research should explore deriving new types of such interaction measures and verifying whether they possess the same properties. Also, empirical tests using distinct kernels and distinct interactions measures are also desired.

\textbf{Examples Using Radial Kernels:}
Future work will also  focus on defining specific parametric families of these radial functions that satisfies   Theorem \ref{bernsksevndimpart3}. A characterization of when the functions in 	Theorem \ref{bernsksevndimpart3} defines a PDI$_{k}$-Characteristic kernel is also aimed. 

\appendix

\section{FIP integrals}\label{FIPintegrals}

In this first appendix, we prove the main results about FIP integrals, which are necessary for the development of Subsection \ref{Integrabilityrestrictionsnk} and subsequent applications. 

\begin{lemma} \label{propertiesFIP}Let $f: \mathds{X}_{n} \to \mathbb{R}$, $g: \mathds{X}_{n} \to \mathbb{R}$, $\mu \in \mathfrak{M}(\mathds{X}_{n})$, and a probability measure $P \in \mathfrak{M}(\mathds{X}_{n})$. Then: 
	\begin{enumerate}
		\item [$(i)$] The set of functions that are FIP with respect to $\mu$ is a vector space.
		\item [$(ii)$] If $f$ is FIP with respect to $\mu$, then for any $\mathcal{F} \subset \{1,\ldots, n\}$ and $x_{\vec{3}} \in \mathds{X}_{n}$, the functions
		$$
		x_{\vec{1}} \in \mathds{X}_{n} \to f(x_{\vec{1}_{\mathcal{F} } + \vec{3}_{\mathcal{F} ^{c}}}), \quad x_{\vec{1}_{F}} \in \mathds{X}_{\mathcal{F} } \to f(x_{\vec{1}_{\mathcal{F} } + \vec{3}_{\mathcal{F} ^{c}}}),
		$$
		are FIP with respect to $\mu$ and $\mu_{\mathcal{F} }$, respectively. 
		\item [$(iii)$] If $f$ is FIP with respect to $\mu$ and $|g(x )| \leq |f(x )|$ for every $x \in \mathds{X}_{n}$, then $g$ is FIP with respect to $\mu$.
		\item [$(iv)$] If $f$ is FIP with respect to $P$, then for any partition $\pi$, the function $f$ is also FIP with respect to $P_{\pi}$.
	\end{enumerate} 
\end{lemma}

\begin{proof} Relation $(i)$ is immediate, since for any $\lambda \in \mathbb{R}$ and $f, g$ which are FIP with respect to $\mu$, we have that
	\begin{align*}
		\int_{\mathds{X}_{n}}&\left |(f+t g)(x_{\vec{1}_{F} + \vec{2}_{F^{c}}}) \right |d|\mu|_{\pi}(x_{\vec{1}})\\
		& \leq \int_{\mathds{X}_{n}}\left |f(x_{\vec{1}_{F} + \vec{2}_{F^{c}}}) \right |d|\mu|_{\pi}(x_{\vec{1}}) + |t| \int_{\mathds{X}_{n}}\left |g(x_{\vec{1}_{F} + \vec{2}_{F^{c}}}) \right |d|\mu|_{\pi}(x_{\vec{1}}) < \infty.
	\end{align*}    
	For relation $(ii)$, note that if $h(x_{\vec{1}}):=f(x_{\vec{1}_{\mathcal{F}} + \vec{3}_{\mathcal{F}^{c}}}) $, then 
	$$
	h(x_{\vec{1}_{F} + \vec{2}_{F^{c}}}) = f(x_{\vec{1}_{F \cap \mathcal{F}} + \vec{2}_{\mathcal{F} - F} + \vec{3 }_{\mathcal{F}^{c}}}), 
	$$
	which is $|\mu|_{\pi}$-integrable as it is a special case of the functions that appear in the assumption that $f$ is FIP with respect to $\mu$.\\
	For relation $(iv)$, similar to relation $(ii)$, all integrals that are necessary to verify that $f$ is FIP with respect to $P_{\pi}$ are special cases of the integrals that appear in the hypothesis that $f$ is FIP with respect to $P$.\end{proof}   

\begin{lemma}\label{FIPmarvecspace} Let $f: \mathds{X}_{n} \to \mathbb{R}$ and a probability measure $P \in \mathfrak{M}(\mathds{X}_{n})$, such that $f$ is FIP with respect to $P$. Then, the set of measures in $\mathds{X}_{n}$ for which $f$ is FIP with respect to them and that are also marginally delimited by $P$ is a vector space.
\end{lemma}

\begin{proof}Let $\mu, \eta \in \mathfrak{M}(\mathds{X}_{n})$ be measures that satisfy these restrictions. Then, for any $t \in \mathbb{R}$, $|\mu| + |t||\eta| - |\mu + t \eta| $ is a nonnegative measure. Further, if $G \subset \{1, \ldots, n\}$ with $|G|\leq n-1$ and $\sigma=\{F_{1}, \ldots, L_{\ell}\}$ is a partition of $G$, then for any measurable set $A \subset \mathds{X}_{n}$,
	\begin{align*}
		(|\mu| + |t||\eta|)_{\sigma}(A) &\leq (M_{1} + |t|M_{2})^{\ell} \left [\prod_{j=1}^{\ell} \sum_{\pi_{j} \in \mathscr{P}(F_{j})} (P_{F})_{\pi_{j}} \right ](A)\\ 
		&\leq (M_{1} + |t|M_{2})^{\ell} \left [ \sum_{\pi \in \mathscr{P}(G)} P_{\pi} \right ](A).
	\end{align*}
	The first inequality is a consequence of the fact that $\mu$ and $\eta$ are marginally delimited by $P$ (where $M_{1}, M_{2}$ are the respective constants). The second inequality holds because for any sets $V,W$, we have that $\mathscr{P}(V) \times \mathscr{P}(W) \subset \mathscr{P}(V \cup W) $. This concludes that $\mu + t\eta$ is marginally delimited by $P$.\\     
	To show that $f$ is also FIP with respect to $\mu + t\eta$, note that if $\sigma$ is a partition of $\{1, \ldots, n\}$ with $|\sigma| \geq 2$, then by the same argument as the previous inequality, we get that for any $F \subset \{1,\ldots, n\}$:
	$$
	\int_{\mathds{X}_{m}}|f(x_{\vec{1}_{F}+ \vec{2}_{F^{c}} })|d(|\mu + t \eta|_{\sigma})(x_{\vec{1}}) \leq (M_{1} + |t|M_{2})^{|\sigma|} \sum_{\pi }\int_{\mathds{X}_{m}}|f(x_{\vec{1}_{F}+ \vec{2}_{F^{c}} })|dP_{\pi}(x_{\vec{1}}) < \infty.
	$$
	Now, if $|\sigma| =1$, the argument is simpler as $f(x_{\vec{1}_{F}+ \vec{2}_{F^{c}} })$ is integrable to both $|\mu|$ and $|\eta|$, as these integrals are a particular case of the integrals that appear in the hypothesis that $f$ is FIP with respect to $\mu$ and $\eta$. 
\end{proof}

Now we present an example of a function for which the set of measures that it is FIP is not a vector space. The example is very similar to the one for kernels in Example \ref{exkerneldiag}.  

\begin{example} \label{FIPbadintegration}Let $X_{1}= X_{2}= \mathbb{N}$ and consider the function $f(n,m):= \frac{nm}{n+m}$, the probabilities $P(n,m) = \zeta(3/2)^{-1} \zeta(3)^{-1}n^{-3/2}m^{-3} $ and $Q(n,m) := P(m,n) $ in $\mathbb{N}^{2}$, where $\zeta$ is the Riemann Zeta function. We claim that $f$ is FIP with respect to $P$ and $Q$, but is not FIP with respect to $(P+Q)/2$.\\
	Indeed,
	$$
	\zeta(3/2) \zeta(3)    \sum_{n,m=1}^{\infty} f(n,m)P(n,m) = \sum_{n,m=1}^{\infty} \frac{1}{n^{1/2}m^{2}(n+m)}\leq \sum_{n,m=1}^{\infty}\frac{1}{n^{3/2}m^{2}}< \infty,
	$$
	and since $P=P_{1}\times P_{2}$, we also have that 
	$$
	\sum_{n,m=1}^{\infty} f(n,m)P_{1}(n)P_{2}(m) < \infty.
	$$   
	Also, for every fixed $m \in \mathbb{N}$,
	$$
	\zeta(3/2)\sum_{n=1}^{\infty} f(n,m)P_{1}(n)= \sum_{n =1 }^{\infty}\frac{m}{n^{1/2}(n+m)}\leq \sum_{n=1 }^{\infty}\frac{m}{n^{3/2}}= m \zeta(3/2),
	$$
	and to conclude, for every fixed $n \in \mathbb{N}$,
	$$
	\zeta(3)\sum_{m=1}^{\infty} f(n,m)P_{2}(m)= \sum_{m =1 }^{\infty}\frac{n}{m^{2}(n+m)}\leq \sum_{m=1 }^{\infty}\frac{1}{m^{2}}= \zeta( 2).
	$$   
	The same properties hold for the probability $Q$, due to the fact that the function $f$ is symmetric in the sense that $f(n,m)=f(m,n)$ and by the definition of $Q$. However, $f$ is not integrable with respect to 
	$$
	\left (\frac{P+Q}{2} \right )_{1}\times \left (\frac{P+Q}{2} \right )_{2}.
	$$     
	To prove this, it is sufficient to show that $f$ is not integrable with respect to $P_{1}\times Q_{2}$. Indeed, our aim is to prove that 
	$$
	\zeta(3/2)^{2}\sum_{n,m=1}^{\infty} f(n,m)P_{1}(n)Q_{2}(m) = \sum_{n,m=1}^{\infty}\frac{1}{n^{1/2}m^{1/2}(n+m)},
	$$
	diverges. After the change of variables $n+m=k $ and $n=l$, this double sum is equivalent to
	$$
	\sum_{k=2}^{\infty}\sum_{l=1}^{k-1}\frac{1}{l^{1/2}(k-l)^{1/2}k},
	$$
	but, for $0\leq l \leq k$, we have that $l^{1/2}(k-l)^{1/2} \leq k/2$, thus 
	\begin{align*}
		\sum_{k=2}^{\infty}\sum_{l=1}^{k-1}\frac{1}{l^{1/2}(k-l)^{1/2}k} &\geq     \sum_{k=2}^{\infty}\sum_{l=1}^{k-1}\frac{2}{k^{2}}= \sum_{k=2}^{\infty}\frac{2(k-1)}{k^{2}}=\infty,
	\end{align*}
	because the harmonic series diverges.
\end{example}

The following two results are used in Section \ref{KroneckerproductsofPDIkernels}  for the  generalization of distance covariance property on the continuous case (see Theorem \ref{Kroenconti}). 

\begin{lemma}\label{integrationpartition} Let $g: \mathds{Y}_{m} \to \mathbb{R}$ be a nonzero function such that $g$ is integrable with respect to $|\mu|$, where $ \mu \in \mathfrak{M}(\mathds{X}_{n}\times \mathds{Y}_{m})$. Then, the measure
	$$
	\mu_{g}(A):=\int_{A\times \mathds{Y}_{m} } g(y_{\vec{1}}) d\mu(x_{\vec{1}}, y_{\vec{1}}), \quad A \in \mathscr{B}(\mathds{X}_{n})
	$$
	is finite. Also, for every partition $\pi$ of $\{1, \ldots, n\}$ with $\pi=\{F^{1}, \ldots, F^{\ell}\}$, it holds that
	\begin{align*}
		\int_{ \mathds{X}_{n} }& h(x) d(\mu_{g})_{\pi}(x)\\
		&= \int_{\mathds{X}_{n} \times \mathds{Y}_{m}} \ldots \int_{\mathds{X}_{n} \times \mathds{Y}_{m}}h(x_{\vec{1}_{{F}^{1}}}, \ldots, x_{\vec{\ell}_{{F}^{\ell}}})g(y_{\vec{1}}) \ldots g(y_{\vec{\ell}})d\mu(x_{\vec{1}}, y_{\vec{1}}) \ldots d\mu(x_{\vec{\ell}}, y_{\vec{\ell}}) 
	\end{align*}
	for any $h: \mathds{X}_{n} \to \mathbb{R} $ such that either side of the equality is an integrable function.        
\end{lemma}

\begin{proof} The measure $\mu_{g}$ is finite because $|\mu|$ is finite and $g$ is integrable with respect to it. As for the equality, if $h = \chi_{A}$, where $A = \prod_{i=1}^{n}A_{i} $, then
	\begin{align*}
		&\int_{ \mathds{X}_{n} }h(x) d(\mu_{g})_{\pi}(x) = (\mu_{g})_{\pi}(A) = \prod_{j=1}^{\ell} \int_{A_{F^{j}} \times \mathds{X}_{(F^{j})^{c}} \times \mathds{Y}_{m} } g(y_{\vec{j}}) d\mu(x_{\vec{j}}, y_{\vec{j}})\\
		&= \int_{\mathds{X}_{n} \times \mathds{Y}_{m}} \ldots \int_{\mathds{X}_{n} \times \mathds{Y}_{m}}\chi_{A}(x_{\vec{1}_{{F}^{1}}}, \ldots, x_{\vec{\ell}_{{F}^{\ell}}})g(y_{\vec{1}}) \ldots g(y_{\vec{\ell}})d\mu(x_{\vec{1}}, y_{\vec{1}}) \ldots d\mu(x_{\vec{\ell}}, y_{\vec{\ell}}). 
	\end{align*}   
	As linear combinations of these functions also satisfy this equality, we obtain the desired result by standard integral approximations.    
\end{proof}

Now, we prove a nontrivial version of Holder's inequality for the FIP integration.  

\begin{theorem}\label{FIPproduct} Let $f: \mathds{X}_{n} \to \mathbb{R} $ and $g: \mathds{Y}_{m} \to \mathbb{R}$ be nonzero functions such that $fg^{2}$ is FIP with respect to $\mu \in \mathfrak{M}(\mathds{X}_{n}\times \mathds{Y}_{m})$. Then, the function $f$ is FIP with respect to $\mu_{g} \in \mathfrak{M}(\mathds{X}_{n})$ as defined in Lemma \ref{integrationpartition}. 
\end{theorem}

\begin{proof} We may, without loss of generality, assume that $P=|\mu|$ is a probability in $\mathds{X}_{n}\times \mathds{Y}_{m}$ and that $g$ is a nonnegative function. Suppose also that $\pi=\{G, G^{c}\}$ is a partition of $\{1,\ldots, n\}$; by the definition of the measure in Lemma \ref{integrationpartition} and its equality, we get that for any nonnegative function $h: \mathds{X}_{n} \to \mathbb{R}$,
	\begin{align*}
		&\int_{ \mathds{X}_{n} } h(x) d(|\mu_{g}|)_{\pi}(x) =\int_{ \mathds{X}_{n} } h(x) d(P_{g})_{\pi}(x) \\
		&= \int_{\mathds{X}_{n} \times \mathds{Y}_{m}} \int_{\mathds{X}_{n} \times \mathds{Y}_{m}}h(x_{\vec{1}_{{G}}}, x_{\vec{2}_{{G}^{c}}}) g(y_{\vec{1}}) g(y_{\vec{2}}) dP(x_{\vec{1}}, y_{\vec{1}}) dP(x_{\vec{2}}, y_{\vec{2}}). 
	\end{align*}
	Note that 
	\begin{align*}
		& \left \| \sqrt{h(x_{\vec{1}_{{G}}}, x_{\vec{2}_{{G}^{c}}}) } g(y_{\vec{1}}) \right \|_{L^{2}(P\times P)} ^{2}\\
		&=\int_{\mathds{X}_{n} \times \mathds{Y}_{m}} \int_{\mathds{X}_{n} \times \mathds{Y}_{m}}h(x_{\vec{1}_{{G}}}, x_{\vec{2}_{{G}^{c}}}) g(y_{\vec{1}}) ^{2}dP(x_{\vec{1}}, y_{\vec{1}}) dP(x_{\vec{2}}, y_{\vec{2}})\\
		&= \int_{\mathds{X}_{n} \times \mathds{Y}_{m}} h(x_{\vec{1}}) g(y_{\vec{1}}) ^{2}dP_{\pi^{\prime}}(x_{\vec{1}}, y_{\vec{1}}) 
	\end{align*}
	where $\pi^{\prime} = \{G^{c}, G\cup\{1,\ldots, m\} \}$ is (with a little abuse of notation) a partition of the set $\{1,\ldots, n\} \cup \{1,\ldots, m\}$ with $|\pi^{\prime}|=2$. Similarly,
	$$
	\left \| \sqrt{h(x_{\vec{1}_{{G}}}, x_{\vec{2}_{{G}^{c}}}) } g(y_{\vec{2}}) \right \|_{L^{2}(P\times P)} ^{2}= \int_{\mathds{X}_{n} \times \mathds{Y}_{m}} h(x_{\vec{1}}) g(y_{\vec{1}}) ^{2}dP_{\pi^{\prime \prime}}(x_{\vec{1}}, y_{\vec{1}}),
	$$
	where $\pi^{\prime \prime} = \{G^{c}\cup\{1,\ldots, m\}, G \}$. Thus, by H\"older's inequality in $L^{2}(P \times P)$, we get that
	\begin{align*}
		& \int_{\mathds{X}_{n} \times \mathds{Y}_{m}} \int_{\mathds{X}_{n} \times \mathds{Y}_{m}}h(x_{\vec{1}_{{G}}}, x_{\vec{2}_{{G}^{c}}}) g(y_{\vec{1}}) g(y_{\vec{2}}) dP(x_{\vec{1}}, y_{\vec{1}}) dP(x_{\vec{2}}, y_{\vec{2}}) \\
		& \leq \left \| \sqrt{h(x_{\vec{1}_{{G}}}, x_{\vec{2}_{{G}^{c}}}) } g(y_{\vec{1}}) \right \|_{L^{2}(P\times P)} \left \| \sqrt{h(x_{\vec{1}_{{G}}}, x_{\vec{2}_{{G}^{c}}}) } g(y_{\vec{2}}) \right \|_{L^{2}(P\times P)}.
	\end{align*}
	Now, if $f$ is as in the hypothesis, for any subset $F$ of $\{1, \ldots, n\}$ and arbitrary element $x_{\vec{3}} \in \mathds{X}_{n}$, the function $h(x_{\vec{1}}):=|f(x_{\vec{1}_{F} + \vec{3}_{F^{c}}})| $ is such that
	$$
	\left \| \sqrt{h(x_{\vec{1}_{{G}}}, x_{\vec{2}_{{G}^{c}}}) } g(y_{\vec{1}}) \right \|_{L^{2}(P\times P)} ^{2}= \int_{\mathds{X}_{n} \times \mathds{Y}_{m}} |f(x_{\vec{1}_{F} + \vec{3}_{F^{c}}})| g(y_{\vec{1}}) ^{2}dP_{\pi^{\prime }}(x_{\vec{1}}, y_{\vec{1}}) < \infty, 
	$$
	and similarly for the other term, reaching that 
	$$
	\int_{ \mathds{X}_{n} } |f(x_{\vec{1}_{F} + \vec{3}_{F^{c}}})| d(|\mu_{g}|)_{\pi}(x_{\vec{1}}) < \infty. 
	$$
	We have proved that Definition \ref{FIP} holds for when $|\pi| = 2$; the other values are obtained recursively. For instance, if $\pi= \{G^{1}, G^{2}, G^{3} \}$, then 
	\begin{align*}
		&\int_{ \mathds{X}_{n} } h(x) d(|\mu_{g}|)_{\pi}(x) = \int_{ \mathds{X}_{n} } h(x) d(P_{g})_{\pi}(x) \\
		&= \int_{(\mathds{X}_{n} \times \mathds{Y}_{m})^{3}}h(x_{\vec{1}_{{G^{1}}}}, x_{\vec{2}_{{G^{2}}}}, x_{\vec{3}_{{G^{3}}}}) g(y_{\vec{1}}) g(y_{\vec{2}}) g(y_{\vec{3}})dP(x_{\vec{1}}, y_{\vec{1}}) dP(x_{\vec{2}}, y_{\vec{2}})dP(x_{\vec{3}}, y_{\vec{3}})\\
		&\leq \left \| \sqrt{h(x_{\vec{1}_{{G^{1}}}}, x_{\vec{2}_{{G^{2}}}}, x_{\vec{3}_{{G^{3}}}}) } g(y_{\vec{1}}) \right \|_{L^{2}(P^{3})}\left \| \sqrt{h(x_{\vec{1}_{{G^{1}}}}, x_{\vec{2}_{{G^{2}}}}, x_{\vec{3}_{{G^{3}}}}) } g(y_{\vec{2}}) g(y_{\vec{3}}) \right \|_{L^{2}(P^{3})}, 
	\end{align*}
	and we may reapply H\"older's inequality to reach
	\begin{align*}
		&\left \| \sqrt{h(x_{\vec{1}_{{G^{1}}}}, x_{\vec{2}_{{G^{2}}}}, x_{\vec{3}_{{G^{3}}}}) } g(y_{\vec{2}}) g(y_{\vec{3}}) \right \|_{L^{2}(P\times P \times P)}^{2}\\
		& \leq \left \| \sqrt{h(x_{\vec{1}_{{G^{1}}}}, x_{\vec{2}_{{G^{2}}}}, x_{\vec{3}_{{G^{3}}}}) } g(y_{\vec{2}}) \right \|_{L^{2}(P^{3})}\left \| \sqrt{h(x_{\vec{1}_{{G^{1}}}}, x_{\vec{2}_{{G^{2}}}}, x_{\vec{3}_{{G^{3}}}}) }g(y_{\vec{3}}) \right \|_{L^{2}(P^{3})}. 
	\end{align*}
\end{proof}

\section{Geometrical representation of PDI$_{k}$ kernels}\label{geoint}

The geometrical interpretation for PDI$_{k}$ kernels defined in $\mathds{X}_{n}$ by using the RKHS of the related positive definite kernel $    K^{\mathfrak{I}}$ becomes more complicated as the codimension $n-k$ increases.

First, we prove two technical results that will be needed for the proof of Theorem \ref{finalgeomkn}. 

\begin{lemma}\label{indpdi2} Let $\mathfrak{I} : \mathds{X}_{n} \times \mathds{X}_{n} \to \mathbb{R} $ be an $n$-symmetric PDI$_{k}$ kernel that is zero at the extended diagonal $\Delta_{k-1}^{n}$, a subset $F \subset \{1, \ldots, n\}$ with $|F| \geq k$, and a fixed $x_{i}^{3} \in X_{i}$, $i \in F^{c}$. Consider the kernel $\mathcal{I} : \mathds{X}_{F} \times \mathds{X}_{F} \to \mathbb{R} $, defined as
	$$
	\mathcal{I}(x_{\vec{1}_{F}}, x_{\vec{2}_{F}}):= \mathfrak{I}(x_{\vec{1}_{F}+\vec{3}_{F^{c}}},x_{\vec{2}_{F}+\vec{3}_{F^{c}}} ).
	$$ 
	Then, $\mathcal{I}$ is an $|F|$-symmetric PDI$_{k}$ kernel in $\mathds{X}_{F}$ that is zero at the extended diagonal $\Delta_{k-1}^{|F|}$ of $\mathds{X}_{F}$. Also, for any     $\eta \in \mathcal{M}_{k }( \mathds{X}_{F})$, it holds that $\eta\times \delta_{\vec{3}_{F^{c}}} \in \mathcal{M}_{k }( \mathds{X}_{n}) $ and 
	\begin{align*}
		\| K^{\mathcal{I}}_{\eta} \|^{2}_{\mathcal{H}_{\mathcal{I}}}=\|K^{\mathfrak{I}}_{\eta\times \delta(\vec{3}_{F^{c}})} \|^{2}_{\mathcal{H}_{\mathfrak{I}}}.
	\end{align*}
\end{lemma}

\begin{proof}The fact that $\mathcal{I}$ is an $|F|$-symmetric PDI$_{k}$ kernel that is zero at the extended diagonal $\Delta_{k-1}^{|F|}$ of $\mathds{X}_{F}$ is a direct consequence of Lemma \ref{projectionkernels} by taking $\lambda^{F}:= \delta(\vec{3}_{F^{c}})$ with $b=0$.\\
	Now, let $\eta \in \mathcal{M}_{k }( \mathds{X}_{F})$; then $\eta\times \delta(\vec{3}_{F^{c}}) \in \mathcal{M}_{k }( \mathds{X}_{n}) $ by relation $(iv)$ in Lemma \ref{inclusionmeasures2indexkroen}. By Equation \eqref{contaPDIntoPDn} and Equation \eqref{contaPDIkntoPD}, it holds that     
	\begin{align*}
		&\| K^{\mathcal{I}}_{\eta} \|^{2}_{\mathcal{H}_{\mathcal{I}}}=\langle K^{\mathcal{I}}_{\eta}, K^{\mathcal{I}}_{\eta} \rangle_{\mathcal{H}_{\mathcal{I}}} = \int_{\mathds{X}_{F}}\int_{\mathds{X}_{F}} (-1)^{k}\mathcal{I}(x_{\vec{1}_{F}},x_{\vec{2}_{F}} ) d\eta(x_{\vec{1}_{F}})d\eta(x_{\vec{2}_{F}})\\
		& = \int_{\mathds{X}_{n}}\int_{\mathds{X}_{n}} (-1)^{k}\mathfrak{I}(x_{\vec{1} },x_{\vec{2} } ) d[\eta\times \delta(\vec{3}_{F^{c}})](x_{\vec{1}})d[\eta\times \delta(\vec{3}_{F^{c}})](x_{\vec{2} })=\| K^{\mathfrak{I}}_{\eta\times \delta(\vec{3}_{F^{c}})} \|^{2}_{\mathcal{H}_{\mathfrak{I}}}. 
	\end{align*}
\end{proof}   

Given a measure $\eta= \sum_{\alpha \in \mathbb{N}_{2}^{n}}C_{\alpha}\delta_{\alpha} \in \mathcal{M}(\mathbb{N}_{2}^{n})$, where as defined in Section \ref{Terminology} the set $\mathbb{N}_{2}^{n}=\{1,2\}^{n}$, we can define a new measure in $\mathcal{M}(\mathds{X}_{n})$, where for a fixed but arbitrary pair of points $x_{\vec{1}}, x_{\vec{2}} \in \mathds{X}_{n}$:
$$
\eta[x_{\vec{1}}, x_{\vec{2}}]:= \sum_{\alpha \in \mathbb{N}_{2}^{n}}C_{\alpha} \delta_{x_{\alpha}}.
$$

Naturally, if $\eta \in \mathcal{M}_{k}(\mathbb{N}_{2}^{n})$, then $\eta[x_{\vec{1}}, x_{\vec{2}}] \in \mathcal{M}_{k}(\mathds{X}_{n})$ for any pair of points $x_{\vec{1}}, x_{\vec{2}} \in \mathds{X}_{n}$. For instance, the measure $\mu_{k}^{n}[x_{\vec{1}}, x_{\vec{2}}]$ can be obtained using 
\begin{align*}
	\mu_{k}^{n}: = \delta_{ \vec{1}}+ \sum_{j=0}^{k-1}(-1)^{k-j}\binom{n-j-1}{n-k}\sum_{|F|=j}\delta_{\vec{1}_{F} + \vec{2}_{F^{c}} }
\end{align*}
and by the well-known relations
$$
\binom{a}{0}=1 \text{ for every nonnegative } a \text{ and } \binom{a}{n-k}=0 \text{ when }|a| < n-k. 
$$    

We say that $\eta= \sum_{\alpha \in \mathbb{N}_{2}^{n}}C_{\alpha}\delta_{\alpha} \in \mathcal{M}(\mathbb{N}_{2}^{n})$ is Permutation Invariant (PI) if 
$$
\eta((\alpha_{1}, \ldots, \alpha_{n} ))= \eta((\alpha_{\sigma(1)}, \ldots, \alpha_{\sigma(n)} )) 
$$
for any $\alpha \in \mathbb{N}_{2}^{n}$ and bijection $\sigma : \{1,\ldots, n\} \to \{1,\ldots, n\}$. It is a simple result that $\eta= \sum_{\alpha \in \mathbb{N}_{2}^{n}}C_{\alpha}\delta_{\alpha} \in \mathcal{M}(\mathbb{N}_{2}^{n})$ is PI if and only if $C_{\alpha}= C_{\alpha^{\prime}}$ whenever $|\alpha| = |\alpha^{\prime}|$, that is, there exist constants $w_{j}$, $0\leq j \leq n$, for which
$$
\eta= \sum_{j=0} ^{n}w_{j} \sum_{|F|=j}\delta(\vec{1}_{F} + \vec{2}_{F^{c}}).
$$ 

In particular, an PI measure as before has marginal $\sum_{j=0} ^{n-1}[w_{j}+w_{j+1}] \sum_{|F|=j}\delta(\vec{1}_{F} + \vec{2}_{F^{c}})$ in the first $n-1$ variables. By recursively using  this relation,  we can get that such a measure is an element of $\mathcal{M}_{k}( \mathbb{N}_{2}^{n})$ if and only if 

$$
\sum_{j=0}^{k} \binom{k}{j}w_{\ell +j} =0, \quad 0 \leq \ell \leq n-k.
$$
As the characteristic polynomial of this recurrence sequence is $(t-1)^{k}$, its solutions are of the type
\begin{equation}\label{recursive}
	w_{j}= (-1)^{j}(A_{0} +A_{1}j +\ldots +A_{k-1}j^{k-1}), 
\end{equation}
for real constants $A_{0}, \ldots A_{k-1}$. The representation and coefficients we aim to obtain are simplified when dealing with PI measures, which occurs by the following result.

\begin{lemma}\label{procedure} Let $\mathfrak{I} : \mathds{X}_{n} \times \mathds{X}_{n} \to \mathbb{R} $ be a complete $n$-symmetric PDI$_{k}$ kernel that is zero at the extended diagonal $\Delta_{k-1}^{n}(\mathds{X}_{n})$, a fixed $x_{\vec{0}} \in \mathds{X}_{n} $, and $\eta= \sum_{\alpha \in \mathbb{N}_{2}^{n}} C_{\alpha}\delta_{\alpha} \in     \mathcal{M}_{k }( \mathbb{N}_{2}^{n})$. Then, for any $x_{\vec{3}} \in \mathds{X}_{n} $, it holds that         
	\begin{align*}
		&\| K^{\mathfrak{I}}_{\eta[x_{\vec{1}}, x_{\vec{2}}]}\|_{\mathcal{H}_{\mathfrak{I}}}^{2}=\sum_{\alpha, \beta \in \mathbb{N}_{2}^{n}} (-1)^{k}C_{\alpha}C_{\beta} \mathfrak{I}(x_{\alpha}, x_{\beta})\\
		&= (-1)^{k}\sum_{|F|=k}^{n} \left [\sum_{\xi \in \mathbb{N}_{2}^{n-|F|}}\sum_{\varsigma \in \mathbb{N}_{2}^{|F|}}C_{(\varsigma_{F}+ \xi_{F^{c}})} C_{((\vec{3}-\varsigma)_{F}+ \xi_{F^{c}})} \right ] \mathfrak{I}(x_{\vec{1}_{F} + \vec{3}_{F^{c}}},x_{\vec{2}_{F} + \vec{3}_{F^{c}}} ). 
	\end{align*}
	Additionally, if $\eta=\sum_{j=0} ^{n}w_{j} \sum_{|F|=j}\delta(\vec{1}_{F} + \vec{2}_{F^{c}})$ is PI, we have that 
	\begin{align*}
		&\| K^{\mathfrak{I}}_{\eta[x_{\vec{1}}, x_{\vec{2}}]}\|_{\mathcal{H}_{\mathfrak{I}}}^{2}= (-1)^{k}\sum_{|F|=k}^{n} D_{|F|}^{n,k} \mathfrak{I}(x_{\vec{1}_{F} + \vec{3}_{F^{c}}},x_{\vec{2}_{F} + \vec{3}_{F^{c}}} ), 
	\end{align*}
	where for $k\leq j \leq n$,
	$$
	D_{j}^{n,k}:=\sum_{p=0}^{n-j} \sum_{q=0}^{j}\binom{j}{q}\binom{n-j}{p} w_{p+q}w_{p + j-q}.
	$$
\end{lemma}

\begin{proof}
	Indeed, for the first equation, the first equality is a consequence of Lemma \ref{pdi2topdn} and the fact that $\eta[x_{\vec{1}}, x_{\vec{2}}] \in \mathcal{M}_{k }( \mathds{X}_{n})$. The second equality is a consequence of Equation \eqref{simplidouble}, the complete $n$-symmetry, and the fact that $\mathfrak{I}$ is zero at the extended diagonal $\Delta_{k-1}^{n}(\mathds{X}_{n})$. \\
	For the description of the coefficients $ D_{j}^{n,k}$, note that for any set with $|F|=j$, by counting how many entries with $1$ the coefficients $(\varsigma_{F}+ \xi_{F^{c}})$ and $((\vec{3}-\varsigma)_{F}+ \xi_{F^{c}})$ have, we obtain that
	\begin{align*}
		D_{j}^{n,k}&= \sum_{\xi \in \mathbb{N}_{2}^{n-|F|}}\sum_{\varsigma \in \mathbb{N}_{2}^{|F|}}C_{(\varsigma_{F}+ \xi_{F^{c}})} C_{((\vec{3}-\varsigma)_{F}+ \xi_{F^{c}})}\\
		&= \sum_{\xi \in \mathbb{N}_{2}^{n-|F|}}\sum_{\varsigma \in \mathbb{N}_{2}^{|F|}}w_{p+q} w_{p + j-q} = \sum_{p=0}^{n-j} \sum_{q=0}^{j}\binom{j}{q}\binom{n-j}{p} w_{p+q} w_{p + j-q} .    
	\end{align*}
\end{proof}

Based on Lemma \ref{procedure}, we define that a measure $\eta= \sum_{\alpha \in \mathbb{N}_{2}^{n}} C_{\alpha}\delta_{\alpha} \in     \mathcal{M}( \{1,2\}^{n} )$ is Reversible if 
$$
\sum_{\alpha \in \mathbb{N}_{2}^{n}}C_{\alpha}C_{\vec{3} -\alpha} \neq 0.
$$
This concept will be needed as we will obtain that $\mathfrak{I}$ is a sum of squares by a recurrence argument, in which we have to divide all terms by the factors $D^{n,k}_{n}$. 

Naturally, if $\eta$ is PI, that is, if $\eta=\sum_{j=0} ^{n}w_{j} \sum_{|F|=j}\delta(\vec{1}_{F} + \vec{2}_{F^{c}})$, then being Reversible is equivalent to 
$$
\sum_{q=0}^{n}\binom{n}{q} w_{q}w_{n-q} \neq 0.
$$

Note that if $\eta= \sum_{\alpha \in \mathbb{N}_{2}^{n}} C_{\alpha}\delta_{\alpha} \in     \mathcal{M}_{n}( \{1,2\}^{n} )$, by solving its restrictions, we obtain that $\eta = C \mu_{n}^{n}$ for some $C \in \mathbb{R}$, and for every nonzero $C$, this measure is Reversible, because
$$
\sum_{\alpha \in \mathbb{N}_{2}^{n}}C_{\alpha}C_{\vec{3} -\alpha}= C^{2} \sum_{\alpha \in \mathbb{N}_{2}^{n}}(-1)^{n-|\vec{2}-\alpha|}(-1)^{n-|\vec{2}-(\vec{3}-\alpha)|}=(-1)^{n}C^{2} 2^{n} \neq 0.
$$
The measure $\mu_{k}^{n}$ in Equation \eqref{measorderk} is Reversible because
$$
\sum_{q=0}^{n}\binom{n}{q} w_{q}w_{n-q}=(-1)^{k}\left [2\binom{n-1}{n-k} +\sum_{q=1}^{n-1}\binom{n}{q}\binom{n-q-1}{n-k}\binom{q-1}{n-k}\right ] \neq 0.
$$

\begin{remark} There are nonzero measures in $\mathcal{M}_{k}( \{1,2\}^{n} )$ for $2 \leq k < n$ that are not Reversible. 
	
	First, note that if $\eta=\sum_{j=0} ^{n}w_{j} \sum_{|F|=j}\delta(\vec{1}_{F} + \vec{2}_{F^{c}}) \in \mathcal{M}_{k}( \{1,2\}^{n} )$ is Reversible, then $\eta^{-1}:=\sum_{j=0} ^{n}w_{n-j} \sum_{|F|=j}\delta(\vec{1}_{F} + \vec{2}_{F^{c}}) \in \mathcal{M}_{k}( \{1,2\}^{n} )$ and is Reversible. For a value $a \in \{0,1\}$, the measure $\eta + (-1)^{a}\eta^{-1} \in \mathcal{M}_{k}( \{1,2\}^{n} )$ satisfies
	$$
	\sum_{q=0}^{n}\binom{n}{q} [ w_{q} + (-1)^{a}w_{n-q}][w_{n-q} + (-1)^{a}w_{q} ]= (-1)^{a}    \sum_{q=0}^{n}\binom{n}{q} [ w_{q} + (-1)^{a}w_{n-q}]^{2},
	$$
	hence it is Reversible unless $\eta + (-1)^{a}\eta^{-1}$ is the zero measure. 
	
	When $\eta = \mu_{k}^{n}$, for $2\leq k < n$, the measure $\eta + (-1)^{a}\eta^{-1}$ is always nonzero. In particular, we can choose the value $a \in \{0,1\}$ (depending on $n$ and $k$) so that the values of the Reversible property for the measures $\eta$ and $\eta + (-1)^{a}\eta^{-1}$ have opposite signs. 
	
	By continuity, there must exist an $r \in (0,1)$ for which $r\eta + (1-r)(\eta + (-1)^{a}\eta^{-1})$ is not Reversible. However, the coefficient that multiplies $\sum_{|F|=n-1}\delta(\vec{1}_{F} + \vec{2}_{F^{c}})$ for this measure is $(1-r)(-1)^{a+k+1}\binom{n-2}{n-k}$, which is nonzero for $r\in (0,1)$, which concludes our claim.\end{remark}   

The next theorem presents a representation of PDI$_{k}$ kernels as sums of squares.  

\begin{theorem}\label{finalgeomkn} Let $\mathfrak{I} : \mathds{X}_{n} \times \mathds{X}_{n} \to \mathbb{R} $ be a complete $n$-symmetric PDI$_{k}$ kernel that is zero at the extended diagonal $\Delta_{k-1}^{n}$ and PI Reversible measures $\eta_{j} \in \mathcal{M}_{k }( \{1,2\}^{j} )$, $k \leq j \leq n $. Then for any fixed $x_{\vec{3}} \in \mathds{X}_{n} $, we have that there exist real constants $b_{j }^{n,k} $, $k \leq j \leq n $, for which 
	$$
	\mathfrak{I}(x_{\vec{1}}, x_{\vec{2}})=     \sum_{|F|=k}^{n} b_{|F|}^{n,k} \| K^{\mathfrak{I}}_{\eta_{|F|}[ \vec{1}_{F} + \vec{3}_{F^{c}}, \vec{2}_{F} + \vec{3}_{F^{c}} ]} \|^{2}_{\mathcal{H}_{\mathfrak{I}}}, 
	$$
	for any $x_{\vec{0}} \in \mathds{X}_{n} $ that is used to define $K^{\mathfrak{I}}$, where 
	$$
	\eta_{|F|}[ \vec{1}_{F} + \vec{3}_{F^{c}}, \vec{2}_{F} + \vec{3}_{F^{c}} ] = \eta_{|F|}[x_{\vec{1}_{F}}, x_{\vec{2}_{F}}] \times \delta_{\vec{3}_{F^{c}}}.
	$$        
\end{theorem}

\begin{proof} The proof proceeds by induction on $n$.\\
	If $n=k$, this result is a consequence of Theorem \ref{PDIngeometryrkhs}, as any nonzero measure $\eta_{n} \in \mathcal{M}_{n}( \{1,2\}^{n} )$ is a multiple of $\mu_{n}^{n}$; thus, if $\eta_{n} = C \mu_{n}^{n} $,
	\begin{align*}
		&\frac{1}{2^{n}C^{2}}\| K^{\mathfrak{I}}_{\eta_{n}[ \vec{1}, \vec{2} ]}\|^{2}_{\mathcal{H}_{\mathfrak{I}}}= \frac{1}{2^{n}C^{2}} \| CK^{\mathfrak{I}}_{\mu_{n}^{n}[ \vec{1}, \vec{2} ]}\|^{2}_{\mathcal{H}_{\mathfrak{I}}} =\frac{1}{2^{n}} \| K^{\mathfrak{I}}_{\mu_{n}^{n}[ \vec{1}, \vec{2} ]}\|^{2}_{\mathcal{H}_{\mathfrak{I}}} =\mathfrak{I}(x_{\vec{1}},x_{\vec{2}}).
	\end{align*}
	Now, suppose that the result is valid for all values of $n \in \{k, \ldots, m-1\}$ and we prove that it also holds for $n=m$. By Lemma \ref{projectionkernels}, for every $G \subset \{1, \ldots, m\}$ with $k\leq |G| \leq m-1$, the kernel 
	$$
	(x_{\vec{1}_{G}},x_{\vec{2}_{G}}) \in \mathds{X}_{G} \times \mathds{X}_{G} \to \mathfrak{I}_{G}(x_{\vec{1}_{G}},x_{\vec{2}_{G}}):=\mathfrak{I}(x_{\vec{1}_{G} + \vec{3}_{G^{c}} },x_{\vec{2}_{G} + \vec{3}_{G^{c}}})
	$$
	is a complete $|G|$-symmetric PDI$_{k}$ kernel that is zero at the extended diagonal $\Delta_{k-1}^{|G|}$ in $\mathds{X}_{G}$. By the induction hypothesis, we have that 
	$$
	\mathfrak{I}(x_{\vec{1}_{G} + \vec{3}_{G^{c}} },x_{\vec{2}_{G} + \vec{3}_{G^{c}}})= \mathfrak{I}_{G}(x_{\vec{1}_{G}}, x_{\vec{2}_{G}})= \sum_{j=k}^{|G|}     \sum_{|F|=j, F \subset G }b_{|F|}^{|G|,k}\| K^{\mathfrak{I}_{G}}_{\eta_{j}[ \vec{1}_{F} + \vec{3}_{G-F}, \vec{2}_{F} + \vec{3}_{G-F} ]} \|^{2}_{ \mathcal{H}_{\mathfrak{I}_{G}}}, 
	$$
	while Lemma \ref{indpdi2} asserts that 
	$$
	\| K^{\mathfrak{I}_{G}}_{\eta_{j}[ \vec{1}_{F} + \vec{3}_{G-F}, \vec{2}_{F} + \vec{3}_{G-F} ]} \|^{2}_{\mathcal{H}_{\mathfrak{I}_{G}}}= \| K^{\mathfrak{I}}_{\eta_{j}[ \vec{1}_{F} + \vec{3}_{F^{c}}, \vec{2}_{F} + \vec{3}_{F^{c}} ]} \|^{2}_{ \mathcal{H}_{\mathfrak{I}}}.
	$$
	To conclude, if $ \eta_{n} = \sum_{\ell=0}^{n}w_{\ell} \sum_{|F|=\ell} \delta_{\vec{1}_{F} + \vec{2}_{F^{c}}}$, using Lemma \ref{procedure} to define
	\begin{align*}
		D^{n,k}_{\ell}:&= \sum_{p=0}^{n-\ell} \sum_{q=0}^{\ell}\binom{\ell}{q}\binom{n-\ell}{p} w_{p+q}w_{p + \ell-q}, 
	\end{align*}
	we reach that
	\begin{align*}
		&\|K^{\mathfrak{I}}_{\eta_{n}[ \vec{1}, \vec{2} ]}\|_{\mathcal{H}_{\mathfrak{I}}}^{2}= (-1)^{k}\sum_{|G|=k}^{n} D^{n,k}_{|G|} \mathfrak{I}(x_{\vec{1}_{G} + \vec{3}_{G^{c}}},x_{\vec{2}_{G} + \vec{3}_{G^{c}}} )\\
		&= (-1)^{k}D_{n}^{n,k} \mathfrak{I}(x_{\vec{1} },x_{\vec{2} } ) +(-1)^{k}\sum_{|G|=k}^{n-1} D^{n,k}_{|G|} \mathfrak{I}(x_{\vec{1}_{G} + \vec{3}_{G^{c}}},x_{\vec{2}_{G} + \vec{3}_{G^{c}}} )\\
		&= (-1)^{k}D_{n}^{n,k} \mathfrak{I}(x_{\vec{1} },x_{\vec{2} } ) +(-1)^{k}\sum_{|G|=k}^{n-1} D^{n,k}_{|G|} \left [ \sum_{j=k}^{|G|}     \sum_{|F|=j, F \subset G } b_{|F|}^{|G|,k}\| K^{\mathfrak{I}}_{\eta_{j}[ \vec{1}_{F} + \vec{3}_{F^{c}}, \vec{2}_{F} + \vec{3}_{F^{c}} ]} \|^{2}_{\mathcal{H}_{\mathfrak{I}}}\right ].   
	\end{align*}
	By a simple combinatorial argument, note that 
	\begin{align*}
		\sum_{|G|=k}^{n-1} D^{n,k}_{|G|} &\left [ \sum_{j=k}^{|G|}     \sum_{|F|=j, F \subset G } b_{|F|}^{|G|,k}\| K^{\mathfrak{I}}_{\eta_{j}[ \vec{1}_{F} + \vec{3}_{F^{c}}, \vec{2}_{F} + \vec{3}_{F^{c}} ]} \|^{2}_{ \mathcal{H}_{\mathfrak{I}}}\right ] \\
		&=     \sum_{|F|=k}^{n-1} E_{|F|}^{n,k}\| K^{\mathfrak{I}}_{\eta_{|F|}[ \vec{1}_{F} + \vec{3}_{F^{c}}, \vec{2}_{F} + \vec{3}_{F^{c}} ]} \|^{2}_{ \mathcal{H}_{\mathfrak{I}}}, 
	\end{align*}
	where 
	$$
	E_{j}^{n,k} = \binom{n-j}{0}D^{n,k}_{j+0}b^{j+0,k}_{j} + \ldots + \binom{n-j}{n-j-1}D^{n,k}_{n-1}b^{n-1,k}_{j} = \sum_{ r=0}^{n-j-1}\binom{n-j}{r}D^{n,k}_{j+r}b^{j+r,k}_{j}, 
	$$
	hence we reach that   
	$$
	\mathfrak{I}(x_{\vec{1} },x_{\vec{2} } ) = \frac{(-1)^{k}}{D_{n}^{n,k}}\|K^{\mathfrak{I}}_{\eta_{n}[ \vec{1}, \vec{2} ]}\|_{\mathcal{H}_{\mathfrak{I}}}^{2} - \sum_{|F|=k}^{n-1} \frac{E_{|F|}^{n,k}}{D_{n}^{n,k}}\| K^{\mathfrak{I}}_{\eta_{|F|}[ \vec{1}_{F} + \vec{3}_{F^{c}}, \vec{2}_{F} + \vec{3}_{F^{c}} ]} \|^{2}_{ \mathcal{H}_{\mathfrak{I}}}, 
	$$
	which concludes the proof.
\end{proof}

A version of Theorem \ref{finalgeomkn} still holds if we remove the assumption that the measures $\eta_{j}$ are PI. A direct application of this description is the following interesting result. 

\begin{corollary}\label{zerofunction} Let $\mathfrak{I} : \mathds{X}_{n} \times \mathds{X}_{n} \to \mathbb{R} $ be a complete $n$-symmetric PDI$_{k}$ kernel that is zero at the extended diagonal $\Delta_{k-1}^{n}$. If 
	$$
	\int_{\mathds{X}_{n}}\int_{\mathds{X}_{n}} (-1)^{k}\mathfrak{I}(x_{\vec{1} },x_{\vec{2} } ) d\mu(x_{\vec{1}})d\mu(x_{\vec{2} })=0
	$$
	for every $\mu \in \mathcal{M}_{k}(\mathds{X}_{n})$, then $\mathfrak{I} $ is the zero function.
	
\end{corollary}

\begin{proof} By the hypothesis and Lemma \ref{pdi2topdn}, we have that for any PI Reversible measure $\eta_{j} \in \mathcal{M}_{k }( \{1,2\}^{j} )$, any fixed $x_{\vec{3}} \in \mathds{X}_{n} $, and $F \subset \{1,\ldots n \}$ with $|F|=j$, where $k \leq j \leq n $, 
	\begin{align*}
		& \| K^{\mathfrak{I}}_{\eta_{j}[ \vec{1}_{F} + \vec{3}_{F^{c}}, \vec{2}_{F} + \vec{3}_{F^{c}} ]} \|^{2}_{\mathcal{H}_{\mathfrak{I}}} \\
		&= \int_{\mathds{X}_{n}}\int_{\mathds{X}_{n}} (-1)^{k}\mathfrak{I}(u,v ) d[\eta_{j}[ \vec{1}_{F} + \vec{3}_{F^{c}}, \vec{2}_{F} + \vec{3}_{F^{c}} ] ](u)d[\eta_{j}[ \vec{1}_{F} + \vec{3}_{F^{c}}, \vec{2}_{F} + \vec{3}_{F^{c}} ]](v)=0.
	\end{align*}
	The expression for $\mathfrak{I}$ in Theorem \ref{finalgeomkn} implies that $    \mathfrak{I}(x_{\vec{1}}, x_{\vec{2}})=0$ for every $x_{\vec{1}}, x_{\vec{2}} \in \mathds{X}_{n}$.
\end{proof}

When the codimension is zero, that is, when $n=k$, as shown at the beginning of the proof of Theorem \ref{finalgeomkn}, we have that  for $C\neq 0$
$$
\mathfrak{I}(x_{\vec{1}},x_{\vec{2}})=\frac{1}{2^{n}C^{2}}\| K^{\mathfrak{I}}_{C\mu_{n}^{n}[ \vec{1}, \vec{2} ]}\|^{2}_{\mathcal{H}_{\mathfrak{I}}},
$$
hence by choosing $C=2^{-n/2}$, that is, $\eta_{n} = 2^{-n/2}\mu_{n}^{n} $, we get that  $b_{n}^{n,n} = 1$ for every $n \in \mathbb{N}$. 

From this initial term we can present an explicit case of Theorem   \ref{finalgeomkn}, however, on it, not all coefficients are nonnegative. 

\begin{example}\label{examplecoefi} Let $\mathfrak{I} : \mathds{X}_{n} \times \mathds{X}_{n} \to \mathbb{R} $ be a complete $n$-symmetric PDI$_{k}$ kernel that is zero at the extended diagonal $\Delta_{k-1}^{n}$. The PI Reversible measures $\eta_{j}= 2^{-j/2}\mu_{j}^{j} \in \mathcal{M}_{j }( \{1,2\}^{j} ) \subset \mathcal{M}_{k }( \{1,2\}^{j} ) $,  where $k \leq j \leq n $ satisfies that
	$$
	\mathfrak{I}(x_{\vec{1}}, x_{\vec{2}})=     \sum_{|F|=k}^{n} (-1)^{|F| + k} \| K^{\mathfrak{I}}_{\eta_{|F|}[ \vec{1}_{F} + \vec{3}_{F^{c}}, \vec{2}_{F} + \vec{3}_{F^{c}} ]} \|^{2}_{\mathcal{H}_{\mathfrak{I}}}, 
	$$
	for any $x_{\vec{0}} \in \mathds{X}_{n} $ that is used to define $K^{\mathfrak{I}}$.      
\end{example}

\begin{proof}We prove by induction on $n$, where the first case was presented before the statement of this Corollary.\\
	Now, suppose that  $b_{j}^{m, k}= (-1)^{j+k}$ for every $k\leq m \leq n-1$ and $k\leq j \leq n-1$, and we shall prove it the representation for $b_{j}^{n,k}$ for $k \leq j \leq n$.\\
	First, note that by Equation \eqref{munn}, we may compute the terms $D^{n,k}_{\ell}$ for the measure $\eta_{n}$    
	\begin{align*}
		D^{n,k}_{\ell}&= \sum_{p=0}^{n-\ell} \sum_{q=0}^{\ell}\binom{\ell}{q}\binom{n-\ell}{p} w_{p+q}w_{p + \ell-q}\\
		&=2^{-n}\sum_{p=0}^{n-\ell} \sum_{q=0}^{\ell}\binom{\ell}{q}\binom{n-\ell}{p} (-1)^{n-p-q}(-1)^{n-p-\ell+q}=(-1)^{\ell}. 
	\end{align*}
	Thus $b_{n}^{n,k}= (-1)^{k}/D_{n}^{n,k}= (-1)^{n+k}$ and since $\sum_{ r=0}^{n-j-1}\binom{n-j}{r} (-1)^{r}=(-1)^{n-j-1}$, we get that
	\begin{align*}
		b_{j}^{n,k}&= \frac{-1}{D_{n}^{n,k}}\left [\sum_{ r=0}^{n-j-1}\binom{n-j}{r}D^{n,k}_{j+r}b^{j+r,k}_{j}\right ]\\
		& =(-1)^{n+1}\left [\sum_{ r=0}^{n-j-1}\binom{n-j}{r} (-1)^{j+r}(-1)^{j+k} \right ]\\
		&=(-1)^{n+1}(-1)^{k+n-j-1}=(-1)^{k+j}.
	\end{align*}
\end{proof}

In the following subsection, we provide explicit examples for when the codimension $n-k$ is $1$ or $2$ in a way that the coefficients are nonnegative. For this, we shall use the moment formulas:

\begin{align*}
	\sum_{q=0}^{n}\binom{n}{q} &= 2^n, \\
	\sum_{q=0}^{n}\binom{n}{q}q &= n 2^{n-1}, \\
	\sum_{q=0}^{n}\binom{n}{q}q^{2} &= n(n+1) 2^{n-2}, \\
	\sum_{q=0}^{n}\binom{n}{q}q^{3} &= n^2(n+3) 2^{n-3}, \\
	\sum_{q=0}^{n}\binom{n}{q}q^{4} &= n(n^3 + 6n^2 + 3n - 2)2^{n-4},
\end{align*}
which are obtained through the binomial equality $(1+x)^{n}= \sum_{j=0}^{n}\binom{n}{j}x^{j}$.

\subsection{Codimension $1$}\label{cod1}
We describe the case where $k=n-1$. The coefficients of the measure  $\eta_{n}$ that  simplifies the most our expressions are 
$$
w_{q}^{n,n-1}:=  (-1)^{q}2^{-n/2}n^{-1/2}(-n+2q)
$$ 
that is, $A_{0}=-n^{1/2}2^{-n/2}$ and $A_{1}=n^{-1/2}2^{(-n+2)/2}$, as explained in Equation \eqref{recursive}. We continue to use $\eta_{n-1}=2^{-(n-1)/2}\mu_{n-1}^{n-1}$ presented before the statement of Example \ref{examplecoefi}, so that $b_{n-1}^{n-1,n-1}=1$.

Note that  $w_{q}^{n,n-1}=(-1)^{n+1}w_{n-q}$ for every possible $q$. By the proof of Theorem \ref{finalgeomkn}, we have that

\begin{align*}
	D^{n,n-1}_{n}&= \sum_{q=0}^{n}\binom{n}{q} w_{q}^{n,n-1}w_{ n-q}^{n,n-1}\\
	&= (-1)^{n+1}n^{-1}2^{-n}\sum_{q=0}^{n}\binom{n}{q}(-n+2q)^{2}\\
	&= (-1)^{n+1}n^{-1}2^{-n}\left [4\sum_{q=0}^{n}\binom{n}{q}q^{2} - 4n\sum_{q=0}^{n}\binom{n}{q}q +n^{2}\sum_{q=0}^{n}\binom{n}{q}\right ]\\
	&=(-1)^{n+1},
\end{align*}

\begin{align*}   
	&D^{n,n-1}_{n-1}:= \sum_{p=0}^{1} \sum_{q=0}^{n-1}\binom{n-1}{q}\binom{1}{p} w_{p+q}^{n,n-1}w_{p + n-1-q}^{n,n-1}\\
	&= \sum_{q=0}^{n-1}\binom{n-1}{q} w_{q}^{n,n-1}w_{ n-1-q}^{n,n-1} + \sum_{q=0}^{n-1}\binom{n-1}{q} w_{1+q}^{n,n-1}w_{ n-q}^{n,n-1}\\
	&= (-1)^{n-1}n^{-1}2^{-n}\sum_{q=0}^{n-1}\binom{n-1}{q}[ (-n+2q)(n-2q-2) +(-n+2q+2)(n-2q)]\\
	&= (-1)^{n-1}n^{-1}2^{-n}\left[ -8\sum_{q=0}^{n-1}\binom{n-1}{q}q^{2} + 8(n-1)\sum_{q=0}^{n-1}\binom{n-1}{q}q + ( 4n -2n^{2})\sum_{q=0}^{n-1}\binom{n-1}{q} \right ]\\
	&= (-1)^{n}\frac{n-2}{n}. 
\end{align*}

Hence    

\begin{align*}
	E_{n-1}^{n,n-1} &=D^{n,n-1}_{n-1}b^{n-1,n-1}_{n-1}=(-1)^{n}\frac{n-2}{n}.
\end{align*}         

Thus
$$
b^{n,n-1}_{n}=\frac{(-1)^{n-1}}{D_{n}^{n,n-1}} =1, 
$$
and
$$
b^{n,n-1}_{n-1}=-\frac{E_{n-1}^{n,n-1}}{D_{n}^{n,n-1}}= \frac{n-2}{n},
$$         
and in particular, $\mathfrak{I}$ is nonnegative as it is the sum of nonnegative terms.

\subsection{Codimension $2$}\label{cod2}
We describe the case where $k=n-2$. For this case we continue to use the PI reversible measures
$$
\eta_{n-1}^{n,n-2}= \sum_{q=0}^{n-1}\sum_{|F|=j}(-1)^{q}(n-1)^{-1/2}2^{-(n-1)/2}(-(n-1)+2q)\delta(\vec{1}_{F} + \vec{2}_{F^{c}}) \in \mathcal{M}_{n-2}(\{1,2\}^{n-1}),
$$

$$
\eta_{n-2}^{n,n-2}= \sum_{q=0}^{n-2}\sum_{|F|=j}(-1)^{q}2^{-n/2}\delta(\vec{1}_{F} + \vec{2}_{F^{c}}) \in \mathcal{M}_{n-2}(\{1,2\}^{n-2})
$$

Following the steps in the proof of Theorem \ref{finalgeomkn}, we have that

$$
b_{n}^{n,n-2}=\frac{(-1)^{n-2}}{D_{n}^{n,n-2}}, 
$$

$$
b_{n-1}^{n,n-2} = -\frac{E_{n-1}^{n,n-2}}{D_{n}^{n,n-2}}= -\frac{D_{n-1}^{n,n-2}b_{n-1}^{n-1,n-2}}{D_{n}^{n,n-2}},
$$

$$
b_{n-2}^{n,n-2} = -\frac{E_{n-2}^{n,n-2}}{D_{n}^{n,n-2}}= -\frac{D_{n-2}^{n,n-2}b_{n-2}^{n-2,n-2} + 2D_{n-1}^{n,n-2}b_{n-2}^{n-1,n-2}}{D_{n}^{n,n-2}}, 
$$ 
where, by the results on codimension $1$ and $0$  
$$
b_{n-1}^{n-1,n-2}=1, \quad b_{n-2}^{n-2,n-2}= 1, \quad b_{n-2}^{n-1,n-2}=\frac{n-3}{n-1}.
$$

We compute the $b_{j}^{n,n-2}$ coefficients using two distinct measures $\eta_{n}^{n,n-2} $

\textbf{Case 1} Using the measure $2^{-n/2}\mu_{n}^{n} \in \mathcal{M}_{n}(\{1,2\}^{n}) \subset  \mathcal{M}_{n-2}(\{1,2\}^{n}) $

The values of $D_{n}^{n,n-2}$, $D_{n-1}^{n,n-2}$ and $D_{n-2}^{n,n-2}$ have already been computed for this measure in Example \ref{examplecoefi} and are respectively $(-1)^{n}$, $(-1)^{n-1}$ and $(-1)^{n-2}$. Plugging them into the expressions we aim to analyze we get that
$$
b_{n}^{n,n-2}=\frac{(-1)^{n-2}}{D_{n}^{n,n-2}}= 1,
$$
$$
b_{n-1}^{n,n-2} = -\frac{D_{n-1}^{n,n-2}b_{n-1}^{n-1,n-2}}{D_{n}^{n,n-2}}=1,
$$
$$
b_{n-2}^{n,n-2}=\frac{-D_{n-2}^{n,n-2}b_{n-2}^{n-2,n-2} - 2D_{n-1}^{n,n-2}b_{n-2}^{n-1,n-2}}{D_{n}^{n,n-2}}= -1+2\frac{n-3}{n-1}= \frac{n-5}{n-1}.
$$
On this case, all $3$ coefficients are nonnegative when $n\geq 5$. 

\textbf{Case 2} Using the measure $\sum_{q=0}^{n}\sum_{|F|=j}(-1)^{q}2^{-n/2}(-n+2q)^{2}\delta(\vec{1}_{F} + \vec{2}_{F^{c}}) \in \mathcal{M}_{n-2}(\{1,2\}^{n})$.

\begin{align*}
	D^{n,n-2}_{n}&= \sum_{q=0}^{n}\binom{n}{q} w_{q}^{n,n-2}w_{ n-q}^{n,n-2}\\
	&= (-1)^{n}2^{-n}\sum_{q=0}^{n}\binom{n}{q}(-n+2q)^{2}(n-2q)^{2}\\
	&= (-1)^{n}2^{-n}\sum_{q=0}^{n}\binom{n}{q} \left [16q^{4} - 32nq^{3} +24n^{2}q^{2}-8n^{3}q + n^{4}\right ]\\
	&=(-1)^{n}(3n^{2} -2n),
\end{align*}

\begin{align*}
	(-1)^{n-1}2^{n}D^{n,n-1}_{n}&= (-1)^{n-1}2^{n}\sum_{p=0}^{1}\sum_{q=0}^{n-1}\binom{n-1}{q}\binom{1}{p} w_{p+q}^{n,n-2}w_{ p+n-q-1}^{n,n-2}\\
	&= \sum_{q=0}^{n-1}\binom{n-1}{q}[(-n+2q)^{2} (n-2q-2)^{2} + (-n+2q+2)^{2}(n-2q)^{2}   ]\\
	&=\sum_{q=0}^{n-1}\binom{n-1}{q} \left [ 32q^{4} +64q^{3}(-n+1) +16q^{2}(3n^{2}-6n+2) \right .\\  
	& \quad \quad \quad \quad \quad \quad \left . +16q(-n^{3} +3n^{2} -2) +2(n^{4}-4n^{3} +4n^{2})     \right ]\\
	&=2^{n}(3n^{2} -10n +8).
\end{align*}

\begin{align*}
	&(-1)^{n-2}2^{n}D^{n,n-2}_{n-2}= \sum_{p=0}^{2}\sum_{q=0}^{n-2}\binom{n-2}{q}\binom{2}{p} w_{p+q}^{n,n-2}w_{ p+n-q-2}^{n,n-2}\\
	&=\sum_{q=0}^{n-2}\binom{n-2}{q}[ (-n+2q)^{2} (n-2q-4)^{2} + 2(-n+2q+2)^{4}+  (-n+2q+4)^{2} (n-2q)^{2}  ].
\end{align*}	
Since 

\begin{align*}
	(-n+2q)^{2} (n-2q-4)^{2}& + 2(-n+2q+2)^{4} + (-n+2q+4)^{2} (n-2q)^{2}\\
	=& (4n^{4} -32n^{3} +80n^{2} -64n +32)+  (-32n^{3} +192n^{2}-320n +128 )q \\
	&+ (96n^{2}-384n +320)q^{2} + ( -128n +256)q^{3} + 64q^{4},
\end{align*}
by using the binomial sum identities we get that
$$
(-1)^{n-1}2^{n}D^{n,n-2}_{n-2}=2^{n}(3n^2 - 18n + 32).
$$
Plugging those values into the expressions we aim to analyze we get that
$$
b_{n}^{n,n-2}=\frac{(-1)^{n-2}}{D_{n}^{n,n-2}}= \frac{1}{3n^{2} -2n} ,
$$
$$
b_{n-1}^{n,n-2} = -\frac{D_{n-1}^{n,n-2}b_{n-1}^{n-1,n-2}}{D_{n}^{n,n-2}}=\frac{3n^{2} -10n +8}{3n^{2} -2n},
$$
\begin{align*}
	b_{n-2}^{n,n-2}&=\frac{-D_{n-2}^{n,n-2}b_{n-2}^{n-2,n-2} - 2D_{n-1}^{n,n-2}b_{n-2}^{n-1,n-2}}{D_{n}^{n,n-2}}\\
	&=\frac{-(3n^{2} -18n +32) +2(\frac{n-3}{n-1})(3n^{2} -10n +8)}{3n^{2}-2n}=\frac{3n^{3} -17n^{2} +26n -16}{(n-1)(3n^{2} -2n)},
\end{align*}
and on this case, all $3$ coefficients are nonnegative when $n\geq 4$. 		
	\bibliographystyle{plainnat}
	\bibliography{sample}
	
	\end{document}